\documentclass[11pt,a4paper]{amsart}
\usepackage{graphicx}
\usepackage{fullpage}

\usepackage[utf8]{inputenc} \usepackage[english]{babel} 
\usepackage{amsmath}
\usepackage{amsfonts}
\usepackage{amssymb}
\usepackage{amsthm} \usepackage{ulem} \usepackage{mathtools}
\usepackage{color}

\usepackage{enumerate}
\usepackage{multirow}

\usepackage{url}

\usepackage{hyperref} 				 	  			

\renewcommand{\geq}{\geqslant}
\renewcommand{\leq}{\leqslant}

\definecolor{ColorImplicit}{rgb}{0.7,0.1,0.1}
\definecolor{red}{rgb}{0.862745,0.137255,0}
\definecolor{myOrange}{RGB}{255, 169, 87 }
\definecolor{myGreen}{RGB}{180, 255, 162  }
\definecolor{myGrey}{RGB}{187, 187, 187  }
\definecolor{myDarkGrey}{RGB}{100, 100, 100  }

\usepackage[hang,small]{caption}

\usepackage[]{subcaption}

\graphicspath{{./notes/}}

\newcommand{\R}{\mathbb R}
\newcommand{\N}{\mathbb N}

\newcommand{\cB}{\mathcal{B}}

\newcommand{\cH}{\mathcal{H}}

\newcommand{\cU}{\mathcal{U}}

\newcommand{\cP}{\mathcal{P}}

\newcommand{\cV}{\mathcal{V}}

\newcommand{\Hessian}{{D^2}}

\newcommand{\e}{\varepsilon}

\def\xC{{\rm C}}
\def\xLip{ {\rm Lip} }

\def\xL{{\rm L}}
\def\xdiv{{\rm div}}
\def\xdiam{{\rm diam}}

\newcommand{\V}{\|V\|}

\newcommand{\G}{G_{d,n}}

\newcommand{\proj}{\Pi}

\renewcommand{\emph}[1]{{\it{{#1}}}}

\DeclareMathOperator*{\supp}{spt}

\theoremstyle{definition}
\newtheorem{thm}{Theorem}[section]
\newtheorem{cor}[thm]{Corollary}
\newtheorem{lem}[thm]{Lemma}
\newtheorem{prop}[thm]{Proposition}
\newtheorem{defin}[thm]{Definition}
\newtheorem{rem}[thm]{Remark}
\newtheorem{exa}[thm]{Example}

\begin{document}

\baselineskip=17pt

\title{Mean curvature motion of point cloud varifolds}

\author{Blanche Buet\\
Universit{\'e} Paris-Saclay, CNRS, Laboratoire de math{\'e}matiques d'Orsay, \\
91405, Orsay, France\\
blanche.buet@math.u-psud.fr\\
Martin Rumpf\\
Institute for Numerical Simulation, University of Bonn, \\
Endenicher Allee 60, 53115 Bonn, Germany\\
martin.rumpf@uni-bonn.de}

\date{}

\maketitle

\begin{abstract}
This paper investigates a discretization scheme for mean curvature motion on point cloud varifolds with particular emphasis on singular evolutions.
To define the varifold a local covariance analysis is applied to compute an approximate tangent plane for the points in the cloud.
The core ingredient of the mean curvature motion model is the regularization of the first variation of the varifold via convolution with kernels with small stencil.
Consistency with the evolution velocity for a smooth surface is proven if a sufficiently small stencil and a regular sampling are taking into account. 
Furthermore, an implicit and a semiimplicit time discretization are derived. 
The implicit scheme comes with discrete barrier properties known for the smooth, continuous evolution, 
whereas the semiimplicit still ensures in all our numerical experiments very good approximation properties while being easy to implement. 
It is shown that the proposed method is robust with respect to noise and recovers the evolution of smooth curves as well as the formation of singularities such as triple points in 2D or minimal cones in 3D. 


\end{abstract}

	\maketitle

\renewcommand{\phi}{\varphi}
\renewcommand{\epsilon}{\varepsilon}
\newcommand{\one}{\mathds{1}}

\section{Introduction} \label{sec:intro}

In this paper we study the discretization of mean curvature motion for point cloud varifolds. 
Point clouds are the raw output data of 3D laser scanning devices and instead of applying a meshing algorithm which approximates the point cloud with a triangular surface
we aim for geometry processing methods directly on the raw data. Particular emphasis is on a proper treatment of possibly noisy point distributions and 
geometric singularities such as triple points or crease lines.

The direct processing of point clouds has intensively been studied in the literature.
Using the normal cycle from geometric measure theory Cohen-Steiner and Morvan \cite{CoMo06} were able to robustly 
compute the shape operator of a triangular mesh and they gave explicit error bounds.
Yang and Qian \cite{YaQi07} used a moving least square approach to compute curvature quantities on point cloud surfaces.
In \cite{ChCoLiTh09} Chazal et al. investigated the curvature estimation problem in particular for point cloud data.
They showed that different curvature measures can stably be computed for compact sets with positive $\mu$-reach using distance functions to the set and evaluating curvatures on them.
M{\'e}rigot et al. \cite{MeOvGu11} used a covariance analysis  based on the local Voronoi tesselation to compute a robust discrete curvature for a point cloud surface.
Yang et al. \cite{YaPo06} used barycenters of spherical neighbourhoods on multiple scales to derive formulas for the principle curvatures of a surface. 
This approach is based on the observation that the vector offset between barycenters for different radii of the balls or spheres depends on the surface curvature.
Chazal et al. \cite{ChCoLiTh09} followed the same approach and showed how this can be used to obtain approximations of generalized notions of curvature proposed by Federer.

In the processing of discrete surfaces the evolution by mean curvature motion is a fundamental tool and one of the basic fairing algorithms \cite{DeMeScBa99}.
For the numerical discretization of mean curvature motion for hypersurfaces there are three widespread approaches 
corresponding to the representation of the hypersurfaces
as a  triangular surface \cite{Dz91, DeDzEl05, BaGaNu08}, a level set \cite{MeBeOs94,EvSp91,Sm03}, 
or a diffusive phase field interface \cite{Feng2003, EvSoSo06, BrPe12}. 
For the representation of surfaces by point clouds a mean curvature motion scheme 
has been derived in \cite{ClRuTe04b} based on the reconstruction of a local triangulation. 
In \cite{Sharp:2020:LNT} a special type of surface covering is defined for point clouds, which enable the stable evaluation of surface Laplacians and associated evolution problems.

The varifold perspective has been used in the context of curve or surface 
matching in \cite{KaChCh17} based on earlier work in \cite{ChTr13}.
Here, a set of simplices and simplex normals is encoded as a measure on $\R^n\times S^{n-1}$ and equipped with the structure of a reproducing kernel Hilbert space with suitable kernels for the position in $\R^n$ and the orientation in $S^{n-1}$. This approach is then used for the registration of  curves and surfaces without point to point correspondence.
A recent overview of these tools in the context of diffeomorphic registration can be found in \cite{ChChGl20}.

In the context of this paper, we adopt a varifold perspective on point clouds and take advantage of the framework developped in \cite{BuetLeonardiMasnou} to estimate mean curvature with theoretical convergence guarantees. More precisely, we define a point cloud varifold by applying a local covariance analysis on the input set of points.
This allows us to assign an approximate tangent space as well as a weight to each point. This leads to a natural varifold structure where point clouds are encoded through a weightes sum of Dirac masses plus orientation.
We can subsequently consider the first variation of such a varifold and apply a suitable regularization via the convolution with a kernel with small stencil.
The resulting regularized first variation is considered as the motion field for the mean curvature flow and an implicit and a semiimplicit time discretization are derived.

A particularity when processing directly point clouds is that it allows topological changes, concentration and merging of points very naturally. As consequence, while the formation of a triple point is a singularity at the continuous level, it is not from the point of view of point cloud evolution. We take advantage of this feature to recover some well--known minimizing sets like Steiner trees connecting the vertices of a square in 2D and minimal cones over the edges of a tetrahedron, which is known to be one of the basic minimal cones in 3D together with the plane and the triple junction of half-planes \cite{Ta76}. Furthermore, we compute minimal area sets spanned by the edges of a cube .
However, while we can compare the limit set for time tending to infinity with sets that are known to be minimal (or at least with competitors with respect to surface area measures), the theoretical context for the evolution is not clear. As we are dealing with a mean curvature discretization derived thanks to varifolds tools, it is very natural to think of Brakke flow \cite{brakke} and the more recent approach by Kim and Tonegawa \cite{KiTo17}. Nevertheless, while the computation of the mean curvature is performed in the varifold setting, the motion is not entirely defined. Indeed, we only flow the positions according to mean curvature, and tangents and masses are implicitly defined given the new positions. Concretely, we compute the masses from the positions assumping a multiplicity $1$ everywhere. Assuming a unit multiplicity is reasonnable when starting from a compact hypersurface, indeed, it has been proven in \cite{EvSp95} that for almost all neighbouring level sets, there exists a unique Brakke flow, with unit multiplicity for almost all times. 

We emphasize that results on existence and uniqueness beyond the creation of singularities are mostly obtained under the asumption that the set evolving by mean curvature can be represented as the boundary of some open set (which is in general not the case for Brakke flow). 
We mention two major approaches: in \cite{ChGiGo91,EvSp91} the issue is tackled through so-called viscosity solutions while \cite{Il93} is based on the Allen-Cahn phase field model. In the case of planar networks, the flow has been studied in \cite{MaNoTo04} starting from a regular network, that is restricted to triple point singularities, and more recently the case of more general junctions could be handled in \cite{IlNeSc19}. Note that in our simulations, we observe that  four junctions in the initial data instantaneously split into two triple junctions, consistently with what is expected for planar networks.
On the numerical side, all parametric (mesh based) approaches work with a fixed topology an enable under this constraint 
very good approximation results, e.g. in \cite{BrSE92,PiPo93,ScWa19} with the drawback that there is an a priori choice or a required combinatorial optimization among all possible layouts of the smooth patches. 

The paper is organized as follows.
In Section \ref{sec:varifold} we recall classical facts concerning varifolds, focusing on the so-called first variation and generalized mean curvature.
Using a regularization via convolution, we define in Section \ref{sec:reg} the generalized mean curvature model for point clouds varifolds and establish its consistency in Proposition~\ref{prop:consistency}. Its stability is then investigated in Section~\ref{sec:stability}, resulting in Theorem~\ref{thm:stability} and convergence stated in Corollary~\ref{cor:CV}. Section~\ref{sec:timecont} describes how we define a point cloud varifold from a set of points i.e. how we compute tangent planes and weights for each point.
Furthermore, we define time continuous curvature motion of point clouds, study planar and spherical comparison principles, and 
derive implicit and semiimplicit time discretization schemes.
Finally in Section \ref{sec:results} we present numerical results and discuss properties of the derived scheme.

\section{Varifolds and generalized mean curvature} \label{sec:varifold}

In this section we will review the basic notions of $d$-varifolds in $\R^n$ and first variation of such varifolds.  
Of particular interest for this paper will be the case of generalized curves in $\R^2$ ($d=1$, $n=2$) and generalized surfaces in $\R^3$ ($d=2$, $n=3$). 
The $d$--dimensional Hausdorff measure in $\R^n$ is denoted by $\cH^d$, and 
the space of continuous compactly supported function between two metric spaces by $\xC_c (X, Y)$ and $\xC_c(X)$ if $Y = \R$. 
The $d$--Grassmannian
\[
\G = \{ d\text{--vector subspace of } \R^n  \} \: .
\]
is embedded into $n\times n$ matrices via the mapping that associates with the $d$--subspace $P \in \G$ the orthogonal projector $\Pi_P$ onto $P$. 
The operator norm on matrices consequently induces a distance on $\G$. 
With this notation at hand let us give the definition of a $d$-varifold.
\begin{defin}[Varifold]
A $d$--varifold in an open set $\Omega \subset \R^n$ is a Radon measure in $\Omega \times \G$.
\end{defin}
For detailed discussions on varifolds and underlying geometric measure theory tools we refer to  \cite{simon}, \cite{ambrosio}. We thereafter consider $s$--varifolds in the whole space $\Omega = \R^n$.
Such measures can be understood as a coupling of spatial (in $\R^n$) and directional (in $\G$) information.
Integrating on all possible directions, that is on the whole Grassmannian, allows to select the spatial information encoded in a $d$--varifold $V$: the resulting Radon measure denoted by $\V$ is called the \emph{mass} of $V$ and is defined in $\R^n$ as
\[
\V (A) =
V(A \times \G) \: .
\]
for all Borel sets $A \subset \R^n$.

In this paper, we will focus on two types of varifolds: varifolds associated with a smooth $d$--submanifold of $\R^n$ (referred to as \emph{smooth varifolds}, see Definition~\ref{dfn:smoothVarifold}) and varifolds associated with a finite set of points in $\R^n$, positive weights and tangent $d$--planes in $\G$ (referred to as \emph{point cloud varifolds}, see Definition~\ref{dfn:pointCloudVarifold}).

\begin{defin}[Smooth varifold] \label{dfn:smoothVarifold}
The $d$--varifold $V$ associated with a $d$--submanifold $M \subset \R^n$
is defined by
\begin{equation} \label{eq:smoothVarifold}
V (B) = \cH^d \left( \{ x \in M : (x, T_x M) \in B \} \right) \: ,
\end{equation}
for every Borel set $B \subset \R^n \times \G$. Here, $T_x M$ denotes the tangent plane at $x$. 
In this case, $\V = \cH^d_{| M}$.
We will use the notation $V = \cH^d_{| M} \otimes \delta_{T_x M}$ for the varifold defined in \eqref{eq:smoothVarifold}.
\end{defin}
Smooth varifolds are a particular case of \emph{rectifiable varifolds} (see \cite{simon}). 
\begin{rem} \label{rmk:viaTestFunctions}
As a $d$-varifold is a Radon measure, it can be equivalently defined by its action on continuous compactly supported functions: $V$ is the smooth varifold associated with $M$ according to Definition~\ref{dfn:smoothVarifold} if and only if 
for every $\phi \in \xC_c (\R^n \times \G)$,
\begin{align*}
\int_{(x,S) \in \R^n \times \G} \phi(x,S) \, dV(x,S) & = \int_{x \in \R^n \cap M} \int_{S \in \G} \phi (x,S) d \delta_{T_x M} (S) \, d \cH^d (x) \\
& = \int_M \phi(x,T_x M) \, d \cH^d (x) \: .
\end{align*}
\end{rem}

\begin{defin}[Point cloud varifold] \label{dfn:pointCloudVarifold}
Given a finite set of points $\{ x_i \}_{i=1}^N \subset \R^n$, masses (weights) $\{ m_i \}_{i=1}^N \subset \R_+^\ast$ and  space of directions $\{ P_i \}_{i =1}^N \subset \G$, we associate the $d$--varifold
\[
V = \sum_{i=1}^N m_i \delta_{(x_i,P_i)} \quad \text{and in this case} \quad \V = \sum_{i=1}^N m_i \delta_{x_i} \: .
\]
\end{defin}
Note that $P_i$ can be any  set of directions in $\G$, however if the points $\{ x_i \}_i$ sample some surface or submanifold $M$, then $\{ P_i \}_i$ can be thought as tangent planes $T_{x_i} M$.

As in the case of smooth varifolds, we can define a point cloud varifold through its action on compactly supported continuous functions. Indeed, we have for $\phi \in \xC_c (\R^n \times \G)$
\[
\int \phi \, dV = \sum_{i=1}^N  m_i \phi (x_i , P_i) \: .
\]

The set of $d$--varifolds is endowed with a weak notion of mean curvature which we eventually define in Definition~\ref{dfn:generalizedCurvature}. 
At first, let us introduce the \emph{first variation},
which is well--defined for any $d$--varifold. For this purpose we need the following differential operators:
let $P \in \G$, $\Pi_P$ be the orthogonal projection onto $P$ and $(\tau_1, \ldots, \tau_d)$ be an orthonormal basis of $P$, let $X = (X_1, \ldots , X_n) \in \xC^1 (\R^n , \R^n)$ be a vector field of class $\xC^1$, $\phi \in \xC^1 (\R^n)$ and $(e_1, \ldots, e_n)$ be the canonical basis of $\R^n$, then
\[
\nabla^P \phi = \Pi_P (\nabla \phi) \quad \text{and} \quad \xdiv_P X = \sum_{i=1}^n \Pi_P (\nabla X_i) \cdot e_i = \sum_{i=1}^d DX \tau_i  \cdot \tau_i \: .
\]
Now, with these differential operators at hand we can define the first variation:

\begin{defin}[First variation of a varifold, \cite{Allard72}] \label{dfn:firstVariation}
The first variation of a $d$--varifold $V$ in $\R^n$ is the \emph{distribution of order $1$}
\[ \begin{array}{lcccl}
\delta V & : & \xC_c^1 (\R^n , \R^n) & \rightarrow & \R \\
         &   &      X                & \mapsto  & \displaystyle \int_{\R^n \times \G} \xdiv_S X(x) \, dV(x,S)
\end{array} \]
\end{defin}
\begin{rem}
We could equivalently define the first variation based on scalar test functions. Indeed, with  
a slight misuse of notation define for $\phi \in \xC_c^1 (\R^n)$,
\[
\delta V(\phi) \coloneqq ( \delta V (\phi e_1), \ldots, \delta V (\phi e_n)) = \sum_{i=1}^n \delta V(\phi e_i) e_i 
\quad \text{so that} \quad \delta V (X) =   \sum_{i=1}^n \delta V (X \cdot e_i) \cdot e_i \: .\]
Moreover, $\displaystyle \xdiv_S (\phi e_i) = \nabla^S \phi \cdot e_i$ and finally $ \displaystyle
\delta V (\phi) = \int_{\R^n \times \G} \nabla^S \phi(x) \, dV(x, S)
$.
\end{rem}
Let $M \subset \R^n$ be a closed $\xC^2$  $d$--submanifold, if $V = \cH^d_{| M} \otimes \delta_{T_x M}$ is the smooth varifold naturally associated with $M$, then by definition of $V$ (see Remark~\ref{rmk:viaTestFunctions}),
\[
\int_{\R^n \times \G} \xdiv_S X(x) \, dV(x,S) = \int_M \xdiv_{T_x M} X(x) \, d \cH^d(x) \: ,
\] 
and thus, thanks to the divergence theorem we obtain for every $X \in \xC_c^1 (\R^n , \R^n)$,
\[
\delta V (X) = \int_M \xdiv_{T_x M} X(x) \, d \cH^d(x) = - \int_M H(x) \cdot X(x) \, d \cH^d(x) \: ,
\]
where $H$ is the mean curvature vector of $M$. In this case $\delta V$ is more regular than a distribution of order $1$, it is a distribution of order $0$ and can be identified thanks to Riesz theorem with the vector valued Radon measure
\begin{equation} \label{eq:firstVariationSmooth}
-H(x) \, \cH^d_{| M} (x) = -H(x) \, \V(x) \: .
\end{equation}
Actually, as soon as $V$ is a $d$--varifold (not necessarily associated with a smooth submanifold) whose first variation is a distribution of order $0$ ($V$ is then said to have \emph{locally bounded first variation}), then there is a weak counter-part to the divergence theorem. Indeed, for such a varifold, we can identify the distribution $\delta V$ with the associated Radon measure provided by Riesz theorem. Then, thanks to the Radon-Nikodym decomposition of $\delta V$ with respect to the mass $\V$, there exist a vector field denoted $\frac{\delta V}{\V} \in \xL^1_{loc} (\V)$ and a Radon measure $(\delta V)_s$ singular with respect to $\V$ (which might vanish) such that the decomposition 
\begin{equation} \label{eq:radonNikodym}
\delta V = \frac{\delta V}{\V}(x) \, \V + (\delta V)_s
\end{equation}
holds. Comparing \eqref{eq:firstVariationSmooth} and \eqref{eq:radonNikodym} 
the \emph{generalized mean curvature vector} naturally arises as the Radon-Nikodym derivative 
$ \displaystyle
H = - \frac{\delta V}{\V} \: .
$
Let us resume the previous observations in the following definition.

\begin{defin}[Generalized mean curvature, \cite{Allard72}] \label{dfn:generalizedCurvature}
Let $V$ be a $d$--varifold in $\R^n$. Assume that $V$ has \emph{locally bounded first variation},  i.e.  $\forall K \subset \R^n$ compact set, $\exists \: c_K > 0$ such that for every $X \in \xC_c^1 (\R^n, \R^n)$ supported in $K$,
\begin{equation} \label{eq:locallyBoundedFirstVariation}
| \delta V (X) | \leq c_K \sup_{K} |X| \: .
\end{equation}
Then $\delta V$ can be identified with a Radon measure and the \emph{generalized mean curvature vector} $H$ is defined as the Radon-Nikodym derivative of $\delta V$ with respect to $\V$, moreover,
for $\V$--a.e. $x$ and for $B(x,r)$ denoting the open ball of radius $r$ centred at $x$ we get
\[
H(x) = - \lim_{r \to 0_+} \frac{\delta V(B(x,r))}{\V(B(x,r))} \: .
\]
\end{defin}
See Section~2.4 in \cite{ambrosio} for more details on differentiation of Radon measures.  
Notice that both classical and the generalized mean curvature coincide in the case of a smooth varifold associated with a closed $\xC^2$ manifold as shown above \eqref{eq:firstVariationSmooth}.
Let us now consider an example involving a triple point singularity.
\begin{exa}[Junction of half-lines] \label{ex:junction}
Let $u_1, u_2, u_3$ be unit vectors of $\R^2$ and $D_i$ for $i \in \{1, 2, 3\}$ be the half-line $\{ t u_i \, : \, t \in \R_+ \}$ and $V_i = \cH^1_{| D_i} \otimes \delta_{\rm{span} (u_i)}$ be the smooth $1$--varifold associated with $D_i$. Then $V = V_1 + V_2 + V_3$ is a $1$--varifold spatially supported by the union of three half-lines meeting at $0$ and by linearity $\delta V = \delta V_1 + \delta V_2 + \delta V_3$. Now, we obtain $\delta V_i = -u_i \delta_0$ for $i \in \{1, 2, 3\}$. Indeed, for $X \in \xC_c^1 (\R^2, \R^2)$ and for $t \in \R_+$,\\
$\displaystyle \xdiv_{\rm{span}(u_i)} X(t u_i) = DX(tu_i)u_i \cdot u_i$ and $\displaystyle DX(tu_i)u_i = \frac{d}{dt} (X(tu_i))$, therefore:
\[
\delta V_i (X) = \int_{D_i} \xdiv_{\rm{span}(u_i)} X \, d \cH^1 = \int_{t=0}^{+ \infty} \frac{d}{dt} (X(tu_i)) \, dt \cdot u_i = - X(0) \cdot u_i = - \delta_0 (X) \cdot u_i \: .
\]
We eventually obtain $\delta V = - (u_1 + u_2 + u_3) \delta_0$ and in particular $\delta V = 0$ if and only if $0$ is a triple point with angles $\tfrac23 \pi$ formed by the half-lines. Otherwise, if $u_1 + u_2 + u_3 \neq 0$, the singularity in $0$ produces a non-zero singular curvature $(\delta V)_s = - (u_1 + u_2 + u_3) \delta_0$.
\end{exa}

We emphasize that the notion of first variation is well--defined for any $d$--varifold while the notion of generalized mean curvature requires some additional regularity of the varifold: it can be defined only if $V$ has locally bounded first variation  in the sense of \eqref{eq:locallyBoundedFirstVariation} that is equivalent to requiring that $\delta V$ identifies with a (vector--valued) Radon measure thanks to Riesz represnetation theorem.
In the previous example, though singular with respect to the mass measure $\V = \cH^1_{| D_1} + \cH^1_{| D_2} + \cH^1_{| D_3}$, the first variation $\delta V = - (u_1 + u_2 + u_3) \delta_0$ is a Dirac mass that is, in particular, a Radon measure. Unfortunately, the first variation of a point cloud varifold does not meet this asumption, as we now explain.
Let us consider a very simple point cloud $1$--varifold in $\R$ consisting of one single point $x_1 = 0$ weighted by mass $m_1 = 1$ and with $P_1 = \R$: $V = \delta_{(0,\R)}$. For $X \in \xC_c^1 (\R, \R)$ we obtain in this case
\[
\delta V(X) = \int_\R \xdiv_{\R} X \, d \delta_0 = \xdiv_\R X (0) = X^\prime (0) = \delta_0 (X^\prime) = - (\delta_0)^\prime (X) \: .
\]
It is well--known that the distributional derivative of a Dirac mass is not a Radon measure. In fact, this observation directly extends to any point cloud varifold 
$V = \sum_{i=1}^N m_i \delta_{(x_i, P_i)}$ and the first variation of $V$ is never locally bounded (take test functions whose support contains a single point $x_i$). This was the motivation for introducing in \cite{BuetLeonardiMasnou} approximate counterparts of first variation and generalized mean curvature via convolution. In this paper we will make extensive use of these notions. We hence summarize in the next section what is needed within the scope of this paper.

\section{Regularization via convolution} \label{sec:reg}

In the first part of this section we recall a regularization of the generalized mean curvature (see Definition~\ref{dfn:approxMeanCurvature}) based on the convolution of both first variation and mass with appropriately chosen kernels.
The original definition of this regularized mean curvature from \cite{BuetLeonardiMasnou} ensures consistency for a very large class of varifolds (rectifiable varfifolds with locally bounded first variation). In the second part of the section,
we will introduce
several variants of a regularized mean curvature vector which are all consistent with the mean curvature vector of a smooth hypersurface (see Proposition~\ref{prop:consistency}). The difference lays in the choice of a projection operator denoted by $\Pi$ hereafter.
As we will see in sections~\ref{sec:time} and \ref{sec:results} the concrete choice of the projection operator $\Pi$ matters on one hand in the presence of singularities and on the other hand for the numerical stability of the time discrete scheme.

From now on and for the sake of simplicity, we only consider varifolds with finite mass, that is $\V (\R^n) < \infty$ (as $\V$ is a Radon measure, $\V$ is finite when restricted to any bounded open set). We refer to \cite{BuetLeonardiMasnou} for additional details on steps leading to the notion of approximate mean curvature (see Definition~\ref{dfn:approxMeanCurvature}) that we now briefly sketch.

\noindent Let us consider fixed functions $\rho, \xi \in \xC^\infty(\R)$, which are nonnegative even and compactly supported in $[-1,1]$. We additionnaly assume that $\rho$ and $\xi$ are positive in $]0,1[$ and $\rho$ is nonincreasing in $[0,1]$.
Then we 
define associated mollifiers $(\rho_\epsilon)_{\epsilon > 0}$, $(\xi_\epsilon)_{\epsilon > 0}$ in $\R^n$ as follows, for $x \in \R^n$, $\epsilon > 0$,
$\displaystyle
\rho_\epsilon (x) = \epsilon^{-n} \rho (|x|/\epsilon)
$ and 
$\displaystyle
\xi_\epsilon (x) = \epsilon^{-n} \xi (|x|/\epsilon) \: .
$
Following \cite{BuetLeonardiMasnou}, we additionaly assume that for all $s \in ]-1,1[$,
\begin{equation} \label{eq:kernelPairs}
n \xi(s) = - s \rho^\prime(s) \: .
\end{equation}
For instance 
we can choose
\[
\rho(s) = \exp \left(\frac{1}{s^2-1} \right) \quad \text{and} \quad \xi(s) = \frac{2}{n} \frac{s^2}{(s^2-1)^2} \exp	\left( \frac{1}{s^2-1} \right) \quad \text{for } s \in ]-1,1[ \: .
\]
The $\epsilon$--\emph{regularized first variation} is then defined as the convolution of $\delta V$ with $\rho_\epsilon$: for $X \in \xC_c^1 (\R^n, \R^n)$,
\begin{equation*} 
\left( \delta V \ast \rho_\epsilon \right) (X) \coloneqq \delta V (X \ast \rho_\epsilon) \quad \text{with} \quad (X \ast \rho_\epsilon)(x) \coloneqq \int_{y \in \R^n} \rho_\epsilon(x-y) X(y) \: dy \: .
\end{equation*}
Easy computations lead to
$\displaystyle
\left( \delta V \ast \rho_\epsilon \right) (X) = \int_{\R^n} g_\epsilon(x) \cdot X(x) \, dx$, where
\begin{align} \label{eq:convV}
g_\epsilon(x)  
 = \frac{1}{\epsilon^{n+1}} \int_{\R^n \times \G} \rho^\prime \left(\frac{|y-x|}{\epsilon} \right) \Pi_S \left( \frac{y-x}{|y-x|}  \right) \, dV(y,S)  =: \delta V \ast \rho_\epsilon(x) \quad \text{for } x \in \R^n \: .
\end{align}
Consequently $\delta V \ast \rho_\epsilon$ identifies with its Lebesgue $\xL^1$--density $g_\epsilon \in \xL^1_{loc} (\R^n)$.
\begin{rem}
Note that $\rho^\prime(|z|) \frac{z}{|z|} \xrightarrow[|z| \to 0]{} 0$ thanks to $\rho^\prime(0) = 0$.
\end{rem}

In order to define a regularized generalized mean curvature, it remains to write the Radon-Nikodym decomposition of the regularized first variation $\delta V \ast \rho_\epsilon$ with respect to the regularized mass $\V \ast \xi_\epsilon$.
As for the first variation, we identify $\V \ast \xi_\epsilon$ with the associated Lebesgue $\xL^1$--density function defined as
\[
\V \ast \xi_\epsilon (x) = \frac{1}{\epsilon^n} \int_{\R^n} \xi \left( \frac{|y-x|}{\epsilon} \right) \, d \V(y)  , \quad x \in \R^n \: .
\]
Considered as Radon measures, both $\delta V \ast \rho_\epsilon$ and $\V \ast \xi_\epsilon$ are absolutely continuous with respect to Lebesgue measure and consequently the Radon-Nykodym decomposition of $\delta V \ast \rho_\epsilon$ with respect to $\V \ast \xi_\epsilon$ is
\[ 
\delta V \ast \rho_\epsilon = \frac{\delta V \ast \rho_\epsilon(x)}{\V \ast \xi_\epsilon (x)}  \left( \V \ast \xi_\epsilon \right) = \underbrace{ \frac{\delta V \ast \rho_\epsilon(x)}{\V \ast \rho_\epsilon (x)} }_{\displaystyle \to \frac{\delta V}{\V}} \underbrace{\frac{\V \ast \rho_\epsilon (x)}{\V \ast \xi_\epsilon (x)} }_{\displaystyle \to \frac{\int_0^1 \rho(s) s^{d-1} \: ds}{\int_0^1 \xi(s) s^{d-1} \: ds} = \frac{n}{d}}  \left( \V \ast \xi_\epsilon \right)
\]
Finally, we obtain the following definition for the approximate mean curvature
\begin{defin}[Approximate mean curvature (\cite{BuetLeonardiMasnou}, 4.1)] \label{dfn:approxMeanCurvature}
Let $V$ be a $d$--varifold in $\R^n$, for $x \in \R^n$, $\epsilon > 0$, we define
\begin{equation} \label{eq:approxMeanCurvature}
H_\epsilon (x,V) \coloneqq - \frac{d}{n} \frac{\delta V \ast \rho_\epsilon (x)}{\V \ast \xi_\epsilon(x)}
\end{equation}
which we refer to as \emph{$\epsilon$--approximate curvature}.
\end{defin}
The ratio $\tfrac{d}{n}$ is due to the particular choice of $\xi$ and $\rho$ fulfilling \eqref{eq:kernelPairs}.
Note that we could take other kernel functions not necessarily satisfying relation \eqref{eq:kernelPairs}) to regularize the first variation $\delta V$ and the mass $\V$ (e.g. the same kernel both for the variation and the mass). In this case the consistency when $\epsilon \to 0$ would still hold replacing $d/n$ in \eqref{eq:approxMeanCurvature} by the appropriate constant. Nevertheless, \eqref{eq:kernelPairs} gives better numerical results that can be understood by expanding $|H(x) - H_\epsilon (x,V)|$ when $\epsilon \to 0$ (see \cite{BuetLeonardiMasnou}). We point out that this very same asumption \eqref{eq:kernelPairs} on the kernels enables us to prove a discrete maximum principle on our time discrete scheme (see Proposition~\ref{discreteSphereInclusion}).

Finally, let us remark that no asumption on $\delta V$ is necessary to define $H_\epsilon (\cdot, V)$.  
It is well--defined even if $\delta V$ is not locally bounded. In particular for a point cloud varifold $\displaystyle V = \sum_{j=1}^N m_j \delta_{(x_j,P_j)}$ we obtain for $i \in \{1, \ldots, N\}$
\[
H_\epsilon (x_i,V) = - \frac{d}{n \epsilon} \frac{ \displaystyle \sum_{\substack{j=1,  j \neq i}}^N m_j \rho^\prime \left(\frac{|x_j-x_i|}{\epsilon} \right) \Pi_{P_j} \left( \frac{x_j-x_i}{|x_j-x_i|}  \right) }{\displaystyle \sum_{j=1}^N m_j \xi \left( \frac{|x_j-x_i|}{\epsilon} \right)} \: .
\]

In addition to consistency when $\epsilon \to 0$, it is furthermore possible to state stability and convergence results with respect to a localized flat distance between varifolds (see \cite{BuetLeonardiMasnou}, thm 4.3 $\&$ 4.5). When dealing with smooth objects, expression \eqref{eq:approxMeanCurvature} can be modified in various ways preserving its consistency. In the rest of this section, we investigate several variants of approximate mean curvatures and establish their consistency with the classical mean curvature for smooth hypersurfaces in Proposition~\ref{prop:consistency}. While their is no significant benefit from considering those variants for computing an approximate curvature on a smooth hypersurface, they lead to numerical schemes for curvature flow behaving quite differently as shown in Section~\ref{sec:time} and \ref{sec:results}.

More precisely, replacing $\Pi_{S}$ in \eqref{eq:convV} by some linear operator
$\Pi : \R^n \rightarrow \R^n$ that may depend on $x_0$, $x \in \R^n$ and $S \in \G$ we take into account 
\begin{equation} \label{eq:projMeanCurvature}
H_\epsilon^{\proj} (x_0,V) = - \frac{d}{n \epsilon} \frac{\displaystyle \int_{\R^n \times \G} \rho^\prime \left( \frac{|x-x_0|}{\epsilon} \right) \frac{\Pi  (x-x_0)}{|x-x_0|} \, dV(x,S) }{\displaystyle \int_{\R^n} \xi \left( \frac{|x-x_0|}{\epsilon} \right) \, d \V(x)}\,.
\end{equation}
We will consider $\Pi  \in \{ \Pi_S, \, -2 \,\Pi_{S^\perp}, \,  2\, \rm{Id} \}$, also post-composed with a projection onto the normal space at $x_0$.
Notice that $\Pi = \Pi_S$ exactly corresponds to Definition~\ref{dfn:approxMeanCurvature}.

\begin{prop}[Consistency for smooth varifolds] \label{prop:consistency} Assume $d = n-1$ and let $M \subset \R^n$ be a $d$--submanifold of class $\xC^2$, whose mean curvature vector is denoted by $H : M \rightarrow \R^n$, and let $V = \cH^d_{| M} \otimes \delta_{T_x M}$. Then, one obtains that 
\[
H_\epsilon^{\proj} (x_0, V) \xrightarrow[\epsilon \to 0]{} H(x_0) \: .
\]
for all $x_0 \in M$ and for 
$\displaystyle \Pi \in \left\lbrace \Pi_S, \, -2 \Pi_{S^\perp}, \,  2 \rm{Id}, \, \Pi_{(T_{x_0} M)^\perp} \circ \Pi_S , \, -2 \Pi_{(T_{x_0} M)^\perp} \circ \Pi_{S^\perp}, \, 2 \Pi_{(T_{x_0} M)^\perp}  \right\rbrace\,$.
If $M$ is at least $\xC^3$ then $\displaystyle \left| H_\epsilon^{\proj} (x_0,V) - H(x_0) \right| = O(\epsilon)$.

\end{prop}
Note that a corresponding result holds true for any codimension greater or equal than $1$. 
Nevertheless, for the sake of simplicity and because we consider only the codimension $1$ case in our numerical applications we state and prove the result only for $d = n-1$.

Before we prove this proposition in the general case, let us depict the simplest case of a quadratic curve in $\R^2$. 
Taking into account $n\xi(s) = -s \rho(s)$ we obtain
\[
H_\epsilon(x_0,V) = \frac{\int_M \xi(\tfrac{x-x_0}{\epsilon}) \tfrac{1}{|x-x_0|^2} \Pi (x-x_0)\; \mathrm{d}\mathcal{H}^1}{\int_M \xi(\tfrac{x-x_0}{\epsilon})\; \mathrm{d}\mathcal{H}^1}
\]
We define $M=\{x=(y,-ay^2)\,|\, y \in \R\}$ and set $x_0=0$. Then the normal on $M$ is $\nu = \tfrac{(2ay,1)}{\sqrt{1+4a^2y^2}}$.
For $\Pi=2 \rm{Id}$ we get  $\tfrac{1}{|x|^2} \Pi x = \tfrac{(2y/y^2,-2a)}{(1+a^2y^2)}$ and thus using the symmetry in $y$ we get $H_\epsilon(x_0,V)= (0,-2a) + O(\epsilon^2)$.
For $\Pi=\Pi_{T_{x}M}$, one computes 
$\tfrac{1}{|x|^2}  \Pi x = \tfrac{1}{y^2+a^2y^4} [(y,-ay^2) - \tfrac{1}{1+4a^2y^2} (4 a^2 y^3,a y^2)]$ which again implies $H_\epsilon(x_0,V)= (0,-2a) + O(\epsilon^2)$.
For the other choices of $\Pi$ listed in the proposition analoguous and easy computations lead to the same approximation result. 
\begin{proof}
Up to an affine isometry, we can assume that $x_0 = 0$ and $T_{x_0} M = \R^{n-1} \times \{0\}$. We locally parametrize $M$ by $u : \cU \subset \R^{n-1} \rightarrow \R$ of class $\xC^2$ 
on an open set $\cU$ containing $0$. Consequently, for all $0 < \epsilon < \epsilon_0 \leq 1$ with $\epsilon_0$ small enough,
\[
M \cap B (0, \epsilon) = \{ (y, u(y)) \in \cU \times \R \: : \: |y|^2 + |u(y)|^2 < \epsilon^2 \} 
\]
and $M \cap B (0, \epsilon)$ is the graph of $u$ over the open set
$\cV_\epsilon = \{ y \in \cU \: : \: |y|^2 + |u(y)|^2 < \epsilon^2 \} \subset \R^{n-1}$. Thereby we have for the mean curvature $H(0) = \rm{trace} (\Hessian u (0) e_n)$ with $\Hessian u$ denoting the Hessian of $u$. Taking into account
\begin{equation} \label{eq:graphAsumption}
u(y) = \frac{1}{2}\Hessian u (0)y \cdot y + o(|y|^2) \quad \text{and} \quad \nabla u(y) = \Hessian u(0) y + o(|y|)\: ,
\end{equation}
we have for the normal vector $\nu(y)$ to $M$ at $(y,u(y)) \in M \cap B(0, \epsilon)$,
\begin{equation} \label{eq:graphAsumption2}
\nu(y) = \frac{(\nabla u(y), -1)}{\sqrt{1 + |\nabla u(y)|^2}} =  (1 + o(|y|)) \begin{pmatrix} \Hessian u(0)y + o(|y|) \\ -1 \end{pmatrix} = \begin{pmatrix} \Hessian u(0)y  \\ -1 
\end{pmatrix}  + o(|y|) \: .
\end{equation}
By definition of the approximate mean curvature \eqref{eq:projMeanCurvature}, 
\begin{align} \label{eq:consitency1}
\left| H_\epsilon^{\proj} (0,V) - H(0) \right| = \frac{ \displaystyle \left| \int_{\cV_\epsilon}f_\epsilon (y,u(y)) \sqrt{1 + |\nabla u (y)|^2} \: dy \: \right|}
{\displaystyle \int_{\cV_\epsilon} \xi \left( \frac{|(y,u(y))|}{\epsilon}\right) \sqrt{1 + |\nabla u (y)|^2} \: dy }
\end{align}
with
$\displaystyle f_\epsilon (z) =  \frac{-(n-1)}{n \epsilon} \rho^\prime \left( \frac{|z|}{\epsilon}\right) \frac{\Pi z }{|z|}  - H(0)  \xi \left( \frac{|z|}{\epsilon} \right)$.

\noindent Up to decreasing $\epsilon_0$ we can assume $D(0, \epsilon_0)  \coloneqq \{ y \in \R^{n-1} \: : \: |y| < \epsilon_0 \} \subset \cU$.
At first, we simplify the nominator and denominator via expansion of the area element and slightly enlarging the integration domain. Obviously, $\cV_\epsilon \subset D(0,\epsilon)$ and 
there exists $\kappa > 0$ such that for all $y \in D(0,\epsilon)$, $|u(y)| \leq \kappa |y|^2$. For $y \in \cU$ 
and $|y| < \eta$ with $\eta = \epsilon  \sqrt{ 1 - \kappa^2 \epsilon^2}$ we obtain
\[
|y|^2 + |u(y)|^2 < \eta^2 + \kappa^2 \eta^4 = \epsilon^2 ( 1 - \kappa^2 \epsilon^2) + \kappa^2  \epsilon^4 ( 1 -  \kappa^2 \epsilon^2)^2 \leq \epsilon^2 \: ,
\] 
which implies that $D(0, \eta) \subset \cV_\epsilon\:.$ 
Moreover, notice that for $g$ bounded and continuous on $D(0,\epsilon)$, and using $D(0,\eta ) \subset \cV_\epsilon \subset D(0,\epsilon)$,
\begin{multline}
\left| \int_{D(0,\epsilon)} g - \int_{\cV_\epsilon} g \, \right| \leq \sup_{D(0, \epsilon)} |g|\, | D(0,\epsilon) - \cV_\epsilon | 
\leq \sup_{D(0, \epsilon)} |g| \,| D(0,\epsilon) - D(0,\eta) |  \\
= \sup_{D(0, \epsilon)} |g|\, \omega_{n-1} \epsilon^{n-1} \left( 1 - (1 - (\kappa \epsilon)^2)^{\frac{n-1}{2}} \right) = O(\epsilon^{n+1}) \: . \label{eq:enlargementDomain}
\end{multline}
As $\rho^\prime (|z|) \frac{z}{|z|} = \nabla (\rho (|z|))$ is uniformly bounded and $\| \Pi \| \leq 2$, we infer that the continuous map
$\epsilon f_\epsilon((y,u(y))) \sqrt{1 + |\nabla u(y)|^2}$ is uniformly bounded in $D(0,\epsilon)$. Now, we apply \eqref{eq:enlargementDomain} and due to 
$\sqrt{1 + |\nabla u(y)|^2} = 1  + O(|y|^2)$ as well as $\epsilon^{-1} \int_{D(0,\epsilon)} O(|y|^2) dy = O(\epsilon^n)$ deduce
\begin{align}
\int_{\cV_\epsilon} f_\epsilon ((y,u(y)) \sqrt{1 + |\nabla u (y)|^2} \, dy & = \int_{D(0,\epsilon)} f_\epsilon ((y,u(y)) \sqrt{1 + |\nabla u (y)|^2} \, dy + \epsilon^{-1} O(\epsilon^{n+1}) \nonumber \\
& =  \int_{D(0,\epsilon)} f_\epsilon ((y,u(y)) \, dy + O(\epsilon^n) \: . \label{eq:parametricFormulation}
\end{align}
Analogously, we get for the denominator
\begin{align}
\int_{\cV_\epsilon}  \xi \left( \frac{|(y,u(y))|}{\epsilon}\right) \sqrt{1 + |\nabla u (y)|^2} \, dy  & = 
\int_{D(0,\epsilon)}\xi \left( \frac{|(y,u(y))|}{\epsilon}\right)  \, dy +  O(\epsilon^{n+1})  \: . \label{eq:parametricFormulationDenominator}
\end{align}
Next, we take into account the expansion of kernels involved in \eqref{eq:consitency1}.
Thanks to \eqref{eq:graphAsumption} we deduce 
$\displaystyle
|z| = |(y,u(y))| = \left( |y|^2 + |u(y)|^2 \right)^{\frac{1}{2}} = \left( |y|^2 + O(|y|^4) \right)^{\frac{1}{2}} = |y| (1+ O(|y|^2))$ so that for $y \in D(0, \epsilon)$,
\begin{align*}
\xi \left( \frac{|z|}{\epsilon} \right) &= \xi \left( \frac{|y|}{\epsilon} + \frac{|y|}{\epsilon} O (|y|^2) \right) = \xi \left( \frac{|y|}{\epsilon} \right) + \frac{|y|}{\epsilon} O (|y|^2) = \xi \left( \frac{|y|}{\epsilon} \right) +  O (|y|^2)\:,  \\
 \frac{1}{|z|} \rho^\prime \left( \frac{|z|}{\epsilon} \right) &= \frac{1}{|y| + O(|y|^3)} \left( \rho^\prime \left( \frac{|y|}{\epsilon} \right) + O (|y|^2) \right) = \frac{1}{|y|} \rho^\prime \left( \frac{|y|}{\epsilon} \right)  + O(|y|) \: . 
\end{align*}
It remains to expand $\Pi z$ for $z=(y,u(y))$. 

\noindent \emph{Case $\Pi = 2 \rm{Id}$: } We directly obtains $\displaystyle \Pi z = 2 (y, u(y)) = \begin{pmatrix}
2y \\ \Hessian u(0) y \cdot y + o(|y|^2)
\end{pmatrix}$.\\
\noindent \emph{Case $\Pi = -2 \Pi_{(T_z M)^\perp}$:}  Using the expansion of the normal  $\nu(y)$ from \eqref{eq:graphAsumption2} we obtain 
\begin{align*}
\Pi z   & =    -2 (z \cdot \nu(y)) \nu(y) \\
& = -2 \left( \begin{pmatrix} y \\ \frac{1}{2} \Hessian u(0)y \cdot y + o(|y|^2) \end{pmatrix}  \cdot 
\begin{pmatrix} \Hessian u(0)y + o(|y|) \\ -1 + o(|y|)  \end{pmatrix} \right) \; \begin{pmatrix} \Hessian u(0)y + o(|y|) \\ -1 + o(|y|)  \end{pmatrix}  \\
& =  - \left( \Hessian u(0) y \cdot y + o(|y|^2)\right) \begin{pmatrix} \Hessian u(0)y + o(|y|) \\ -1 + o(|y|)  \end{pmatrix} = \begin{pmatrix} o(|y|^2) \\  \Hessian u(0) y \cdot y + o(|y|^2) \end{pmatrix}.
\end{align*}
\noindent \emph{Case $\Pi = \Pi_{T_z M}$:}  Taking into accout the estimates in the previous two cases we achieve 
\begin{align*}
\Pi z =  \frac{1}{2} \left( 2 \rm{Id} - 2 \Pi_{(T_z M)^\perp} \right) z  = \begin{pmatrix}
y + o(|y|^2) \\ \displaystyle \Hessian u(0) y \cdot y + o(|y|^2) \end{pmatrix}\:.
\end{align*}
To summarize, in all three cases, we obtain
\begin{equation} \label{eq:projectorExpansions}
\Pi z \cdot e_i = \lambda y_i + o(|y|^2) \; \forall i \in \{ 1, \ldots, n-1 \} \quad \text{and} \quad \Pi z \cdot e_n = \Hessian u(0) y \cdot y + o(|y|^2) \: ,
\end{equation}
with $\lambda = 2$ for $\Pi = 2 \rm{Id}$, $\lambda = 0$ for $\Pi = -2 \Pi_{(T_z M)^\perp}$, and $\lambda = 1$ for $\Pi = \Pi_{T_z M}$.
\medskip
Now, applying the co-area formula to \eqref{eq:parametricFormulation} and using the above kernel expansions for $z = (y,u(y))$ we get
\begin{align}
\int_{D(0,\epsilon)}  \xi \left( \frac{| z |}{\e} \right) \, dy & =  \int_{D(0,\epsilon)} \xi \left( \frac{|y|}{\epsilon} \right) +  O (|y|^2) \, dy \nonumber \\
& =  \int_{r=0}^\e \left( \xi \left( \frac{r}{\e} \right) + O(r^2) \right) \cH^{n-2} \left( \partial D(0,r) \right) \, dr \nonumber \\
& = \sigma_{n-2}  \int_{r=0}^\e  \xi \left( \frac{r}{\e} \right) r^{n-2} \, dr + O(\e^{n+1}) \: , \label{eq:consistency3}
\end{align}
where $\sigma_{n-2} = (n-1) \omega_{n-1}$ is the area of the unit sphere in $\R^{n-1}$. Using \eqref{eq:projectorExpansions} we obtain for $i \in \{ 1, \ldots, n-1 \}$,
\begin{align}
\int_{D(0,\epsilon)} \frac{1}{|z|} \rho^\prime \left( \frac{| z |}{\e} \right) \Pi (z) \, dy \cdot e_i & =  \int_{D(0,\epsilon)} \left[ \frac{1}{|y|} \rho \left( \frac{|y|}{\epsilon} \right) +  O (|y|) \right] ( \lambda y_i + o(|y|^2 ) ) \, dy \nonumber \\
& =  \int_{r=0}^\e \left( \frac{1}{r}  \rho^\prime \left( \frac{r}{\e} \right) + O(r) \right) \int_{\partial D(0,r)} (\lambda y_i + o(r^2)) \, d \cH^{n-2} \, dr \nonumber \\
& = \int_{r=0}^\e \left( \frac{1}{r}  \rho^\prime \left( \frac{r}{\e} \right) + O(r) \right)  r^{n-2} o(r^2) \, dr  = \int_{r=0}^\epsilon o(r^{n-1}) dr \nonumber \\
& = o(\e^{n}) \: , \label{eq:consistency4}
\end{align}
where we used that $\displaystyle \int_{\partial D(0,r)}  y_i \, d \cH^{n-2}$ vanishes.
Furthermore, using the kernel expansions and \eqref{eq:projectorExpansions} we achieve
\begin{align}
\int_{D(0,\epsilon)} \frac{1}{|z|} \rho^\prime \left( \frac{| z |}{\e} \right) \Pi (z) \, dy \cdot e_n &  =  \int\limits_{r=0}^\e \left( \frac{1}{r}  \rho^\prime \left( \frac{r}{\e} \right) + O(r) \right) \int_{\partial D(0,r)} \hspace{-5pt} (\Hessian u(0) y \cdot y + o(r^2))  d \cH^{n-2}  dr \nonumber \\
& = \int\limits_{r=0}^\e \ \frac{1}{r}  \rho^\prime \left( \frac{r}{\e} \right) \int_{\partial D(0,r)} \hspace{-5pt} (\Hessian u(0) y \cdot y ) \:  d \cH^{n-2} \:  dr  +  o (\e^n) \: . \label{eq:consistency5}
\end{align}
Next, we verify that  $\displaystyle \int_{\partial D(0,r)} (\Hessian u(0) y \cdot y ) \:  d \cH^{n-2} = H(0) \cdot e_n  \frac{\sigma_{n-2}}{n-1} r^n$. 
Indeed,  let $\{v_1, \ldots, v_{n-1}\}$ be an orthonormal basis of $\R^{n-1}$ of eigenvectors of $\Hessian u(0)$ associated with eigenvalues $\kappa_1, \ldots, \kappa_{n-1}$. 
Then decomposing $\displaystyle y = \sum_{j=1}^{n-1} \hat y_j v_i$ in $\cB$, we have
$\displaystyle \Hessian u(0) y \cdot y = \sum_{j = 1}^{n-1} \kappa_j \hat y_j^2$
and by symmetry $\displaystyle \int_{\partial D(0,r)} \hat y_j^2 \, dy = \frac{1}{n-1} \int_{\partial D(0,r)} |y|^2 \, dy = \frac{\sigma_{n-2}}{n-1} r^n$. Hence, the claim follows from 
\begin{equation} \label{eq:consistency6}
\int_{\partial D(0,r)} (\Hessian u(0) y \cdot y ) \:  d \cH^{n-2} = \sum_{j=1}^{n-1} \kappa_j \frac{\sigma_{n-2}}{n-1} r^n = \frac{\sigma_{n-2}}{n-1} r^n \; {\rm{trace}} \, \Hessian u(0).
\end{equation}
Let $\Pi \in \left\lbrace \Pi_S, \, -2 \Pi_{S^\perp}, \,  2 \rm{Id} \right\rbrace$. Gathering the above estimates (\eqref{eq:parametricFormulation} to \eqref{eq:consistency6}) and inserting them in \eqref{eq:consitency1}, we conclude
\begin{align*}
\left| H_\epsilon^{\proj} \right.&\left. (0,V)  - H(0) \right| \\
& = \left( \int_{r=0}^\e \xi \left( \frac{r}{\epsilon} \right) r^{n-2} \, dr + O(\e^n) \right)^{-1}   \left| H(0) \int_{r=0}^\e \left( \frac{1}{n} \frac{r}{\e} \rho^\prime \left( \frac{r}{\e} \right) + \xi \left( \frac{r}{\e} \right) \right) r^{n-2} \, dr + o( \epsilon^{n-1} ) \right| \\
& =  \frac{\e^{-(n-1)}}{C_\xi + O(\e)} \left(|H(0)| \: \epsilon^{n-1} \int_{s=0}^1 \left( \frac{1}{n} s \rho^\prime (s) + \xi(s) \right) s^{n-2} \, ds + o (\e^{n-1}) \right)\\
& = 0 + o(1) \: ,
\end{align*}
with $\displaystyle C_\xi = \int_{s=0}^1 \xi(s) s^{n-2} \, ds$ and $\displaystyle \int_{s=0}^1 \left( \frac{1}{n} s \rho^\prime (s) + \xi(s) \right) s^{n-2} \, ds = 0$
by asumption \eqref{eq:kernelPairs}.

As $H(0)$ is othogonal to $T_0 M$, $\Pi_{(T_0 M)^\perp} H(0) = H(0)$ and thus for $\Pi \in \{ \Pi_S, -2 \Pi_{S^\perp}, 2 \rm{Id}  \}$ we have
$$\displaystyle \left| \Pi_{(T_0 M)^\perp} H_\epsilon^{\proj} (0,V) - H(0) \right| = o(1)\,.$$ 
Finally, it is straightforward to verify that for $M$ being at least $\xC^3$ one obtains the improved convergence estimate $\displaystyle \left| H_\epsilon^{\proj} (0,V) - H(0) \right| = O(\epsilon)$.
\end{proof}

In summary, consistency with the usual mean curvature holds for smooth varifolds (Proposition~\ref{prop:consistency}). In the particular case $\Pi = \Pi_S$, consistency with the generalized mean curvature holds almost everywhere (w.r.t. the mass measure) for rectifiable varifolds whose first variation is a Radon measure (\cite{BuetLeonardiMasnou}).
In view of stating convergence of approximate mean curvature $H_{\epsilon} (\cdot, V_i)$ computed on point cloud varifolds $(V_i)_i$, when $\epsilon$ tends to $0$ and $V_i$ tends to a smooth (or rectifiable with first variation Radon) varifold, we tackle the stability issue in next section.

\section{Stability of the approximate mean curvature} \label{sec:stability}
 
We are now going to state a quite general stability result on the approximate mean curvature introduced in \eqref{eq:projMeanCurvature}.
The stability will hold with respect to weak star convergence of varifolds, assuming that the limit varifold has finite mass and is \emph{$d$--regular} in the sense that its mass $\V$ is $d$--Ahlfors regular, i.e.~there exists $C_0 \geq 1$, $r_0 > 0$ such that for all $x \in \supp \V$ and $0 < r \leq r_0$,
\begin{equation} \label{eq:ahlfors}
C_0^{-1} r^d \leq  \V(B(x,r)) \leq C_0 r^d \: .
\end{equation}
Note that $r_0$ can be chosen as large as needed : if condition \eqref{eq:ahlfors} holds for some $r_0 > 0$ then it holds for any $r_1 \geq r_0$ possibly adapting the regularity constant $C_0$.\\
Part of the result  
can be easily obtained by adapting the case $\Pi = \Pi_S$ dealt with in \cite{BuetLeonardiMasnou} (Theorem 4.5), using the fact that the family of maps (indexed by $x \in \R^n$)
\begin{equation} \label{eq:deltaPi}
\Phi_x^\epsilon : (y,S) \mapsto \frac{1}{\epsilon} \rho^\prime \left( \frac{|y-x|}{\epsilon} \right) \frac{\Pi  (y-x)}{|y-x|} = \Pi \nabla_y \left( \rho \left( \frac{|y-x|}{\epsilon} \right) \right) 
\end{equation}
for $\Pi$ as in \eqref{eq:projMeanCurvature}  is equi-Lipschitz with Lipschitz constant bounded by $\displaystyle \epsilon^{-2}\xLip (\rho^\prime)$. 
The other part of the result requires some work on weak star and weak convergence of finite Radon measures as well as on the \emph{flat distance} and \emph{Prokhorov distance} that metrizes weak convergence (see Proposition~\ref{thm:weakAndDistance} below).
The section is organized as follows, in a first part 
we introduce some material on weak and weak star convergences as well as \emph{flat distance} and \emph{Prokhorov distance}. In a second part we establish a general stability result (Proposition~\ref{prop:stability}) holding in the neighbourhood (with respect to aforementioned distances) of a $d$--regular varifold. Let us emphasize that $d$--regularity of a Radon measure is a weak asumption when it comes to prove stability for curvature estimates. In the third and last part, we put the results of the section together and state a convergence theorem (Theorem~\ref{thm:stability} and Corollary~\ref{cor:CV}) for the approximate mean curvature \eqref{eq:projMeanCurvature} of a sequence of weak star converging varifods.

All Radon measures we consider in this section are {\bf nonnegative} (and nonzero) Radon measures.

\subsection{Prokhorov and flat distance} \label{subsec:distances}

\begin{defin}[Weak and weak star convergence] \label{dfn:weakCV}
Let $(X,d)$ be a locally compact and separable metric space (for us $X = \R^n$ or $X = \R^n \times \G$) and let $(\mu_i)_{i \in \N}$, $\mu$ be Radon measures in $X$. We say that
\begin{enumerate}[$(i)$]
\item $(\mu_i)_i$ \emph{weak star} converges to $\mu$ if for every $\phi : X \rightarrow \R$ continuous and {\bf compactly supported}, 
$\displaystyle \int \phi \: d \mu_i \xrightarrow[i \to \infty]{} \int \phi \: d \mu $.
\item If in addition the measures $(\mu_i)_i$, $\mu$ are finite, we say that $(\mu_i)_i$ \emph{weak} converges to $\mu$ if for every $\phi : X \rightarrow \R$ continuous and {\bf bounded}, $\displaystyle \int \phi \: d \mu_i \xrightarrow[i \to \infty]{} \int \phi \: d \mu $.
\end{enumerate}
\end{defin}

Weak star convergence is also referred as vague convergence. By definition, weak convergence implies weak star convergence, whereas the converse is not true in general: compactly supported functions are blind to mass escaping at infinity or accumulating on the boundary. Consider for instance $\mu_i = \delta_i$ in $\R$ or $\mu_i = \delta_{\frac{1}{i}}$ in $]0,1[$. In both cases $\mu_i$ weak star converge to $0$ but does not weak converge.

A simple example of weak star convergence of varifolds is defined via a sequence of sawtooth type 
polygonal curves $C_\rho$ of amplitude $\rho^\alpha$ for some $\alpha \geq 1$
and frequency $\tfrac{1}{2\rho}$ oscillating around the $e_1$ axis in $\R^2$. 
Given these curves we define a sequence of varifolds $V_\rho$ via
\[
V_\rho(B) = (1+\rho^{2(\alpha-1)})^{-\frac12} \mathcal{H}^1(\{x\in C_\rho\,|\, (x, T_xC_\rho) \in B\}) \quad \text{i.e.} \quad V_\rho = (1+\rho^{2(\alpha-1)})^{-\frac12} \mathcal{H}^1_{| C_\rho} \otimes \delta_{T_xC_\rho} \: .
\]
We denote by $D_{e_1} \subset \R^2$ the straight line along the $e_1$ axis. For $\alpha > 1$, the family $(V_\rho)$ converges weak star for $\rho$ tending to $0$ to the varifold $V = \cH^1_{D_{e_1}} \otimes \delta_{{\rm span}(e_1)}$ that is
the smooth varifold associated with the straight line $D_{e_1}$.
For $\alpha=1$ the  weak star limit of the family of varifolds is
 $V = \cH^1_{D_{e_1}} \otimes \frac{1}{2}\left( \delta_{{\rm span}(e_1 + e_2)} + \delta_{{\rm span}(e_1 - e_2)} \right)$,
note that the measure in the Grassmannian is constant along the $e_1$ axis and consists of $2$ atomic weights in the Grassmannian.

Note that finite Radon measures inherit the Banach structure of linear forms on $C_c (X)$ through Riesz representation theorem. However, the resulting total variation distance
\[
d_{TV} (\mu, \nu) = | \mu - \nu |(X) = \sup \left\lbrace \left. \int_X \phi \: d \mu - \int_X \phi \: d \nu \: \right| \phi \in \xC_c (X), \: \sup | \phi | \leq 1 \right\rbrace
\]
is much too strong with respect to compactness issues as well as approximation questions. Indeed, if $x$, $y \in X$ then whenever $x \neq y$,
$
d_{TV} (\delta_x, \delta_y) = 2
$ no matter how small $d(x,y)$ is. Therefore, we introduce the so-called \emph{flat distance} and \emph{Prokhorov distance} that behave more consistently with weak and weak star topologies.

\begin{defin}[Flat distance] \label{dfn:flatDistance}
Let $(X,d)$ be a locally compact and separable metric space (for us $X = \R^n$ or $X = \R^n \times \G$) and let $\mu$, $\nu$ be Radon measures. We define the (localized) \emph{bounded Lipschitz distance} in the open set $U \subset X$:
\[
\Delta_U (\mu, \nu) 
: = \sup \left\lbrace \int_X \phi \, d \mu - \int_X \phi \, d \nu  \left| 
\begin{array}{l}
\phi \text{ is } 1\text{--Lipschitz} \\
 \sup_X | \phi | \leq 1 \\
\supp \: \phi \subset U
\end{array} \right. \right\rbrace \: ,
\]
in the case $U = X$, we simply denote $\Delta = \Delta_X$.
Note that $\Delta$ is a distance in the space of Radon measures.
\end{defin}

For balls of radius less than $1$, $\sup_X | \phi | \leq 1$ is automatically satisfied in Definition~\ref{dfn:flatDistance}.

\begin{defin}[Prokhorov distance] \label{dfn:prokhorovDistance}
Let $\mu$, $\nu$ be finite Radon measures in $\R^n$ we recall that the 
\emph{Prokhorov distance} is defined as
\[
d_{\cP} (\mu, \nu) \coloneqq \inf \left\lbrace \delta > 0 \: \left| \: \mu(A) \leq \nu (A^\delta) + \delta \text{ and } \nu(A) \leq \mu(A^\delta) + \delta, \: \forall \: A \subset X \text{ Borel set}  \right. \right\rbrace \: ,
\]
with $A^\delta = \bigcup_{x \in A} B(x,\delta)$.
We introduce a slightly modified version of Prokhorov distance, for $d \in \N^\ast$,
\[
\eta_d (\mu, \nu) \coloneqq \inf \left\lbrace \delta > 0 \: \left| \: \mu(B) \leq \nu (B^\delta) + \delta^d \text{ and } \nu(B) \leq \mu(B^\delta) + \delta^d, \: \forall \: B \subset \R^n \text{ closed ball}  \right. \right\rbrace \: .
\]
\end{defin}

As Radon measures we work with are $d$--dimensional, homogeneity considerations lead to the modified Prokhorov distance $\eta_d$ introduced in Definition~\ref{dfn:prokhorovDistance}.
Notice that ${\overline{A}}^\delta = A^\delta$ for a Borel set $A \subset \R^n$ and then $\mu (A) \leq \mu(\overline{A}) \leq \nu(A^\delta) + \delta$ and it is natural to restrict to closed sets. Moreover, the restriction to balls is due to the convergence result we are interested in, nevertheless, as $\mu$ and $\nu$ are Radon measures in $\R^n$, if they coincide on balls then they are equal (thanks to Radon-Nikodym differentiation theorem for Radon measures) and thus $\eta_d$ defines a distance among finite Radon measures. The triangular inequality is straightforward, using that $a^d + b^d \leq (a+b)^d$ for $a,b \geq 0$ and $d$ a positive integer.
The next proposition connects weak and weak star topologies and topology induced by both Prokhorov and flat distances.

\begin{prop} \label{thm:weakAndDistance}
Let $(X,d)$ be a locally compact separable metric space (for us $X = \R^n$ or $X = \R^n \times \G$) and let $\mu$, $(\mu_i)_{i \in \N}$ be finite (nonzero) Radon measures. 
\begin{enumerate}
\item $(\mu_i)_i$ weak star converges to $\mu$ and $\mu_i(X) \xrightarrow[i\to \infty]{} \mu(X)$ if and only if $(\mu_i)$ weak converges to $\mu$.
\item If $(\mu_i)_i$ weak converges to $\mu$, then both $\Delta (\mu_i, \mu) \xrightarrow[i\to \infty]{} 0$ and $d_{\cP} (\mu_i, \mu) \xrightarrow[i \to \infty]{} 0$.
\end{enumerate}
\end{prop}

We refer to \cite{ambrosio}[1.80] for the first point of Proposition~\ref{thm:weakAndDistance} and \cite{Bogachev2007}[Section 8.3] for the second point, let us mention that considering $\frac{\mu_i}{\mu_i(X)}$ and $\frac{\mu}{\mu(X)}$ allows to work with probability measures, for which the second point is more commonly stated.
We now check that weak convergence implies convergence for $\eta_d$, which is all we need in this work.

\begin{lem}
Let $\mu$, $\nu$ be finite Radon measures. Then
\begin{equation}
\eta_d (\mu, \nu) \leq \left\lbrace \begin{array}{lcl}
\left( \Delta (\mu, \nu) \right)^\frac{1}{d+1} & \text{if} & \Delta(\mu,\nu) \leq 1 \\
\left( \Delta (\mu, \nu) \right)^\frac{1}{d} & \text{if} & \Delta(\mu,\nu) \geq 1
\end{array} \right. \: .
\end{equation}
In particular, if $(\mu_i)_{i \in \N}$ is a sequence of finite Radon measures weakly converging to $\mu$ then $\eta_d (\mu_i, \mu)$ tends to $0$ when $i \to \infty$.
\end{lem}

\begin{proof}
The proof is standard, we give it for the sake of clarity. Let $B = \overline{B(x,r)} \subset \R^n$ and let $\epsilon > 0$. We define $h_\epsilon : \R_+ \rightarrow [0,1]$ and $\phi_\epsilon : \R^n \rightarrow [0,1]$ by 
\[ h_\epsilon(t) =  \left\lbrace \begin{array}{lcl}
1 & \text{if} & 0 \leq t \leq r \\
0 & \text{if} & t \geq r + \epsilon \\
1 - \frac{t-r}{\epsilon} & \text{if} & r<t<r+\epsilon
\end{array}\right. 
\quad \text{and for } y \in \R^n, \, \phi_\epsilon (y) = h_\epsilon(|y-x|)
\]
so that $\phi_\epsilon$ is radial, $\| \phi_\epsilon \|_\infty \leq 1$ and $\phi_\epsilon$ is $\frac{1}{\epsilon}$--Lipschitz. We infer that
\[
\int \phi_\epsilon \: d\mu - \int \phi_\epsilon \: d \nu \leq \max \left( 1, \frac{1}{\epsilon} \right) \Delta(\mu,\nu) \: .
\]
Consequently,
\[
\mu(B) \leq \int \phi_\epsilon \: d \mu \leq \int \phi_\epsilon \: d \nu + \max \left( 1, \frac{1}{\epsilon} \right) \Delta(\mu,\nu) \leq \nu (B^\epsilon) + \max \left( 1, \frac{1}{\epsilon} \right) \Delta(\mu,\nu) \: .
\]
If $ \Delta(\mu,\nu) \geq 1$, we can take $\displaystyle \epsilon = \left( \Delta (\mu, \nu) \right)^\frac{1}{d} \geq 1$ and we obtain
$\displaystyle
\eta_d (\mu, \nu) \leq \left( \Delta (\mu, \nu) \right)^\frac{1}{d}
$. If $ \Delta(\mu,\nu) \leq 1$, we can take $\displaystyle \epsilon = \left( \Delta (\mu, \nu) \right)^\frac{1}{d+1} \leq 1$ and we obtain $\displaystyle
\eta_d (\mu, \nu) \leq \left( \Delta (\mu, \nu) \right)^\frac{1}{d+1}
$.\\
\noindent The last part of the statement follows from second point of Proposition~\ref{thm:weakAndDistance}.
\end{proof}

\subsection{A stability result}

Unfortunately, stability of the approximate mean curvature does not hold directly with respect to Prokhorov or flat distance but with respect to a combination  $\delta(\cdot, \cdot)$ of both that is not a distance, defined in \eqref{eq:deltaUnif} below.
Nevertheless, we prove in Theorem~\ref{thm:stability} $(i)$ that $\delta (V,V_i)$  tends to $0$ when the sequence of varifolds $(V_i)_i$ weak star converge to a $d$--regular varifold $V$.
We define for $V$, $W$ two $d$--varifolds in $\R^n$,
\begin{equation} \label{eq:deltaUnif}
\delta(V,W) = \sup \left\lbrace \left. \frac{\Delta_B (V, W)}{\left(\eta_d(\V, \|W\|) + \xdiam \: B / 2 \right)^d } \right| B \subset \R^n \text{ ball centered in } \supp \V \right\rbrace \: .
\end{equation}

\begin{prop} \label{prop:stability}
Let $V$ be a $d$--regular varifold in $\R^n$ of finite mass, let $x \in \supp \V$ and let $\epsilon \in ]0,1[$. Then for \emph{any} $d$--varifold $W$ and $z \in \R^n$ satisfying
\begin{equation} \label{eq:stabilityAsumption}
|x-z| + \eta_d (\V, \| W \|) \leq \gamma \epsilon \quad \text{with} \quad \gamma = \frac{1}{8 \left(1 + 2C_0^\frac{1}{d} + C_0^\frac{2}{d} \right)} \: ,
\end{equation}
where $C_0 \geq 1$ is the $d$--regularity constant from \eqref{eq:ahlfors}, we have
\begin{equation}
\left| H_{\epsilon}^\Pi (z, W) - H_{\epsilon}^\Pi (x,V) \right| \leq C \frac{\delta(V,W) + |x-z|}{\epsilon^2} \: ,
\end{equation}
where $C > 0$ only depends on $d$, $n$, $C_0$, $\xi$ and $\rho$.
\end{prop}

\begin{proof}
We shorten notations
$
\eta \coloneqq \eta_d (\V, \|W\|)$ and $\delta \coloneqq \delta (V,W)$. 
The proof is a consequence of the following estimates (see also \cite{BuetLeonardiMasnou} Lemma 4.4): with $B = B(x,\epsilon + |x-z|)$ and using that $\xi_\epsilon$ is $\frac{1}{\epsilon^n}\frac{\xLip(\xi)}{\epsilon}$--Lipschitz, we obtain
\begin{align}
\epsilon^n \Big| \| W \| \ast \xi_{\epsilon} (z) - \V \ast \xi_{\epsilon} (x) \Big| & \leq \frac{ \xLip( \xi)}{\epsilon} \left( \Delta_{B} (\| W \| , \V) + |x-z|  \V \left( B \right) \right) \nonumber \\
& \leq \frac{ \xLip( \xi)}{\epsilon} \left( \Delta_{B} ( W , V) + C_0 |x-z|  \left( \epsilon + |x-z| \right)^d \right) \nonumber \\
& \leq \frac{ \xLip( \xi)}{\epsilon} \left( \delta \left( \epsilon + |x-z| +\eta \right)^d  + C_0 |x-z|  \left( \epsilon + |x-z| \right)^d \right) \nonumber \\
& \leq C_0 \xLip( \xi) \frac{\delta + |x-z| }{\epsilon}    \left( \epsilon + |x-z| + \eta \right)^d  \nonumber \\ 
& \leq C_1  \frac{\delta + |x-z| }{\epsilon}  \epsilon^d \label{eq:massFlat} 
\end{align}
with $C_1 \coloneqq 2^d C_0 \xLip( \xi)$ using $\epsilon + |x-z| + \eta \leq 2 \epsilon$.
We repeat the same argument with $\Phi_x^\epsilon$ defined in \eqref{eq:deltaPi}, $\xLip(\Phi_{z}^{\epsilon}) \leq \epsilon^{-2} \xLip(\rho^\prime)$ and then
\begin{align}
\left| \int_{\R^n \times \G} \Phi_{z}^{\epsilon} \, d W - \int_{\R^n \times \G} \Phi_{x}^{\epsilon}  \, d V \right| & \leq
 \frac{\xLip(\rho^\prime)}{\epsilon^2} \left( \Delta_{B} (W,V) + |x-z| \V \left( B \right) \right) \nonumber \\
& \leq C_2 \frac{\delta + |x-z| }{\epsilon^2}  \epsilon^d \quad \text{with} \quad C_2 \coloneqq 2^d C_0 \xLip( \rho^\prime) \label{eq:deltaPiFlat}
\end{align}
Last but not least, we estimate $\epsilon^n \| W \| \ast \xi_{\epsilon} (z)$ from below thanks to the mass $\V$ of rings of radii comparable to $\epsilon$. Thanks to the $d$--regularity asumption \eqref{eq:ahlfors}, it is possible to choose the ratio of radii so that the $\V$--mass of the ring is comparable to $\epsilon^d$, and thus comparable to the $\V$--mass of a ball of radius $\epsilon$.
Let us denote $\beta = \displaystyle \min \left\lbrace  \xi(s) \: \left| \: s \in \left[ \tfrac{C_0^{-2/d}}{4} , \tfrac{1}{2} \right] \right. \right\rbrace > 0$. We then have, using \eqref{eq:ahlfors} and the definition of $\eta$:
\begin{align*}
&\epsilon^n \| W \| \ast \xi_{\epsilon} (z)  \geq \beta \left( \| W \| ( B(z, \tfrac{\epsilon}{2})) - \|W\| (B(z,  C_0^{-2/d} \tfrac{\epsilon}{4}) \right) \\
& \qquad \geq \beta \left( \left[ \V (B(z, \tfrac{\epsilon}{2} - \eta)) - \eta^{d} \right] - \left[ \V (B(z,  C_0^{-2/d}\tfrac{\epsilon}{4} + \eta)) + \eta^{d} \right]  \right) \nonumber \\
& \qquad \geq \beta \left(  \V (B(x, \tfrac{\epsilon}{2} - (\eta + |x-z|))) - \V (B(x,  C_0^{-2/d} \tfrac{\epsilon}{4} + (\eta + |x-z|))) - 2\eta^{d} \right) \nonumber \\
& \qquad\geq  \beta \left( C_0^{-1} 2^{-d} (\epsilon - 2 (\eta + |x-z|))^d - C_0 2^{-d} \left(C_0^{-2/d}\frac{\epsilon}{2} + 2 (\eta + |x-z|) \right)^d - 2\eta^{d} \right)
\nonumber \\
& \qquad \geq \beta C_0^{-1} 2^{-d} \underbrace{ \left( (\epsilon - 2 (\eta + |x-z|))^d - (\frac{\epsilon}{2} + 2 C_0^{2/d} (\eta + |x-z|))^d - (2 (2C_0)^{1/d} \eta)^d \right)}_{=:A}  \: .
\end{align*}
 Using once again that $a^d + b^d \leq (a+b)^d$ for $a,b \geq 0$ we estimate $A$ as follows:
\begin{equation*}
A \geq (\epsilon - 2 (\eta + |x-z|))^d  - \left(\frac{\epsilon}{2} + 2 C_0^{2/d} (\eta + |x-z|) + 2 (2C_0)^{1/d} \eta \right)^d
\end{equation*}
and then using that for $a \geq b \geq 0$, $a^d - b^d \geq (a-b)a^{d-1}$ with $a = \epsilon - 2 (\eta + |x-z|)$ and $b = \frac{\epsilon}{2} + 2 C_0^{2/d} (\eta + |x-z|) + 2 (2C_0)^{1/d} \eta $, we get
\[
A \geq  
\underbrace{\frac{1}{2}\left( \epsilon - 4(1 + (2C_0)^{1/d} + C_0^{2/d}) (\eta + |x-z|) \right) \left(\epsilon - 2 (\eta + |x-z| ) \right)^{d-1} }_{\geq \frac{1}{2}\left( \frac{\epsilon}{2} \right)^d \text{ by asumption in \eqref{eq:stabilityAsumption}}} \: .
\]
We conclude that 
\begin{equation} \label{eq:massBound}
\epsilon^n \| W \| \ast \xi_{\epsilon} (z) \geq C_3^{-1} \epsilon^d \quad \text{with} \quad C_3 = \beta^{-1} C_0 2^{2d+1}
\end{equation}
We similarly have
\begin{equation} \label{eq:massBound2}
\epsilon^n \| V \| \ast \xi_{\epsilon} (x) \geq 
C_3^{-1} \epsilon^d \: .
\end{equation}
Eventually, for $\Pi$ as in \eqref{eq:projMeanCurvature} and using $\| \Pi \| \leq 2$ and  \eqref{eq:ahlfors},
\begin{align}
\left| \int \Phi_x^{\epsilon} dV \right| & \leq \int_{B(x,\epsilon) \times \G} \frac{1}{\epsilon} \left| \rho^\prime \left( \frac{|y-x|}{\epsilon} \right) \frac{\Pi (y-x)}{|y-x|}  \right| \: dV(y,S) \leq \frac{2}{\epsilon} \| \rho^\prime \|_\infty \V (B(x, \epsilon)) \nonumber \\
& \leq  \frac{2}{\epsilon} \| \rho^\prime \|_\infty C_0 \epsilon^d \leq C_4 \epsilon^{d-1} \quad \text{with } C_4 \coloneqq 2 C_0 \| \rho^\prime \|_\infty \label{eq:PiBound}
\end{align}
We combine \eqref{eq:massFlat}, \eqref{eq:deltaPiFlat}, \eqref{eq:massBound}, \eqref{eq:massBound2} and \eqref{eq:PiBound} so that
\begin{align*}
\left| H_{\epsilon}^\Pi (z, W) - H_{\epsilon}^\Pi (x,V) \right| & = \left| \frac{\int \Phi_{z}^{\epsilon}  \, d W}{\epsilon^n \| W \| \ast \xi_{\epsilon} (z)} - \frac{\int \Phi_{x}^{\epsilon}  \, d V}{\epsilon^n \| V \| \ast \xi_{\epsilon} (x)} \right| \\
& \leq \frac{1}{\epsilon^n \| W \| \ast \xi_{\epsilon} (z)} \Bigg( \left| \int \Phi_{z}^{\epsilon} \, d W - \int \Phi_{x}^{\epsilon}  \, d V \right| \\
& \qquad\qquad\qquad\qquad +  \frac{\displaystyle \left| \int \Phi_x^{\epsilon} dV \right|}{\displaystyle  \epsilon^n \| V \| \ast \xi_{\epsilon} (x)} \epsilon^n \Big| \| W \| \ast \xi_{\epsilon} (z) - \V \ast \xi_{\epsilon} (x) \Big| \Bigg) \\
& \leq C_3 \epsilon^{-d} \left( C_2 \frac{\delta + |x-z| }{\epsilon^2}  \epsilon^d + C_3 \epsilon^{-d} C_4 \epsilon^{d-1}  C_1 \frac{\delta + |x-z| }{\epsilon} \epsilon^d \right) \\
& \leq C_3\left(C_2 + C_1 C_3 C_4 \right) \frac{\delta + |x-z| }{\epsilon^2} \: ,
\end{align*}
and $C = C_3\left(C_2 + C_1 C_3 C_4 \right) \geq 1$ is then a constant depending on $d$, $n$, $C_0$, $\rho$ and $\xi$ and in particular uniform w.r.t. $x$.
\end{proof}

In order to establish convergence of approximate mean curvature under weak star convergence of varifolds, it remains to prove that $\delta (V,V_i)$ tends to $0$ when $V_i$ weak star converges to $V$, which is the key point of the end of the section.

\subsection{Convergence of approximate mean curvature}

Let us transfer the measure setting introduced in Section~\ref{subsec:distances} to our varifolds framework $X = \R^n \times \G$. For a $d$--varifold $V$, the total variation is $V(X) = V(\R^n \times \G) = \V(\R^n)$. Let $(V_i)_{i \in \N}$, $V$ be $d$--varifolds such that $(V_i)_i$ weak star converges to $V$, with $ \V(\R^n) < + \infty$ and with support contained in a fixed compact $K \times \G \subset \R^n \times \G$. Then, $\|V_i\|$ weak star converges to $\V$ and moreover $\| V_i \| (\R^n) \rightarrow \V(\R^n)$. Indeed, $\V(\R^n) \leq \liminf_{i \to \infty} \| V_i \| (\R^n)$ by lower semi continuity of total variation. And in addition,
as all $V_i$ are supported in the same compact set $K \times \G$, we have that 
\[ \limsup_{i \to \infty} \| V_i \| (\R^n)  = \limsup_{i \to \infty} V_i (K \times \G) \leq V (K \times \G) = \V (\R^n) \: ,
\]
hence $\lim_{i \to \infty} \| V_i \| (\R^n) = \V(\R^n) < +\infty$.
Consequenlty, thanks to Proposition~\ref{thm:weakAndDistance}, $(V_i)_i$ (resp. $(\|V_i\|)_i$) weak converges to $V$ (resp. $\V$) and both
\begin{equation} \label{eq:flatProkhorov}
\sup_{B \subset \R^n \text{ ball}} \Delta_B (V, V_i) \leq \Delta (V, V_i)  \xrightarrow[i\infty]{} 0  \quad \text{and} \quad
\eta_d (\|V_i\|, \V)  \xrightarrow[i\infty]{} 0 \text{  hold.}
\end{equation}

\begin{thm} \label{thm:stability}
Let $V$ be a $d$--regular varifold in $\R^n$ with regularity constant $C_0 \geq 1$ in \eqref{eq:ahlfors} and assume $\V (\R^n) < \infty$.
Let $(V_i)_i$ be a sequence of $d$--varifolds weak star converging to $V$. Assume that there exists a compact set $K \subset \R^n$ such that $\supp V_i \subset K \times \G$ for all $i$. Then, 
\begin{enumerate}[$(i)$]
\item setting $d_i \coloneqq \delta (V, V_i)$ ($\delta$ has been defined in \eqref{eq:deltaUnif}) we have
\begin{equation} \label{eq:di}
d_i \rightarrow 0 \quad \text{for } i \to \infty \: ;
\end{equation}
\item setting $\eta_i \coloneqq \eta_d (\V, \|V_i\|)$, there exists a constant $C > 0$ depending only on $d$, $n$, $C_0$, $\xi$ and $\rho$ such that: if $x \in \supp \V$ and $(z_i)_i \subset \R^n$ converges to $x$, then for any sequence $(\epsilon_i)_i \subset ]0,1[$ tending to $0$ and satisfying
\begin{equation} \label{eq:gamma}
\displaystyle |x-z_i| + \eta_i \leq \gamma \epsilon_i \: , \quad \text{for} \quad \gamma = \left(8 (1 + (2C_0)^{1/d} + C_0^{2/d} ) \right)^{-1} 
\end{equation}
we have 
\begin{equation} \label{eq:stabilityAhlforsCond}
\left| H_{\epsilon_i}^\Pi (z_i, V_i) - H_{\epsilon_i}^\Pi (x,V) \right| \leq C \frac{d_i + |x-z_i|}{\epsilon_i^2} \: .
\end{equation}
\end{enumerate}
\end{thm}

Note that $(ii)$ directly follows from Proposition~\ref{prop:stability}, and $(i)$ again
holds under the flexible asumption that $V$ is $d$--regular in the sense of \eqref{eq:ahlfors}. 
\begin{proof}
{\it Proof of $(i)$}: We first show that $d_i \rightarrow 0$.\\
As varifolds are supported in the same compact set $K \times \G$, we have \eqref{eq:flatProkhorov}, i.e.
\begin{enumerate}[$-$]
\item $(V_i)_i$ weak converges to $V$ and the flat distance $\Delta (V,V_i)$ tends to $0$, thus \[ \sup_{B \subset \R^n \text{ ball}} \Delta_B (V, V_i) \leq \Delta (V,V_i) \xrightarrow[i\to \infty]{} 0 \: ,
\]
\item $(\|V_i\|)_i$ weak converge to $\V$ and both $\Delta (\|V_i\|, \V)  \xrightarrow[i\to \infty]{} 0$ and
$\displaystyle \eta_i \coloneqq  \eta_d (\|V_i\|, \V)  \xrightarrow[i\to \infty]{} 0$.
\end{enumerate}
Let us argue by contradiction and, up to extraction of a subsequence, assume that $\exists \bar\delta > 0$ such that $\forall i \in \N$, $d_i > \bar\delta$. By definition of $d_i$, there is a sequence of balls $(B_i)_i \subset \R^n$, $B_i = B(x_i,r_i)$ with $x_i \in \supp \V$ and $r_i > 0$, such that for all $i$, \[ \Delta_{B_i} (V,V_i) > \bar\delta  (r_i + \eta_i)^d  \: . \]
And we already know that $\Delta_{B_i} (V_i, V) \leq \Delta (V,V_i) \xrightarrow[i\to\infty]{} 0$, thus $ (r_i + \eta_i)^d $ must tend to $0$ as well, and $\eta_i$ already tends to $0$ so that we eventually conclude that $r_i \xrightarrow[i\to  \infty]{} 0$.\\
Now, from the definition of $\Delta_{B_i}$, $\Delta_{B_i} (V,V_i) > \bar\delta  (r_i + \eta_i)^d $ implies that there exists a sequence of $1$--Lipschitz functions $(\phi_i)_i \in \xC(\R^n \times \G)$ with $\supp \phi_i \subset B_i$, $\| \phi_i \|_{\infty} \leq 1$ and such that for all $i$,
\begin{equation} \label{eq:stability1}
\left| \int_{B_i \times \G} \phi_i \: dV_i - \int_{B_i \times \G} \phi_i \: dV \right| > \bar\delta  (r_i + \eta_i)^d \: .
\end{equation}
Applying Ascoli compactness theorem in $\xC (K \times \G)$, up to extracting a subsequence, there exists a continuous function $\phi \in \xC (K \times \G)$ such that
$\phi_{i} \xrightarrow[i \to \infty]{} \phi$ uniformly in $K \times \G$.\\
It is not difficult to see that $\phi = 0$. Indeed, let us consider 
\[ 
X = \{ y \in K \: : \: |\phi(y)| > 0 \} \: .
\]
Let $x \in X$, then $\exists N = N_x \in \N$ such that for all $i \geq N$, $| \phi_{i} (x) | > 0$ and thus $x \in B_{i}$. Therefore, 
 $| x - x_{i} | \leq r_i \xrightarrow[i \to \infty]{} 0$ and thus, $(x_{i})_i$ converges to $x$. Consequently $X$ contains at most one point and on the other hand $X$ is open by continuity of $\phi$ so that $X = \emptyset$.\\
Coming back to \eqref{eq:stability1} we first get by definition of $\eta_i$ that $\| V_{i} \|  (B_{i}) \leq \V  (B_{i}^{\eta_i}) + (\eta_i)^d $ and then by \eqref{eq:ahlfors}
\begin{align*}
\bar\delta  (r_{i} + \eta_i)^d  & < \sup_{K \times \G} | \phi_{i} | \left( \| V_{i} \|  (B_{i}) + \V(B_{i}) \right)  \leq \left(  C_0 (r_i + \eta_i)^d + (\eta_i)^d + C_0 (r_i)^d \right) \sup_{K \times \G} |\phi_{i}| \\
& \leq  2C_0 (r_i + \eta_i)^d \sup_{K \times \G} | \phi_{i} | \: .
\end{align*}
It follows that,
$\displaystyle
0 < \bar\delta < \frac{2 C_0 (r_i + \eta_i)^d }{(r_i + \eta_i)^d } \sup_{K \times \G} | \phi_{i} | \leq 2 C_0 \sup_{K \times \G} | \phi_{i} |
$
leading to a contradiction since $\sup_{K \times \G} | \phi_{i} | \xrightarrow[i \to \infty]{} \sup_{K \times \G} | \phi | = 0$.

{\it Proof of $(ii)$}: apply Proposition~\ref{prop:stability} with $W = V_i$.
\end{proof}

Eventually combining consistency (Proposition~\ref{prop:consistency}) and stability (Theorem~\ref{thm:stability}) we obtain the convergence of the approximate mean curvature. The following result (Corollary~\ref{cor:CV}) is a particular case where strong regularity of the limit varifold ensures that asumptions of both consistency and stability are fullfilled. However, the $\xC^3$ regularity asumption is stronger than necesary and more general results essentially require to check that Proposition~\ref{prop:consistency} and stability Theorem~\ref{thm:stability} apply.

\begin{cor}[Convergence] \label{cor:CV}
Let $V$ be a $d$--varifold associated with a compact $d$--submanifold $M \subset \R^n$ without boundary of class 
 $\xC^3$. 
Let $(V_i)_i$ be a sequence of $d$--varifolds weak star converging to $V$. Assume that there exists a compact set $K \subset \R^n$ such that $\supp V_i \subset K \times \G$ for all $i$. Then, define $d_i = \delta (V,V_i)$ (as in \eqref{eq:deltaUnif})  and
let $x \in M$, $(z_i)_i \subset \R^n$ such that $| x- z_i| \xrightarrow[]{} 0$, for any sequence $(\epsilon_i)_i \subset (0,1)$ tending to $0$ and satisfying $|x-z_i| + \eta_d( \V, \|V_i\|) = o(\epsilon_i)$ we have
\begin{align*}
\left| H_{\epsilon_i}^\Pi (z_i, V_i) - H(x,V) \right| & = O \left( \frac{d_i + |x-z_i|}{\epsilon_i^2} \right) + O(\epsilon_i) \nonumber \\
& \xrightarrow[i \to \infty]{} 0 
\quad \text{as soon as } \sqrt{d_i + |x-z_i|} = o(\epsilon_i) \: .
\end{align*} 
\end{cor}

\begin{proof}
As $M$ is assumed to be compact without boundary and of class $\xC^2$, then $\exists C_0 \geq 1$ such that for all $x \in M$ and $0 < r \leq \xdiam \: M$,
\begin{equation*}
C_0^{-1} r^d \leq  \V(B(x,r)) = \cH^d (M \cap B(x,r)) \leq C_0 r^d \: .
\end{equation*}
Applying Theorem~\ref{thm:stability} and Proposition~\ref{prop:consistency} to
\[
\left| H_{\epsilon_i}^\Pi (z_i, V_i) - H(x,V) \right| \leq \left| H_{\epsilon_i}^\Pi (z_i, V_i) - H_{\epsilon_i}^\Pi (x,V) \right| + \left| H_{\epsilon_i}^\Pi (x, V) - H(x,V) \right|
\]
concludes the proof.
\end{proof}

In the case of a point cloud varifold $V = \sum_{i=1}^N m_i \delta_{(x_i, P_i)}$, and for
\begin{equation} \label{eq:Pi_ij}
\Pi = \Pi_{ij} \in \{ \Pi_{P_j}, \, -2 \Pi_{P_j^\perp}, \, 2 \,{\rm Id},  \, \Pi_{P_i^\perp} \circ \Pi_{P_j}, \, -2 \Pi_{P_i^\perp} \circ \Pi_{P_j^\perp}, \, 2 \Pi_{P_j^\perp} \}
\end{equation}
we rewrite the approximate mean curvature
\begin{equation} \label{meanCurvatureEquation}
H_\epsilon^\Pi (x_i, V)  = - \frac{d}{n} \frac{1}{\epsilon} \frac{ \displaystyle \sum_{j=1}^N m_j \rho^\prime \left(\frac{|x_j-x_i|}{\epsilon} \right) \Pi_{ij} \left( \frac{x_j-x_i}{|x_j-x_i|}  \right) }{\displaystyle \sum_{j=1}^N m_j \xi \left( \frac{|x_j-x_i|}{\epsilon} \right)} \: .
\end{equation}
Proposition~\ref{prop:consistency} leaves us with at least $6$ possible choices for the definition of an approximate mean curvature, 
more or less equivalently reasonable in the continuous smooth case. 
Our numerical experiments in Section~\ref{sec:results} indicate that those formulas can behave very differently when used in the context of a time 
discretization for the simulation of mean curvature flows.

\begin{rem}[$k$--nearest neighbours] \label{rk:knn}
Notice that in the use of formula \eqref{meanCurvatureEquation}, it is also possible to prescribe the number $k$ of nearest points to be considered and infer $\epsilon$ so that the support of $\rho \left( \frac{|\cdot \,-\,x_i|}{\epsilon} \right)$ exactly contains $x_i$ plus $k$ other points. It is more convenient to fix the number of nearest points in numerical simulations, especially for point clouds that are not uniformly sampled.
\end{rem}

\section{Mean curvature motion and comparison principles} \label{sec:timecont}

\subsection{Computation of masses and directions for point cloud varifolds} \label{sec:masses}
Before we discuss the actual time discretization let us detail how we derive masses and directions from positions.
In the case of a smooth varifold $V = \cH^d_{| M} \otimes \delta_{T_x M}$ associated with a $d$--submanifold $M \subset \R^n$, the tangent plane and the varifold are completely determined by 
the knowledge of $M$. In the case of a point cloud varifold, there is no unique choice of masses $\{ m_i \}_{i = 1}^N$ and sets of directions $\{ P_i \}_{i=1}^N$. Here, we will follow standard approaches.
The ansatz to define the masses $m_i$ is as follows. We assume that the positions $\{x_i\}_{i}$ are close to some $d$--submanifold $M$ and we want to define masses $\{ m_i \}_i$ such 
that the resulting Radon measures $\mu \coloneqq \sum_{i=1}^N m_i \delta_{x_i}$ and $\nu \coloneqq \cH^d_{| M}$ are close  in the sense of measures. To this end, we regularize $\mu$ and $\nu$ via convolution.  I.e.~we define 
$\lambda_\delta(x) = \delta^{-n} \lambda(|x| / \delta)$ for $\lambda: \R \to \R$ 
nonnegative, even, and compactly supported in $[-1,1]$. Then we renormalize the regularized $\mu$ by the regularized $\nu$ in the sense that  
\[
m_i = \frac{\nu \ast \lambda_\delta (x_i)}{\mu \ast \lambda_\delta (x_i)} 
\]
Unfortunately, $M$ and thus $\nu$ are not known and we replace $\delta^n \nu \ast \lambda_\delta (x_i)$ with its first order approximation $C_\lambda \delta^d$  with 
$C_\lambda = \int_{\R^d}  \lambda(|y|) \,dy =  \sigma_{d-1} \int_{s=0}^1 \lambda (s) s^{d-1} \, ds$ being the volume weighted by $\lambda$ of the unit ball in $\R^d$.
From this we deduce for the masses
\[
m_i 
= \frac{C_\lambda \delta^d}{\displaystyle \sum_{j=1}^N \lambda \left( \frac{|x_j - x_i |}{\delta} \right)} \: .
\]
To the best of our knowledge, there is unfortunately no general result of convergence of such estimators, 
assuming for instance a control of the Hausdorff distance between $M$ and $\{x_i\}_i$ and asking for $\{ m_i \}_i$ ensuring that $\mu$ and $\nu$ are close in flat distance or  
in Wasserstein type distance. 
In our numerical experiments, we will consider either  $\lambda$ smooth and compactly supported or $\lambda = \chi_{]-1,1[}$ which implies
\[
m_i = \frac{\omega_d \delta^d}{k_\delta} \quad \text{with} \quad  k_\delta = {\rm{card}} \left\lbrace j \in \{ 1, \ldots, N \} \: : \: |x_j -x_i | < \delta \right\rbrace \: .
\]
As it is usually done, sets of directions $\{P_i\}_i$ are computed through a local weighted linear regression. 
We fix a further nonnegative and even profile kernel $\zeta : \R \rightarrow \R $ supported in $]-1,1[$ and a sufficiently large parameter $\sigma > 0$ and define, based on a $\sigma$--neighbourhood of some point $x_i$ containing $k_\sigma$ points,
a center of mass
\[
\bar{x} = \frac{1}{k_\sigma} \sum_{j=1}^N \chi_{]-1,1[} \left( \frac{|x_j - x_i|}{\sigma} \right) x_j
\]
of the points $\sigma$--close to the point $x_i$.
Furthermore, with the notation $x = ( x^{(1)}, \ldots, x^{(n)} )$ for the $n$--components of $x\in \R^n$, we compute the $n$ by $n$ covariance matrix 
$M^i=(M^i_{kl})_{k,l=1,\ldots n}$ of coefficient $(k,l)$:
\[
M^i_{kl}  = 
\sum_{j=1}^N \zeta \left( \frac{|x_j - x_i|}{\sigma}\right) 
 \left( x_j^{(k)} - {\bar x}^{(k)} \right) \left( x_j^{(l)} - {\bar x}^{(l)} \right) \: .
\]
The matrix $M^i$ is symmetric and positive semi-definite.
The $d$ eigenvectors associated with the $d$ highest eigenvalues provide a basis of an approximate tangent space $P_i$
and the $(n-d)$ eigenvectors corresponding to the $(n-d)$ smallest eigenvalues provide a basis of an approximate normal space $P_i^\perp$.
When using a smooth profile $\zeta$, this way of computing tangent plane ensures its spatial regularity. 
In our numerical experiments, we have chosen $\zeta = \xi$ smooth and compactly supported in $]-1,1[$.
As pointed out in Remark~\ref{rk:knn} concerning $\epsilon$, for practical reasons
it is advisable to fix $(k_\epsilon, k_\sigma, k_\delta)$ and to define $(\epsilon,\sigma,\delta)$ accordingly (as the radius of the smallest ball containing the right number of nearest points), in this case $(\epsilon,\sigma,\delta)$ vary, depending on the sampling of the point cloud. This is what we finally have used in the applications. 

\subsection{Time continuous mean curvature motion and comparison principles} \label{sec:timecontsub}

Now, we are in the position to formulate mean curvature motion for point cloud varifolds.
More precisely, 
given a point cloud $d$--varifold $ V = \sum_{i=1\,\ldots,N} m_i  \delta_{(x_i, P_i)}$ 
in $\R^n$, we consider the following system of ordinary differential equations:
\medskip 

\noindent
Find a family of varifolds $(V(t))_{t\geq 0}$ with 
\[
V(t) = \sum_{i=1}^N m_i  \delta_{(x_i(t), P_i(X(t)))} \quad \text{and} \quad 
X(t) = (x_1(t), \ldots, x_N(t)) \in \R^{nN}
\]
such that 
\begin{equation} \label{eq:continuousFormulation}  
\displaystyle \tfrac{\mathrm{d}}{\mathrm{d}t} x_i(t)  =  \displaystyle H_\epsilon^\Pi (x_i(t), V(t)) 
\end{equation}
for prescribed initial data $V(0)  =  V$ and $i = 1 \ldots N$.
Here,  
$m_i( X(t))$ and $P_i(X(t))$ are computed from the positions 
as functions of neighbouring positions (see the beginning of the current section) while $H_\epsilon^\Pi$ is defined in \eqref{meanCurvatureEquation}. 
Thus, the evolution equation turns into
\begin{equation} \label{meanCurvatureEvolution}
\tfrac{\mathrm{d}}{\mathrm{d}t} x_i(t) = \frac{1}{\epsilon} \sum_{j=1}^N \omega_{ij}(t) \Pi_{ij}(t) \left( x_j(t) - x_i(t) \right), \quad i = 1 \ldots N \: .
\end{equation}
with 
 \[
\omega_{ij}(t) = - \frac{d}{n} \frac{ \displaystyle m_j(t) \rho^\prime \left(\frac{|x_j(t) - x_i(t) |}{\epsilon} \right) \frac{1}{| x_j(t) - x_i(t) |} }{\displaystyle \sum_{l=1}^N m_l(t) \xi \left( \frac{|x_l(t) - x_i(t) |}{\epsilon} \right)}
\]
for $i \neq j$ and $\omega_{ii}^{k} = 0$, where $m_i(t) = m_i(X(t))$ denotes the masses at time $t$. We observe that $\omega_{ij}(t) \geq 0$ since $\rho$ is nonincreasing in $[0,1]$ and $\xi$ is positive in $]0,1[$. 
As a first consequence of this rewritten evolution problem we obtain the following comparison result which establishes planar barriers for the flow.

\begin{prop} [planar barrier] \label{PlaneBarrier} 
Let $(X(t))_{0 \leq t < T}$ be a family of point clouds evolving according to the flow defined in \eqref{meanCurvatureEvolution} up to some time $T \in ]0,+\infty]$. Suppose that 
\begin{enumerate}[$(i)$]
\item the initial point cloud $X(0) = \{ x_i(0) \}_{i=1}^N \subset \R^n$ fulfills $x_i(0) \cdot \nu \leq \mu$ for $i = 1 \ldots N$ with $\nu\in \R^n$, $\mu \in \R$, 
\item  for all $0 \leq t < T$, if $x_i(t)$ is a point on the boundary of the convex hull of $X(t)$ then for all points $x_j(t)$ such that $|x_j(t) - x_i(t)| < \epsilon$, the vector $\Pi_{ij}(t) (x_j(t) - x_i(t))$ at $x_i(t)$ is pointing inside the convex hull of $X(t)$.
\end{enumerate} 
Then independently of the choices of masses $m_i(t)$ anf for all $0 \leq t < T$,
\[
x_i(t) \cdot \nu \leq \mu , \quad \text{for all } i = 1 \ldots N \: .
\]
\end{prop} 

\begin{proof}
Assume that the point cloud $X(t)$ is touching the plane $\{x\in \R^n\,|\, x \cdot \nu = \mu\}$ at time $t\geq 0$.
Consider any $x_i(t)$ with  $x_i(t) \cdot \nu = \mu$. Then, our assumption ensures that $\tfrac{\mathrm{d}}{\mathrm{d}t} x_i(t) \cdot \nu \leq 0$
and $X(t)$ will not penetrate the plane.
\end{proof}
For $\Pi_{ij}= 2 \,{\rm Id}$ the assumption in the proposition is obviously fulfilled and for one of the choices 
$\Pi_{ij} \in \{-2 \Pi_{P_j^\perp}, \, -2 \Pi_{P_i^\perp} \circ \Pi_{P_j^\perp}, \, 2 \Pi_{P_j^\perp}\}$ this assumption appears to be a useful constraint to define 
$\{P_i\}_{i=1,\ldots N}$. For 
$\Pi_{ij} \in \{ \Pi_{P_j},  \, \Pi_{P_i^\perp} \circ \Pi_{P_j} \}$ the assumption will fail in general for most practical choices of the $P_i$. In fact, in this case
the verification of a planar barrier depends in a subtle way on the weights $\omega_{ij}$.

Next, let us consider a spherical comparison principle. A sphere of radius $R^0$ around a center point $x\in \R^n$ stay spherical 
under mean curvature motion with radius $R(t) = \sqrt{(R^0)^2-2dt}$ and surfaces inside the initial sphere stay inside the evolving spheres until 
they become singular, e.g. at the extinction time (see \cite{Ecker2004} Proposition 3.1 and Remark 4.10 for a measure theoretic version due to Brakke). 
In what follows, we will here study a discrete counter part of this comparison principle 
in case of point cloud varifolds.

\begin{prop} [sphere comparison principle] \label{SphereInclusion} 
For point cloud $X^0 = \{ x_i^0 \}_{i=1}^N \subset \R^n$ and $z\in \R^n$ define $R^0 = \max_{i=1,\ldots,N} |x_i^0 -z|$ and assume that 
$(X(t))_{0 \leq t <T}$ satisfies \eqref{eq:continuousFormulation} with $X(0) = X^0$. Then $R(t) = \max_{i=1,\ldots,N} |x_i(t) -z|$ fulfills
\[
R(t) \leq \sqrt{(R^0)^2 - 2d \int_0^t c(s) \mathrm{d}s}
\]
with 
\begin{equation} \label{eq:Pt}
c(t) = \min \left\lbrace \left. \frac{ \Pi_{ij}(t) (x_i(t) - x_j(t)) \cdot (x_i(t) - z) }{| x_i(t) - x_j(t) |^2  } \: \right| \: \begin{array}{l}
i \in \{1, \ldots, N \}, \: \displaystyle |x_i(t) - z| = R(t)  \text{ and } \\
j \in \{1, \ldots, N \}, \: \displaystyle 0< | x_i(t) - x_j(t) | <  \epsilon 
\end{array}   \right\rbrace \: .
\end{equation}
\end{prop}  
\begin{proof}
Without any restriction assume $z=0$ and choose any $x_i(t)$ with $R(t) =  |x_i(t)|$. Multiplying \eqref{meanCurvatureEvolution} with $x_i(t)$ 
we obtain 
\begin{align*}
\tfrac12 \tfrac{\mathrm{d}}{\mathrm{d}t} |x_i(t)|^2 &= \tfrac{\mathrm{d}}{\mathrm{d}t} x_i(t) \cdot x_i(t) 
= \frac{1}{\epsilon} \sum_{j=1}^N \omega_{ij}(t) \Pi_{ij}(t) \left( x_j(t) - x_i(t) \right) \cdot x_i(t) \\
&\leq -\frac{c(t)}{\epsilon} \sum_{j=1}^N \omega_{ij}(t) |x_i(t)-x_j(t) |^2 \\
&=  \frac{c(t)d}{n} \sum_{j=1}^N \frac{ \displaystyle m_j(t) \rho^\prime \left(\frac{|x_j(t) - x_i(t) |}{\epsilon} \right) \frac{| x_j(t) - x_i(t) |}{\epsilon} }{\displaystyle \sum_{l=1}^N m_l(t) \xi \left( \frac{|x_l(t) - x_i(t) |}{\epsilon} \right)} 
\leq - c(t) d\:,
\end{align*} 
where we have used $n \xi(s) = - s \rho^\prime (s)$ from \eqref{eq:kernelPairs}. Thus, we obtain $R(t)^2 \leq (R^0)^2 - 2d \int_0^t c(s) \mathrm{d}s$, which 
proves the claim.
\end{proof}

For $\Pi_{ij} = 2 \,{\rm Id}$, the constant in \eqref{eq:P} is $c(t) = 1$ for all $t$ and thus the conclusion of Proposition~\ref{discreteSphereInclusion} 
recovers the classical spherical comparison principle $R(t) \leq \sqrt{(R^0)^2 - 2d t}$. 
Indeed, dropping the dependence on time and fixing $i$ and $j$ according to \eqref{eq:Pt} and so that $|x_j - z| \leq |x_i - z| = R$ we then obtain
\[
(x_i-x_j) \cdot (x_i-z) = |x_i-z|^2 - (x_j-z) \cdot (x_i-z) \geq \tfrac12 |x_i-z|^2 - (x_j-z) \cdot (x_i-z) + \tfrac12 |x_j-z|^2 = \tfrac12 |x_i-x_j|^2 \: .
\] 
Unfortunately, regarding to other choices for $\Pi_{ij}$, $c(t)$ strongly depends on the computation of normals and could even be negative.

\subsection{Time discrete curvature flow} \label{sec:time}
Let us now proceed with a time discretization of \eqref{eq:continuousFormulation}.
It is well--known, for parametric mean curvature flows for instance, that explicit time discretization implies a restrictive stability condition, imposing small time steps. A common solution is then to consider implicit or partially implicit time discretizations. Let us consider a time step $\tau >0$ and approximate 
the solution $x_i(t_k)$ at a discrete time $t_k = k \tau$ using the
adopted notation
\begin{align*}
& m_i^k = m_i(X^k), \quad P_i^k = P_i(X^k), \quad \Pi_{ij}^k = \Pi_{ij} (P_i^k, P_j^k) \text{ as in \eqref{eq:Pi_ij}}, \\
&X^k  = (x_i^k)_{i=  1 \ldots N} \in \R^{nN}
\quad \text{and} \quad V^k = \displaystyle \sum_{i=1}^N m_i^k \delta_{(x_{i}^k, P_i^k)} \: .
\end{align*}
A first natural choice is the implicit scheme (actually implicit with respect to positions $X^k$ but explicit with respect to masses and sets of directions):
\begin{align}\label{eqSemiImplicitScheme1}
x_i^{k+1} &  =   \displaystyle x_i^k + \tau H_\e^\Pi ( x_i^{k+1} , \widehat{V^{k}} ) \quad \text{with}
\quad \widehat{V^{k}}  =  \displaystyle \sum_{i=1}^N m_i^k \delta_{(x_i^{\textcolor{ColorImplicit}{k+1}} , P_i^k)},  \\
\label{eqSemiImplicitScheme2}
V^{k+1} &= \sum_{i=1}^N m_i^{\textcolor{ColorImplicit}{k+1}} \delta_{(x_i^{k+1} , P_i^{\textcolor{ColorImplicit}{k+1}})},
\end{align}
where in \eqref{eqSemiImplicitScheme2} the updated masses $m_i^{k+1}$ and tangent directions $P_i^{k+1}$ are 
computed from the new positions $X^{k+1}$.
In other words, the positions $X^{k+1} = (x_i^{k+1})_{i=1 \ldots N} \in \R^{nN}$ must satisfy the following equations for $i = 1 \ldots N$
\begin{equation} \label{eq:implicitOmega}
x_i^{k+1} = x_i^k + \frac{\tau}{\epsilon} \sum_{j=1}^N \omega_{ij}^{k+1} \Pi_{ij}^{k} \left( x_j^{k+1} - x_i^{k+1} \right)
\end{equation}
with  
\[
\omega_{ij}^{k} = - \frac{d}{n} \frac{ \displaystyle m_j^k \rho^\prime \left(\frac{|x_j^{k} - x_i^{k} |}{\epsilon} \right) \frac{1}{| x_j^{k} - x_i^{k} |} }{\displaystyle \sum_{l=1}^N m_l^k \xi \left( \frac{|x_l^{k} - x_i^{k} |}{\epsilon} \right)} 
\]
for $i \neq j$, and $\omega_{ii}^{k} = 0$. As for the continuous counterpart $\omega_{ij}^k \geq 0$. 
At first, we obtain as in the time continuous case a comparison principle with planar barriers for the flow.
\begin{prop} [planar barrier in the time discrete case] \label{discretePlaneBarrier} 
Suppose that a sequence of point cloud varifolds $(V_k)_{k}$ is solution of the implicit scheme \eqref{eqSemiImplicitScheme1}. In addition,
\begin{enumerate}
\item assume that the initial points $X^0 = \{ x_i^0 \}_{i=1}^N \subset \R^n$ fulfill $x_i^0 \cdot \nu \leq \mu$ for $i = 1 \ldots N$ with $\nu\in \R^n$, $\mu \in \R$, 
\item for all $k$, if $x_i^k$ is a point on the boundary of the convex hull of $X^k$, then for points $x_j^k$ such that $|x_j^k - x_i^k| < \epsilon$ assume (implicitly) that 
 $\Pi_{ij}^k (x_j^{k+1} - x_i^{k+1})$ at $x_i^{k+1}$ is pointing inside the convex hull of $X^{k+1}$.
\end{enumerate}
Then independently of the choices of masses $m_i^k$ 
\[
x_i^k \cdot \nu \leq \mu , \quad \text{for all } i = 1 \ldots N
\]
and for all $k$ such that $V^k$ is defined.
\end{prop}
\begin{proof} 
The proof is analoguous to the proof of Proposition \ref{PlaneBarrier} in the time continuous case, now considering $i\in\{1,\ldots N\}$ with 
$x_i^{k+1} \cdot \nu = \max_{j \in\{1,\ldots N\}} x_j^{k+1} \cdot \nu$. 
\end{proof}
Next, we study a fully discrete version of the sphere comparison property (cf. Proposition~\ref{SphereInclusion}).

\begin{prop} [time discrete sphere comparison principle] \label{discreteSphereInclusion} 
Suppose that a point cloud $X^0 = \{ x_i^0 \}_{i=1}^N \subset \R^n$ is contained in a ball $B^0$ of radius $R^0$ and centered at $z$ and assume that $(V^k)_{k}$ is a sequence of point cloud varifolds which are solutions of  the implicit scheme \eqref{eqSemiImplicitScheme1}.\\
For each $k \geq 1$, define $\displaystyle R^{k} = \max_{i \in \{ 1, \ldots, N \} } | x_i^{k} - z|$ and
\begin{equation} \label{eq:P}
c^k = \min \left\lbrace \left. \frac{ \Pi_{ij}^k (x_i^{k+1} - x_j^{k+1}) (x_i^{k+1} - z) }{| x_i^{k+1} - x_j^{k+1} |^2  } \: \right| \: \begin{array}{l}
i \in \{1, \ldots, N \}, \: \displaystyle |x_i^{k+1} - z| = R^{k+1}  \text{ and } \\
j \in \{1, \ldots, N \}, \: \displaystyle | x_i^{k+1} - x_j^{k+1} | <  \epsilon 
\end{array}   \right\rbrace \: .
\end{equation}
Then, independently of the choices of masses, for all $k$ such that $V^k$ is defined, the 
positions $\{ x_j^k \}_{j=1}^N \subset \R^n$ are contained in the ball of center $z$ and radius $\sqrt{(R^0)^2 -2 d \tau \sum_{l=1}^k c^k}$. 
\end{prop} 
Let us remark, that in analogy to the time continuous case $c^k = 1$ for $\Pi_{ij} = 2 \,{\rm Id}$ and hence $\tau \sum_{l=1}^k c^k = t^k$.

\begin{proof}
Without loss of generality let  $z=0$.
Assume that at time $t_k$ the point cloud $X^k$ satisfies $\max_{j= 1 \ldots N} |x_j^k | \leq R^k$.
We choose $i\in \{1,\ldots, N\}$ such that $|x_i^{k+1}|  = \max_{j \in\{1,\ldots N\}} |x_j^{k+1}|$.
and proceed in analogy to the proof of Proposition \ref{SphereInclusion}. Thereby, one obtains 
\begin{align}
| x_i^{k+1} |^2 & = x_i^k \cdot x_i^{k+1} +  \frac{\tau}{\e} \sum_{j=1}^N  \omega_{i_0 j}^{k+1}   \Pi_{ij}^k (x_j^{k+1} - x_i^{k+1} ) \cdot x_i^{k+1} \label{eq:sphereInclusion1} \\
& \leq |x_i^k| | x_i^{k+1} | - c_k  \frac{\tau}{\e} \sum_{j=1}^N  \omega_{i_0 j}^{k+1}  |x_j^{k+1} - x_i^{k+1}|^2 \\
&=  |x_i^k| | x_i^{k+1} | -  d c_k \tau \leq \tfrac12 | x_i^{k+1} |^2 + \tfrac12 | x_i^{k} |^2 -  d c_k \tau \:. \nonumber
\end{align}
This implies $(R^{k+1})^2 = | x_i^{k+1} |^2 \leq | x_i^k |^2 - 2 c_k d \tau \leq (R^k)^2 - 2 c_k d \tau$
and thus establishes the claim.
\end{proof}

However, for practical reasons, we rather propose a linearized version of the previous implicit scheme \eqref{eqSemiImplicitScheme1}
in which we choose $V^k$ as the geometric reference in the discrete evolution of the $k$th time step and we introduce the following semi-implicit scheme  
\begin{equation} \label{eqLinearSemiImplicitScheme}
x_i^{k+1} = x_i^k + \frac{\tau}{\epsilon} \sum_{j=1}^N \omega_{ij}^{k} \Pi_{ij}^{k} \left( x_j^{\textcolor{ColorImplicit}{k+1}} - x_i^{\textcolor{ColorImplicit}{k+1}} \right)
\end{equation}
which leads to a linear system to be solved in each time step.
With $X^k = (x_1^k, \ldots , x_N^k) \in \R^{nN}$, we can rewrite this linear system as  
\begin{align}\label{eq:Mmatrix}
(M - \frac{\tau}{\epsilon} L) X^{k+1} = M X^{k}  
\end{align}
where $M$ is a diagonal matrix of size $(nN,nN)$ defined as
\[
M = {\left( \begin{array}{c|c|c}
\mu_1 {\rm I}_n   &         &              \\
\hline
                  &  \ddots &              \\
\hline
                  &         & \mu_N {\rm I}_n            
\end{array} \right)} \quad \text{with} \quad \mu_i = \displaystyle \sum_{l=1}^N m_l^k \xi \left( \frac{|x_l^k-x_i^k|}{\epsilon} \right)
\]
and $L$ is a matrix of size $(nN,nN)$, which we can see as a matrix of $(N,N)$ blocks $L_{(i,j)}$ of size $(n,n)$. 
For a fixed $i$, the block $L_{(i,j)}$ is the null matrix if $|x_j^k - x_i^k| > \e$ and
\[
L_{(i,j)} = - \frac{d}{n} \frac{m_j^k}{|x_j^k-x_i^k|}  \rho^\prime \left(\frac{|x_j^k-x_i^k|}{\epsilon} \right) \, \Pi_{ij}^k \quad \text{for } j \neq i \; ,
\]
\[
L_{(i,i)} = - \sum_{j=1, j \neq i}^N L_{(i,j)} = \frac{d}{n} \sum_{j=1}^N \frac{m_j^k}{|x_j^k-x_i^k|}  \rho^\prime \left(\frac{|x_j^k-x_i^k|}{\epsilon} \right) \, \Pi_{ij}^k \; .
\]
where $\Pi_{ij}^k$ is the $(n,n)$--matrix corresponding to the linear operator $\Pi_{ij}^k$.
 The matrix $(M - \frac{\tau}{\epsilon} L)$ is strongly diagonal dominant and thus an M-matrix. 
Hence, the system \eqref{eq:Mmatrix} is uniquely solvable and one obtains
\( X^{k+1} = (\rm{Id} - \frac{\tau}{\epsilon} M^{-1} L)^{-1} X^k\).

\section{Numerical Results} \label{sec:results}

We perform numerical tests with the semi-implicit scheme \eqref{eqLinearSemiImplicitScheme}. First we test its consistency and robustness to white noise in the simple case of a circle evolving through curvature flow (Section~\ref{sec:circle}).  The absence of singularities allows us to discard most choices for the projector $\Pi_{ij}$ in \eqref{eq:Pi_ij} and we carry on the study with $\Pi_{ij} = 2 \Pi_{P_i^\perp}$ in Section~\ref{sec:flower}. In particular, we observe that the  implicit assumption for the discrete sphere inclusion, though not proven in this case, is satisfied in our numerical experiment (see Figure~\ref{fig:sphereCond}). Next we focus in Section~\ref{sec:singular} on curvature flows of curves with crossings and junctions, taking advantage of the flexibility with respect to topological changes of point cloud representation. Our curvature flow is able to converge to both Steiner trees spanning the four vertices of a square in Figure~\ref{figSteiner1}. Last, in Section~\ref{sec:singularSurf}, we draw our attention to surfaces and recover the minimal cone spanning the edges of a tetraedron in the limit of our mean curvature flow as well as a candidate minimal surface spanning the edges of a cube, see Figures~\ref{figTetra} and \ref{figCube}.

\begin{rem}[Nearest neighbour graph and $kd$--tree structure] \label{remk:knnTopology}
From a practical perspective, we compute neighbourhoods thanks to a $kd$--tree structure that is computed with the library \emph{Nanoflann} \cite{nanoflann}. Note that the knowledge of the nearest neighbour graph encodes information that can be interpreted as a discrete counterpart to the local topology of the object. Comparing with triangulated surface, it can be seen as the counterpart to the mesh connectivity information. In fact, in case of triangular surfaces 
one can either move points and obtain a new mesh with the same connectivity structure, or one can move points and remesh the set of points, which is a computationally demanding operation but often necessary since moving points may create overlaps or crossings in the mesh. In a point cloud, after moving the positions of the points, one has the same choice: one can recompute  the nearest neighbour graph or leave it unchanged. As pointed out below, this will be useful to stabilize the  point cloud evolution  
close to triple points. Such operations are straightforward to handle in a $kd$--tree.
\end{rem}

\subsection{Evolution of circles} \label{sec:circle}
 
As already mentioned when we discussed  sphere comparison principle (see Section~\ref{sec:timecontsub}) a circle of initial radius $R_0$ evolves into concentric circles of radius $R(t) = \sqrt{R_0^2 - 2t}$ at time $t$ under mean curvature flow. A first step in order to validate our approach is to check this property on our scheme. 
\begin{center}
\setcounter{subfigure}{0}
\begin{figure}[!htp]
\subcaptionbox{$\Pi_{ij} = \Pi_{P_j}$,\\ $t = 0$, $0.03$, $0.06$, $0.09$, $0.12$ \label{figComparisonOperatorsCircle1}}{\includegraphics[width=0.31\textwidth]{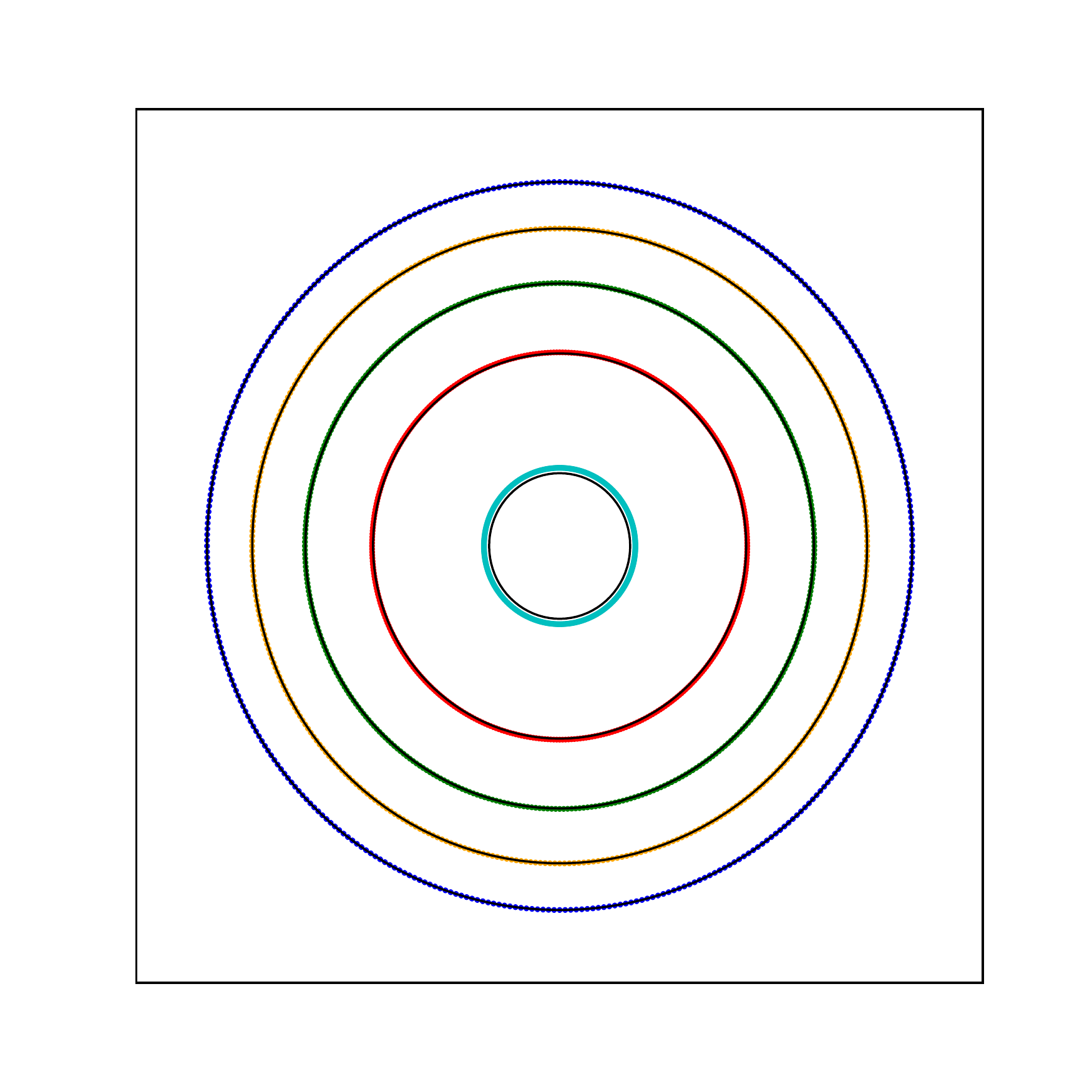} }
\subcaptionbox{$\Pi_{ij} = \Pi_{P_i^\perp} \circ \Pi_{P_j}$,\\ $t = 0$, $0.03$, $0.06$, $0.09$, $0.12$ \label{figComparisonOperatorsCircle2}}{\includegraphics[width=0.31\textwidth]{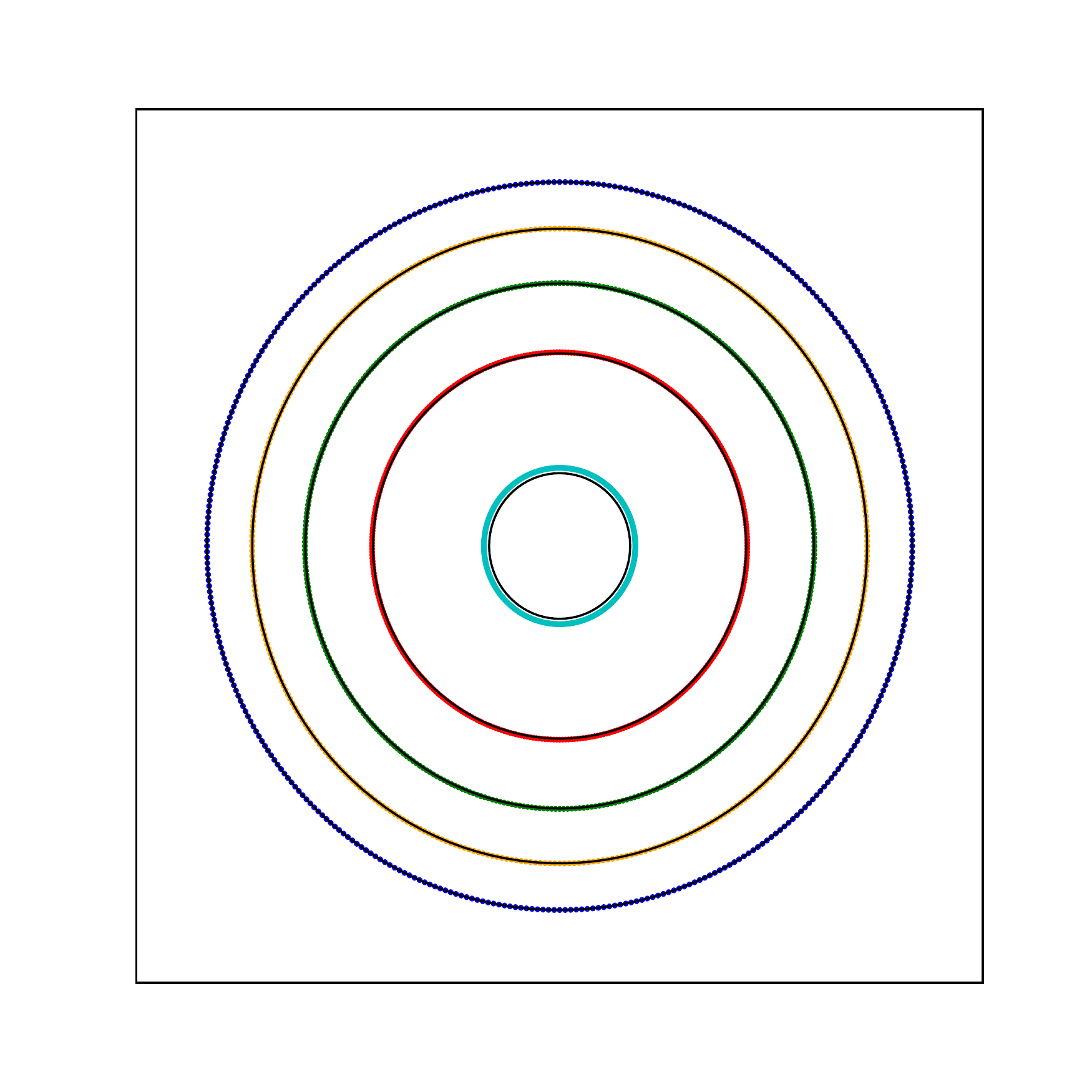} }
\subcaptionbox{$\Pi_{ij} = 2\Pi_{P_i^\perp}$,\\ $t = 0$, $0.03$, $0.06$, $0.09$, $0.12$  \label{figComparisonOperatorsCircle3}}{\includegraphics[width=0.31\textwidth]{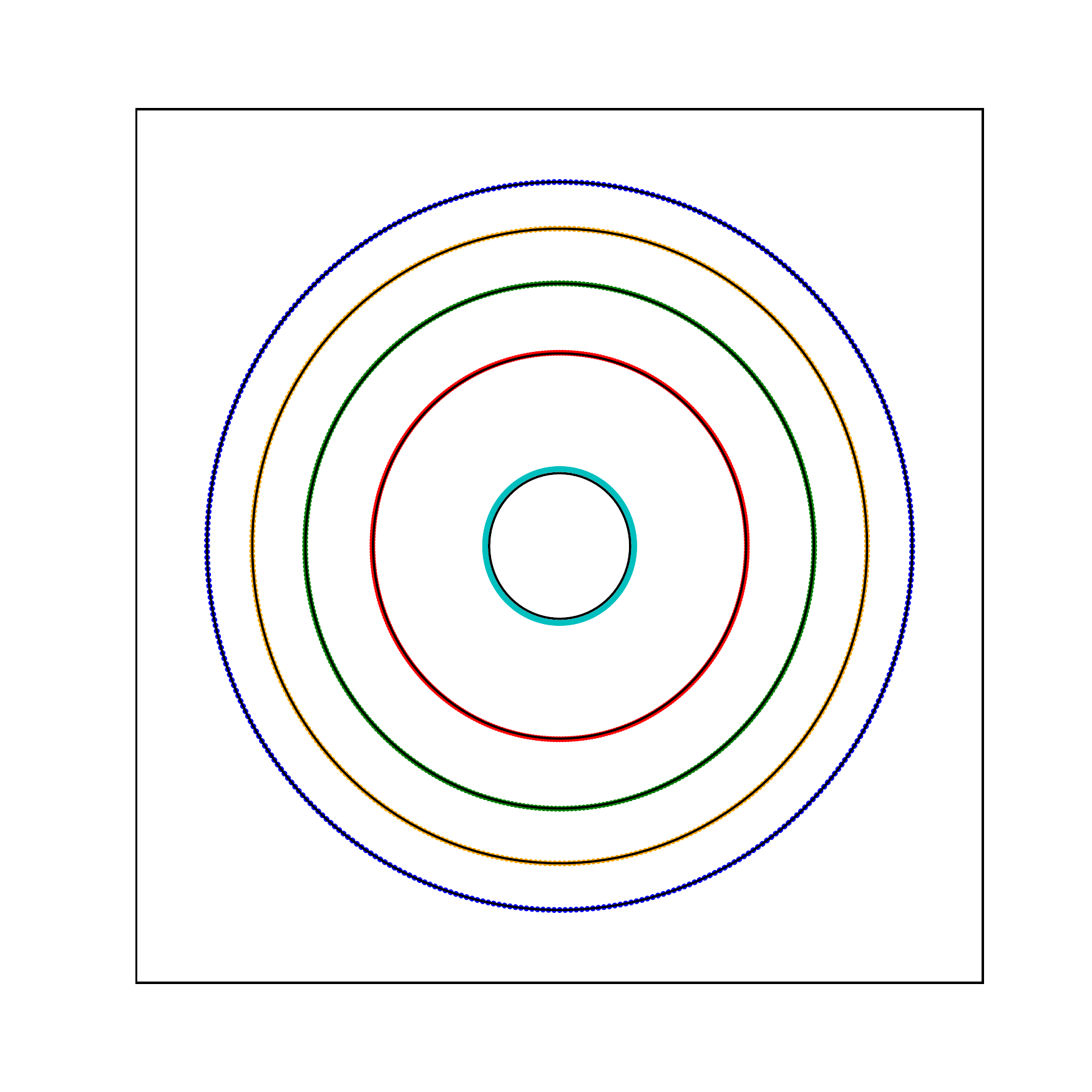}}\\
\subcaptionbox{$\Pi_{ij} = 2{\rm Id}$,\\ $t = 0$, $0.03$, $0.06$, $0.09$, $0.117$ \label{figComparisonOperatorsCircle4}}{\includegraphics[width=0.31\textwidth]{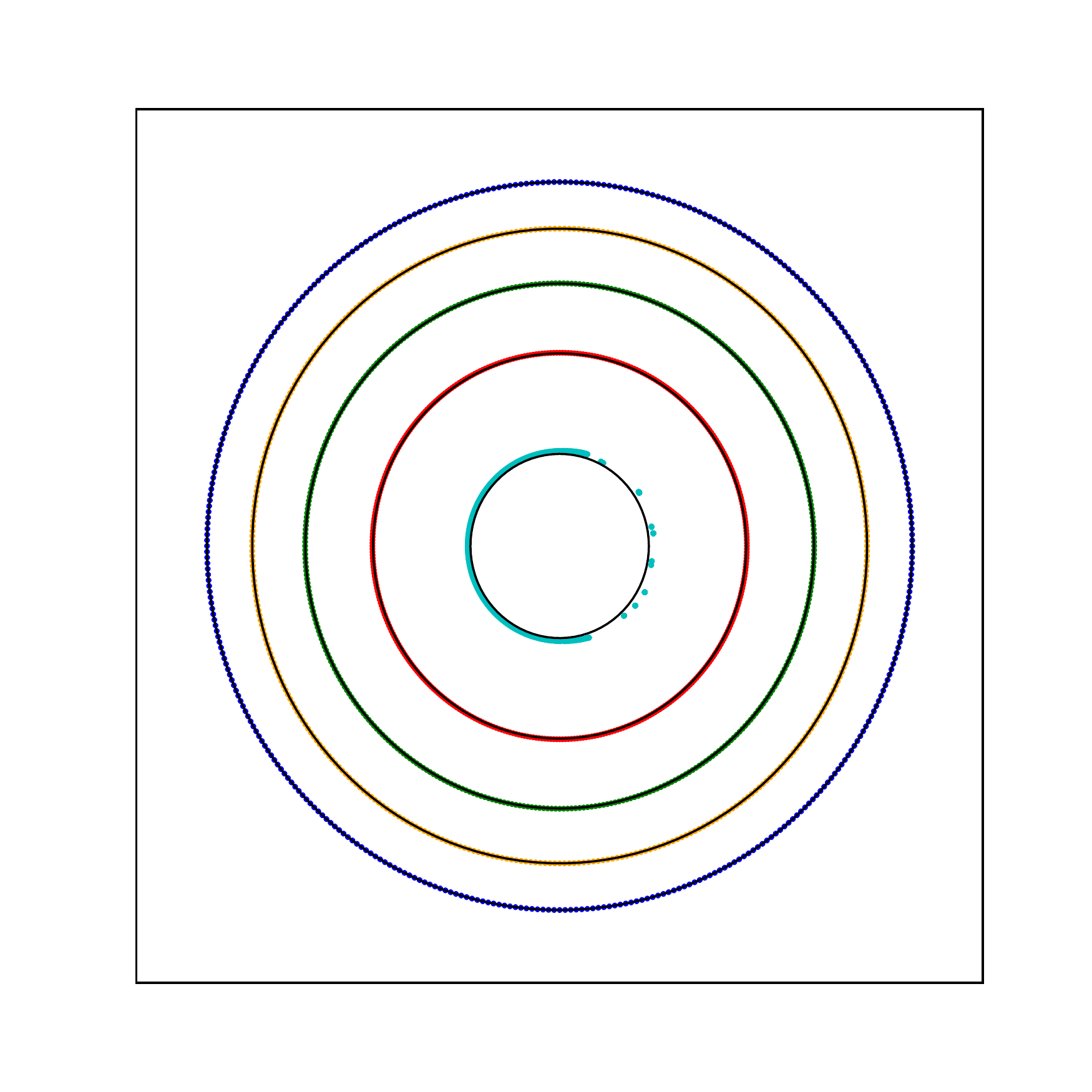} }
\subcaptionbox{$\Pi_{ij} = -2\Pi_{P_j^\perp}$,\\ $t = 0$, $0.02$, $0.03$, $0.12$ \label{figComparisonOperatorsCircle5}}{\includegraphics[width=0.31\textwidth]{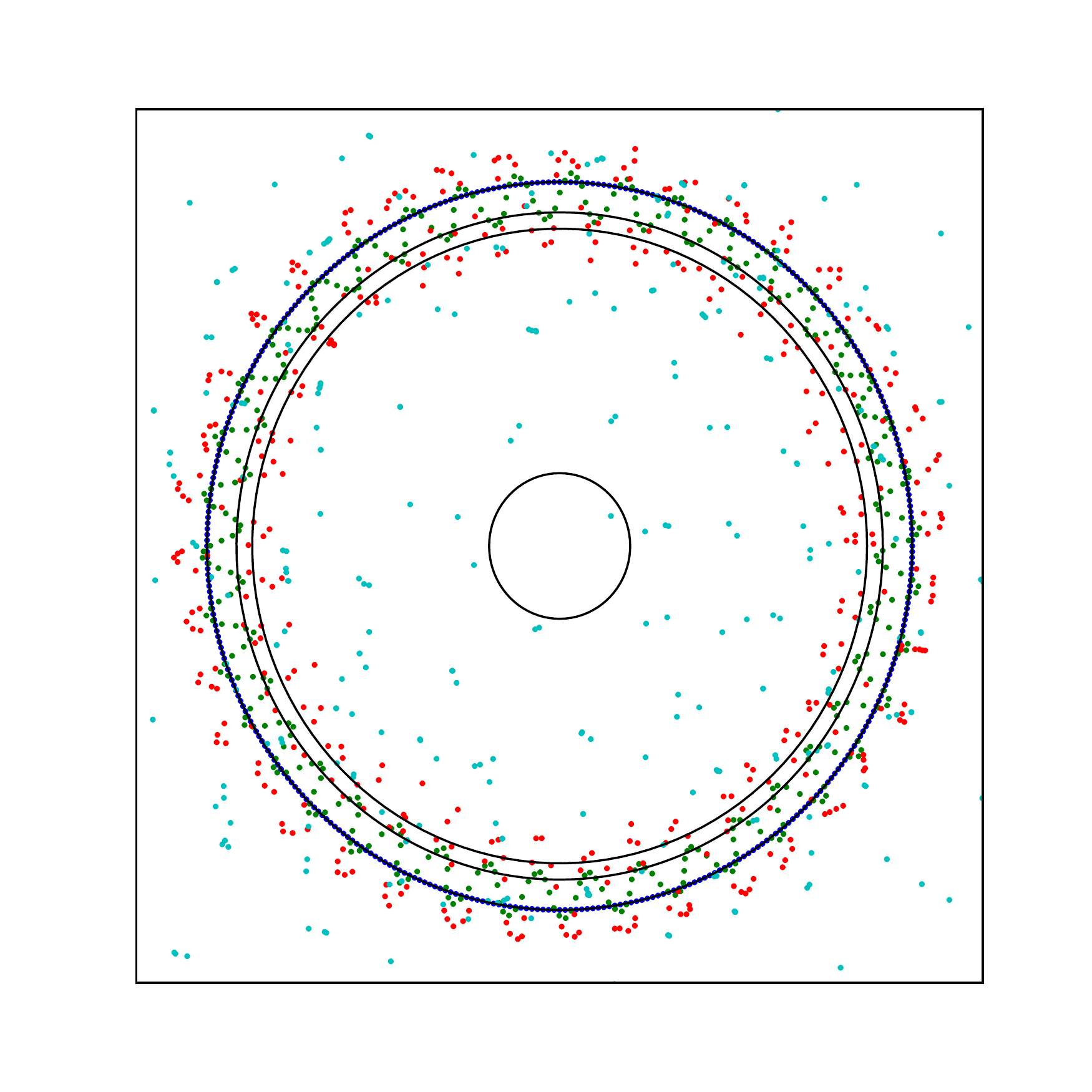} }
\subcaptionbox{$\Pi_{ij} = -2\Pi_{P_i^\perp} \circ \Pi_{P_j^\perp}$,\\ $t = 0$, $0.01$, $0.03$, $0.12$ \label{figComparisonOperatorsCircle6}}{\includegraphics[width=0.31\textwidth]{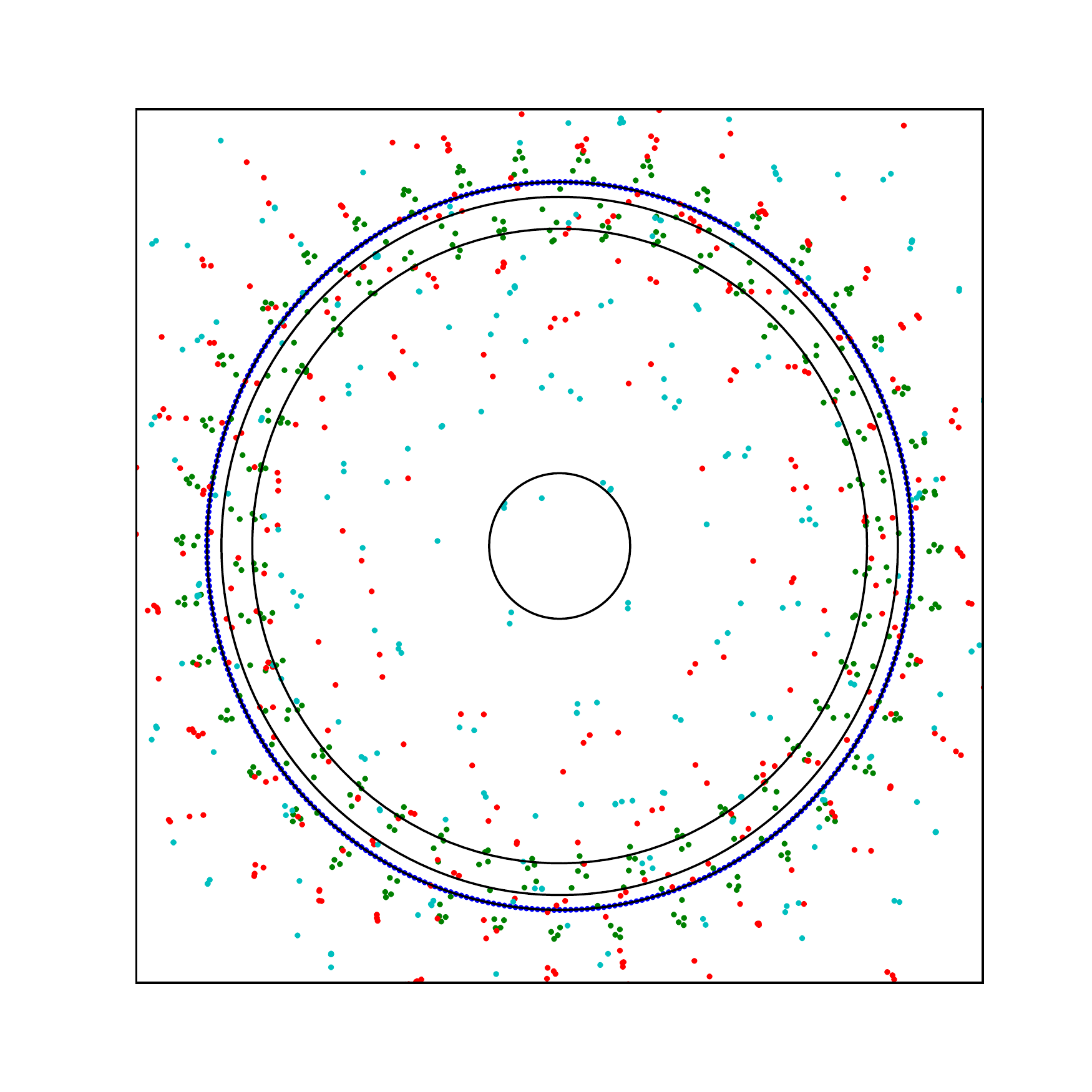} }
\caption{Comparison of the results given by the semi-implicit scheme \eqref{eqLinearSemiImplicitScheme} for different choices of operator $\Pi_{ij}$. The test is performed on a circle uniformly discretized with $N = 400$ points, the number of points for computing the mass is fixed to $k_\delta = 3$, the number of points for the regression is fixed to $k_\sigma = 17$ and the number of points for the computation of curvature is fixed to $k_\epsilon = 15$, the time step is $\tau = 0.1/N = 0.0005$. The numerical solutions and the exact solutions in black are represented at same times $t$. Simulations are stopped then resulting in clutter due to noise amplification.
\label{figComparisonOperatorsCircle}}
\end{figure}
\end{center}
We start with a circle of radius $R_0 = 0.5$ uniformly discretized with $N$ points. In Figure~\ref{figComparisonOperatorsCircle}, we first perform a qualitative comparison of the behaviour of the scheme depending on the operator $\Pi_{ij}$.
In Figure~\ref{figComparisonOperatorsCircle5} and \ref{figComparisonOperatorsCircle6} we observe strong instabilities after short time, while in Figure~\ref{figComparisonOperatorsCircle4}, instabilities appear after longer time. Consequently, we focus on projection operators tested in Figure~\ref{figComparisonOperatorsCircle1}, \ref{figComparisonOperatorsCircle2} and \ref{figComparisonOperatorsCircle3}, that is $\Pi_{ij} \in \left\lbrace \Pi_{P_j},  \, \Pi_{P_i^\perp} \circ \Pi_{P_j} , \, \Pi_{P_i^\perp}   \right\rbrace$.

As a second step to validate our approach, we then test the robustness with respect to white noise: we introduce an initial white noise of standard deviation $s$ on the circle of radius $R_0 = 0.5$.

\noindent In Figure~\ref{fig:circleTnoise}, we observe tangential instabilities with agglomeration of points in very short time. In Figure~\ref{fig:circleTNnoise}, we observe that noise is not smoothed but transported, which is reasonable, given that the  initial projection onto $P_j$ makes $\Pi_{ij} = \Pi_{P_i^\perp} \circ \Pi_{P_j}$ blind to normal noise. As a consequence of this non-smoothing effect, the speed of evolution is considerably slowed down. In Figure~\ref{fig:circleNnoise} and \ref{fig:circleNnoise2} we observe that noise is smoothed in a few steps. The evolution is then close to the exact one.  This is further improved in Figure~\ref{fig:circleNnoise} and in Figure~\ref{fig:circleNnoise2} with even higher initial noise. From this first analysis, we conclude that the most robust choice among \eqref{eq:Pi_ij} is $\Pi_{ij} = 2 \Pi_{P_i^\perp}$ when dealing with discretization of smooth curves. 
We then carry on our study with $\Pi_{ij} = 2 \Pi_{P_i^\perp}$.
\setcounter{subfigure}{0}
\begin{center}
\begin{figure}[!htp]
\subcaptionbox{$\Pi_{ij} = \Pi_{P_j}$,\\ $t = 0$, $0.006$,\\ $s = 0.0125$. \label{fig:circleTnoise}}{\includegraphics[width=0.23\textwidth]{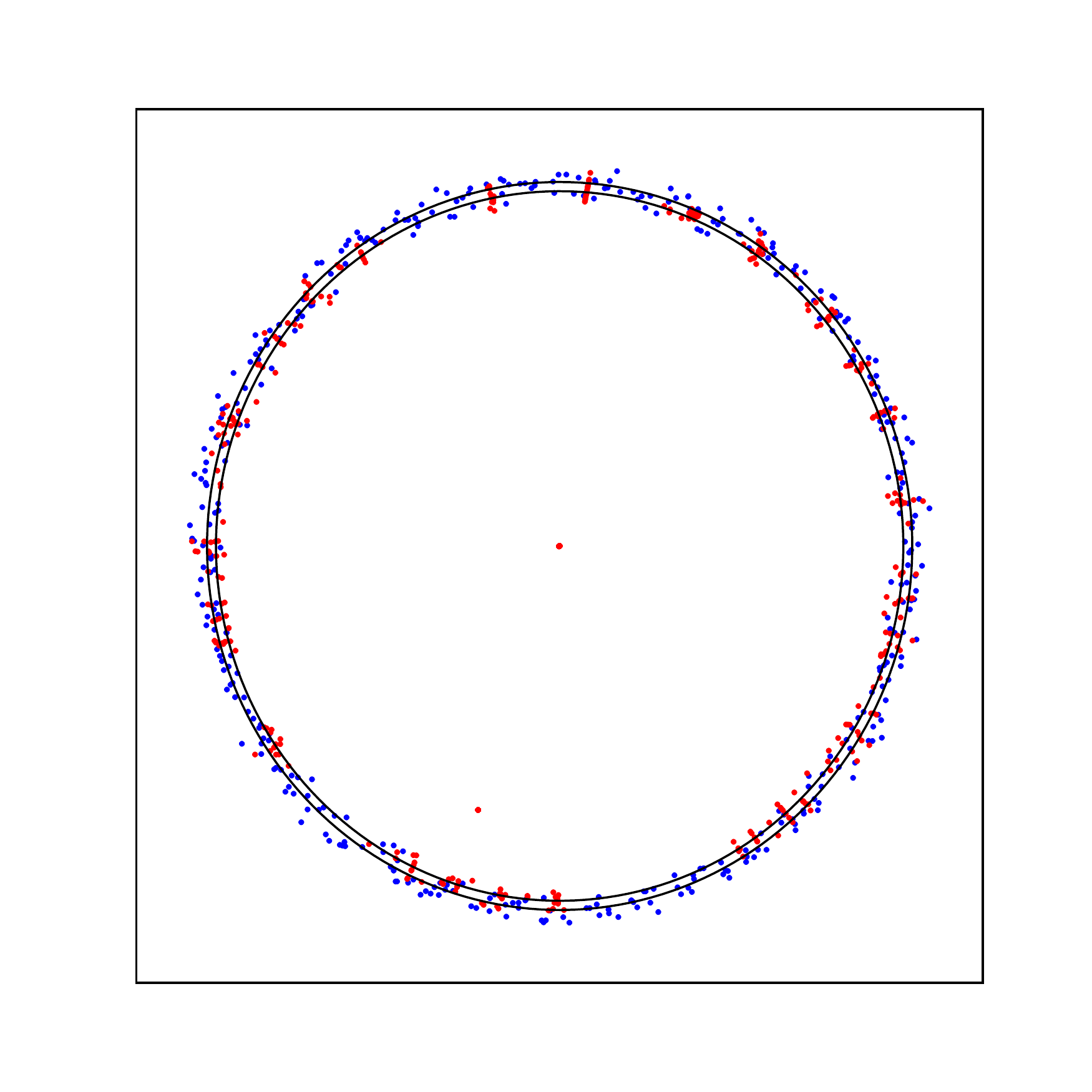} }
\subcaptionbox{$\Pi_{ij} = \Pi_{P_i^\perp} \circ \Pi_{P_j}$,\\ $t = 0$, $0.03$, $0.06$, $0.09$, $0.12$,\\ $s = 0.0125$. \label{fig:circleTNnoise}}{\includegraphics[width=0.23\textwidth]{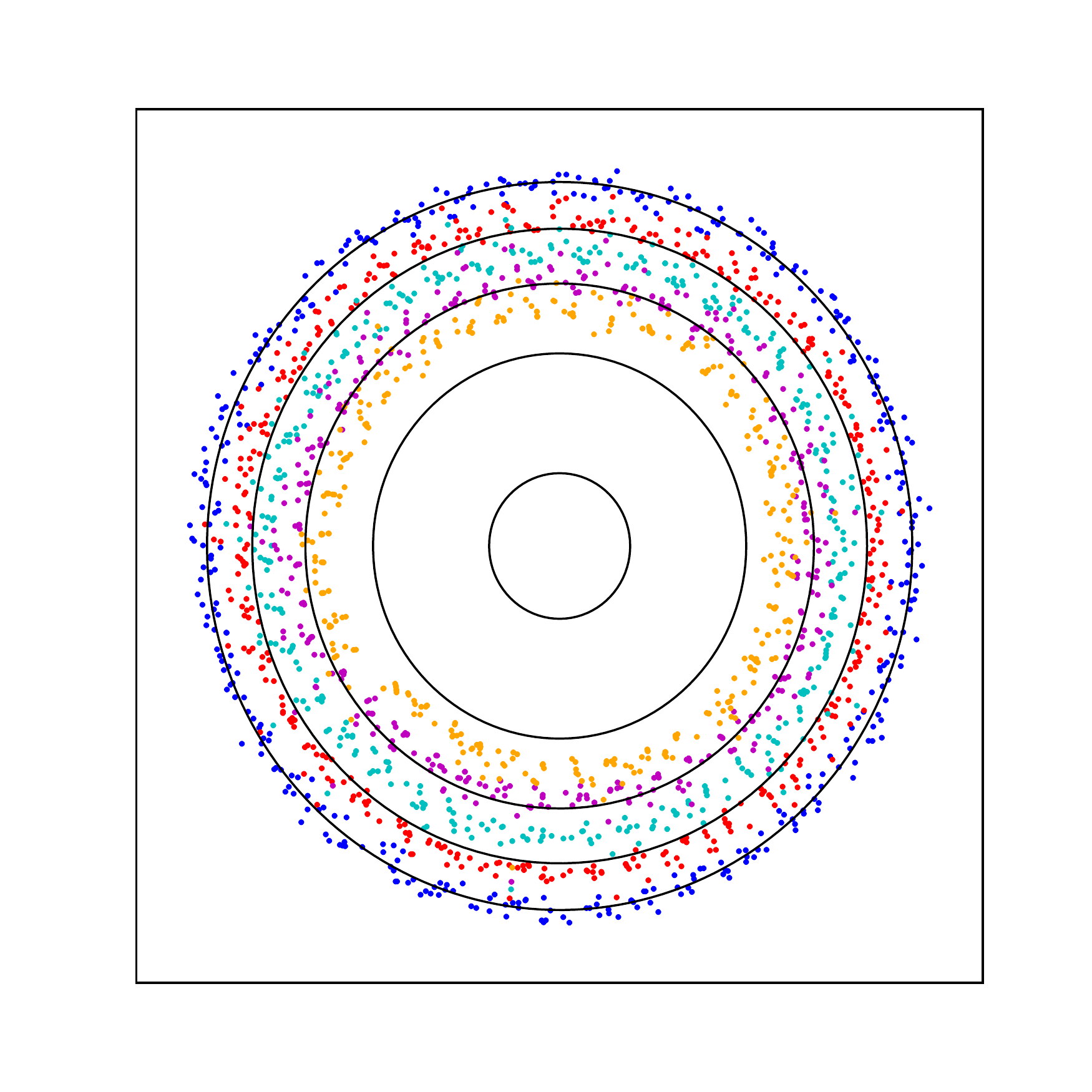} }
\subcaptionbox{$\Pi_{ij} = 2\Pi_{P_i^\perp}$,\\ $t = 0$, $0.03$, $0.06$, $0.09$, $0.12$,\\ $s = 0.0125$. \label{fig:circleNnoise}}{\includegraphics[width=0.23\textwidth]{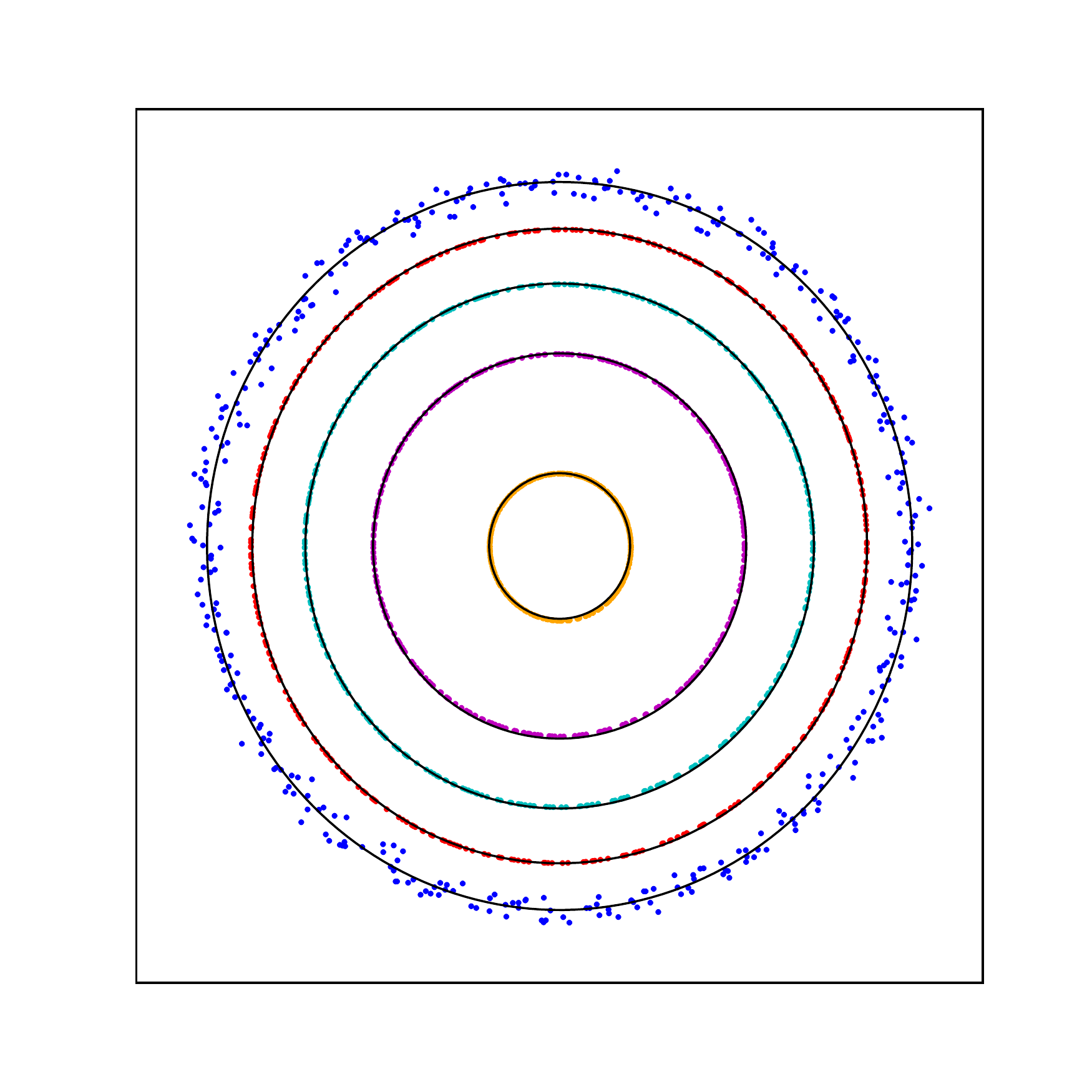} }
\subcaptionbox{$\Pi_{ij} = 2\Pi_{P_i^\perp}$,\\ $t = 0$, $0.03$, $0.06$, $0.09$, $0.12$,\\ $s = 0.0375$. \label{fig:circleNnoise2}}{\includegraphics[width=0.23\textwidth]{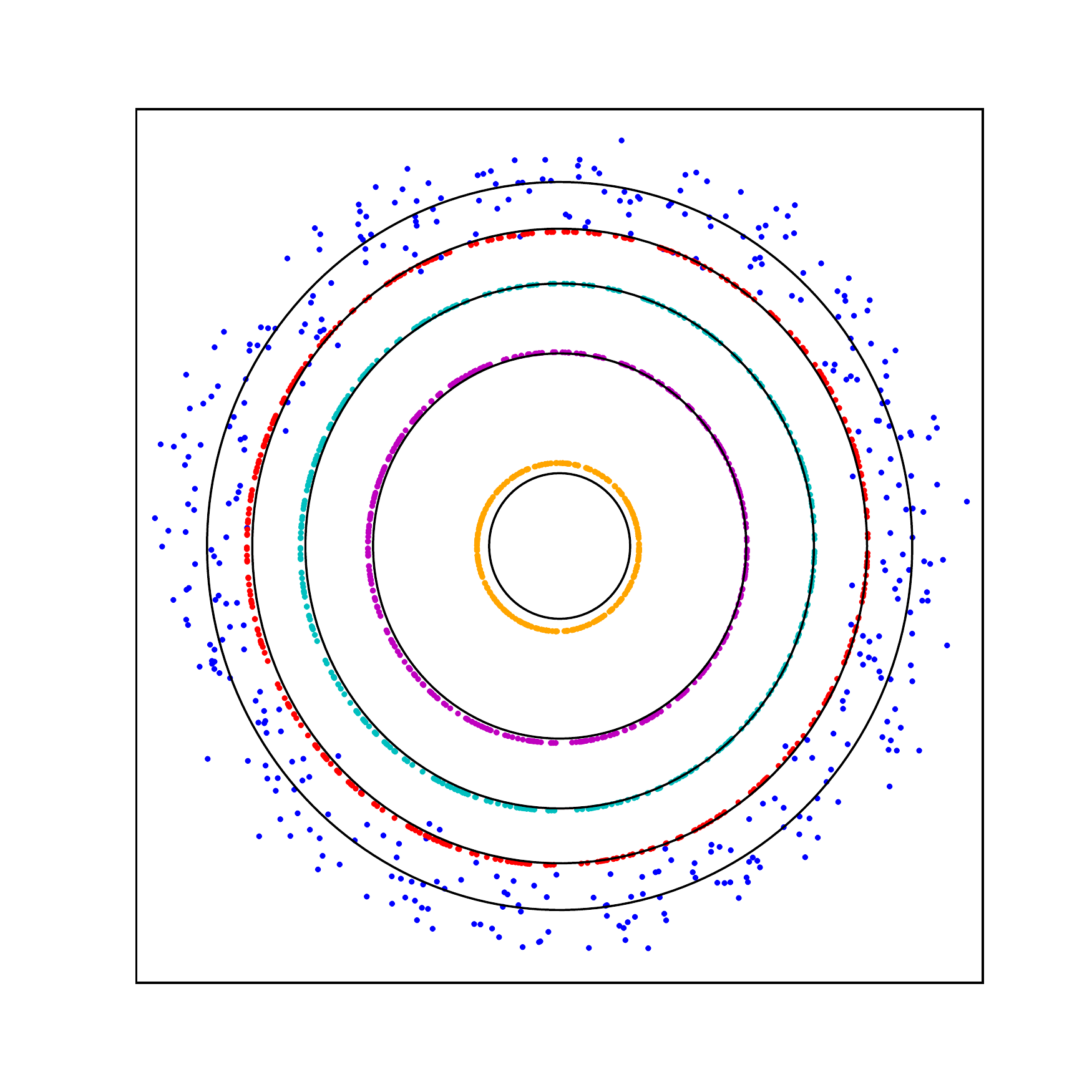} }
\caption{Comparison of the results given by the semi-implicit scheme \eqref{eqLinearSemiImplicitScheme} for different choices of operator $\Pi_{ij}$ when adding white noise. The test is performed on circle uniformly discretized with $N = 400$ points, the number of points for computing the mass is fixed to $k _\delta = 3$, the number of points for the regression is fixed to $k_\sigma = 17$ and the number of points for the computation of curvature is fixed to $k_\epsilon = 15$, the time step is $\tau = 0.1/N = 0.0005$. The numerical solutions and the exact (black line) solutions are represented at same times $t$. To the initial circle is added a Gaussian noise with standard deviation $s = 5/N = 0.0125$ in \ref{fig:circleTnoise}--\ref{fig:circleNnoise} and $s = 15/N = 0.0375$ in \ref{fig:circleNnoise2}.  \label{figComparisonCircleOperatorsNoise}}
\end{figure}
\end{center}

 Next, we check the first order convergence in time. To that extend, we compute for the circle evolution the relative mean error after a time $T$ defined as
\begin{equation} \label{eqMeanError}
e(T) = \frac{1}{N} \sum_{i=1}^N \frac{|R(T) -  |x_i(T)| |}{R(T)} \quad \text{with} \quad R(T) = \sqrt{R_0^2 - 2T} \: ,
\end{equation}
for  successively smaller time steps $\tau = 2^{-k}/N$, $k \in \{ 0, 1, \ldots, 8\}$. 
The test is performed on a uniformly discretized circle of radius $R_0 = 0.5$ with $N = 400$ points, the number of points used for the  computation of masses is $k_\delta = 3$, of tangent directions is $k_\sigma = 17$ and of curvatures is $k_\epsilon = 15$. 
The error $e(T)$ is computed at time $T = 0.1$ while the extinction time is $T_{ext} = R_0^2/2 = 0.125$. 
In Figure~\ref{figTimeStepErrorN400}, 
we observe a convergence of first order in time in the case without noise (blue curve labelled "without noise"), while when adding an initial noise independent of $\tau$, the error $e_T$ decay stabilizes and grows (green curve labelled "fixed noise"). So as to understand the behaviour of $e_T$ in the case of noise, we perform the same experiment, but adding a white noise of standard deviation $s = \sqrt{2^{-k}} 5/N$  which is linked via $N$ to the time step size $\tau = 2^{-k}/N$ that is $s^2 = (25/N) \tau$ (red curve labelled "adaptive noise"). We then retrieve the  first order convergence in time previously observed without noise.
Notice that due to the lack of a Lipschitz bound for the map $t \mapsto R(t)$ near the extinction time $T_{ext} = 0.125$, the error
 tends to explode for times close to $T_{ext}$ as it is pointed out in Figure~\ref{figRevol}.

We eventually study the influence of the number of points $k_\sigma$ used for the regression and  the number of points $k_\epsilon$ used for the computation of the curvature on a circle uniformly discretized with $N = 400$ points.
Due to the symmetry of the configuration, all choices happen to be equivalent and we hence add an initial noise of standard deviation $s=5/N = 0.0125$. The number of points used to compute the mass is $3$. We compute the mean error $e(T)$ \eqref{eqMeanError} obtained at time $T = 0.1$ for a time step $\tau = 2^{-5}/N = 7.8125 \, 10^{-5}$. In each box of Figure~\ref{fig:errorksigmakepsilon}, the error $e(T)$ corresponding to $(k_\sigma, k_\epsilon)$ is given.
We observe that the number of points used to compute curvature must be large enough (at least $k_\epsilon = 13$ in this case) to obtain acceptable errors. The choice of $k_\sigma$ appears to be less crutial and $k_\sigma =7$ seems to be  already sufficiently  large.

\setcounter{subfigure}{0}
\begin{figure}[!htbp]
\centering
\begin{tabular}{cc}
\multirow{-8}{*}{\subcaptionbox{\label{figTimeStepErrorN400}}{\includegraphics[width=0.40\textwidth]{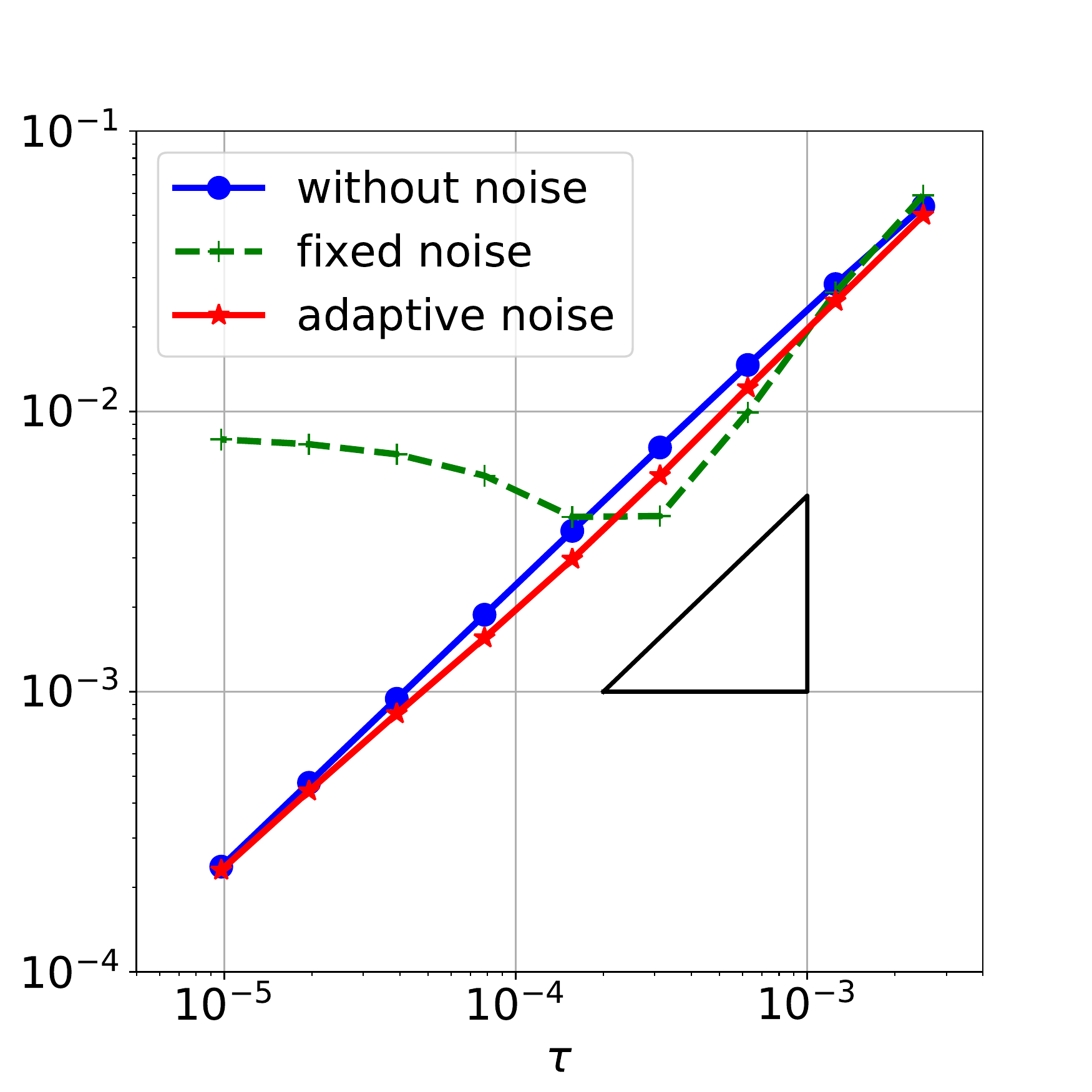}} }
& \subcaptionbox{\label{figRevol}}{\includegraphics[width=0.53\textwidth]
{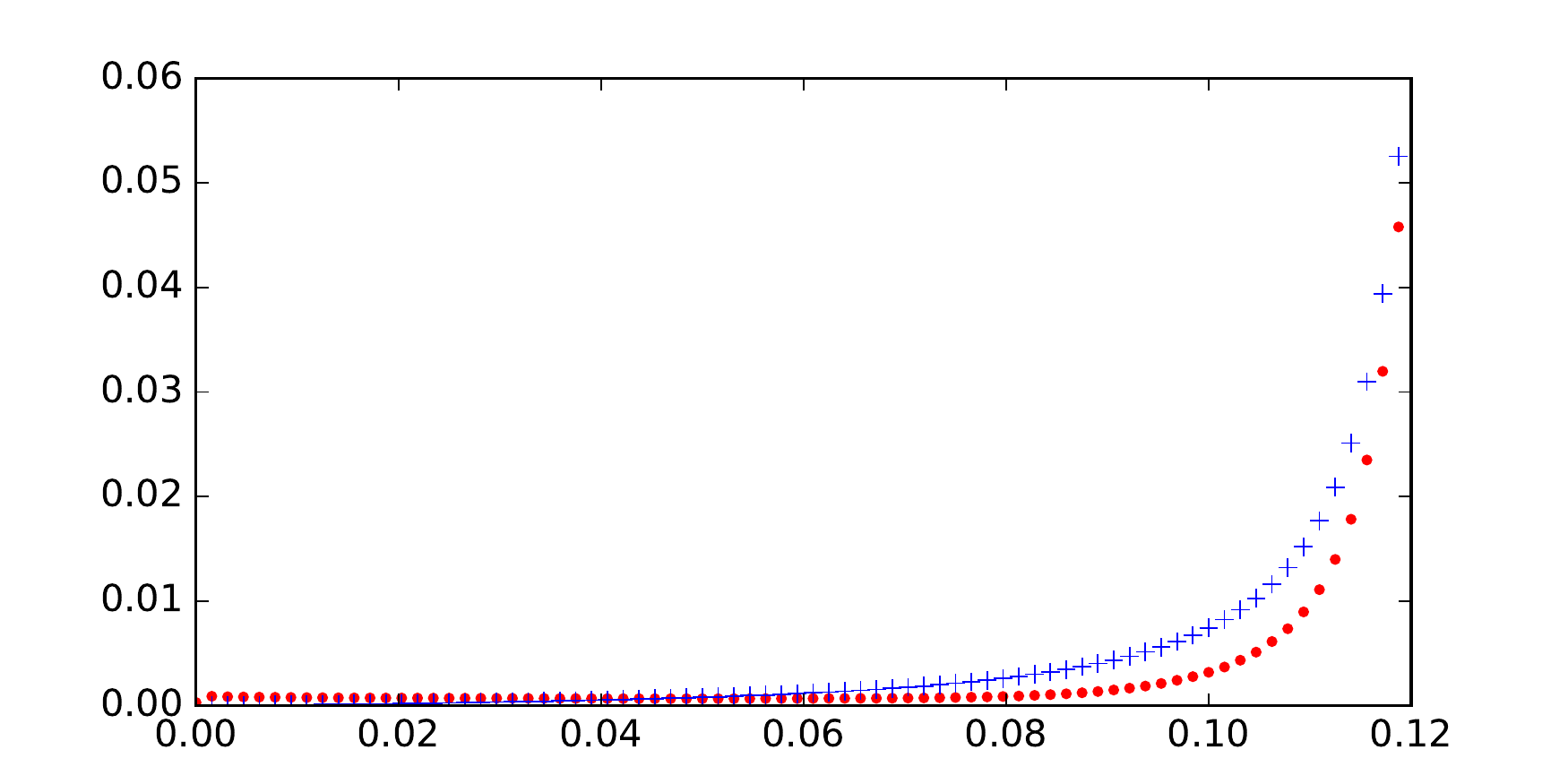} }\\
& \subcaptionbox{\label{fig:errorksigmakepsilon}}{
\begin{tabular}{|c|c|c|c|c|c|}
\hline
$k_\sigma \setminus k_\epsilon$ & 7    & 9    & 13     & 21      & 37     \\
\hline 
7 & 0.9 & 0.3 & 0.03 & 0.02 & 0.02 \\
\hline 
9 & 0.03 & 0.1 & 0.008 & 0.01 & 0.009 \\
\hline 
13 & 0.05 & 0.02 & 0.06 & 0.01 & 0.008 \\
\hline 
21 & 0.06 & 0.03 & 0.01 & 0.03 & 0.005 \\
\hline 
37 & 0.15 & 0.08 & 0.03 & 0.004 & 0.01 \\
\hline 
\end{tabular} }
\end{tabular}
\caption{(a) Decay of the mean error \eqref{eqMeanError} at time $T = 0.1$ when the time step $\tau \in \{ 2^{-k} / N \: : \: k = 0, \ldots, 8 \}$ is refined, $N = 400$. Black triangle indicates slope $1$ in log--log scale.
(b) Error $e(t)$ represented with respect to time $t$, for $\tau = 6.25 \: 10^{-4}$, in red with initial white noise (of standard deviation $1.25 \: 10^{-3}$) and in blue without noise. (c) Error $e(0.1)$ for different numbers of points $k_\sigma$ used for computing tangent and $k_\epsilon$ used for computing curvature.}
\end{figure}
\subsection{Evolution of more general curves} \label{sec:flower}

In this section, we apply our scheme \eqref{eqLinearSemiImplicitScheme} with $\Pi_{ij} = \Pi_{P_i^\perp}$ to a point cloud varifold $V = \sum_{i=1}^N m_i \delta_{x_i} \otimes \delta_{P_i}$ associated to the discretization with $N$ points of the following parametrized curve:
\[
x(t) = \left( r(t) \cos t, \: r(t) \sin t\right) 
\quad \text{with} \quad r(t) = \frac{1}{2} \left(1+r_0 \sin \left(6t+ \frac{\pi}{2} \right) \right), \, r_0 = 0,4 \quad t \in [0, 2\pi[ \: .
\]
The parameter interval is uniformly discretized so that for $i \in \{0, \ldots, N-1\}$, $x_i = x(2i\pi/N)$, the masses $m_i$ and tangents $P_i$ are then computed from the positions. In Figure~\ref{figFlowerNoise}, we apply our linear semi-implicit scheme \eqref{eqLinearSemiImplicitScheme} (for $\Pi_{ij} = 2 \Pi_{P_i^\perp}$) to $V$ both with and without noise.
The test is performed with $N = 400$ points, the number of points for computing the masses is set to $k_\delta = 7$, the number of points for the regression is  $k_\sigma = 19$ and the number of points for the computation of curvature is $k_\epsilon = 25$. We finally choose the time step $\tau = 1/N = 0.0025$. 
We compare the results
to a reference solution computed thanks to a usual parametric mean curvature flow (\cite{Dz91}) with a fine discretization and a time step $\tau_{ref} = 5.10^{-8}=10^{-5} \tau$. We observe a consistent evolution in Figure~\ref{figFlowerNoise}$(a)$ to \ref{figFlowerNoise}$(d)$, even when some noise is added to the initial shape. In Figure~\ref{figFlowerNoise}$(e)$, we check that the discrete sphere inclusion, that was only  established for the fully implicit scheme (see Proposition~\ref{discreteSphereInclusion}), is however satisfied in our experiment.
\setcounter{subfigure}{0}
\begin{figure}[!ht]
\begin{minipage}{0.55\textwidth}
\subcaptionbox{time $t=0$}{\includegraphics[width=0.48\textwidth]{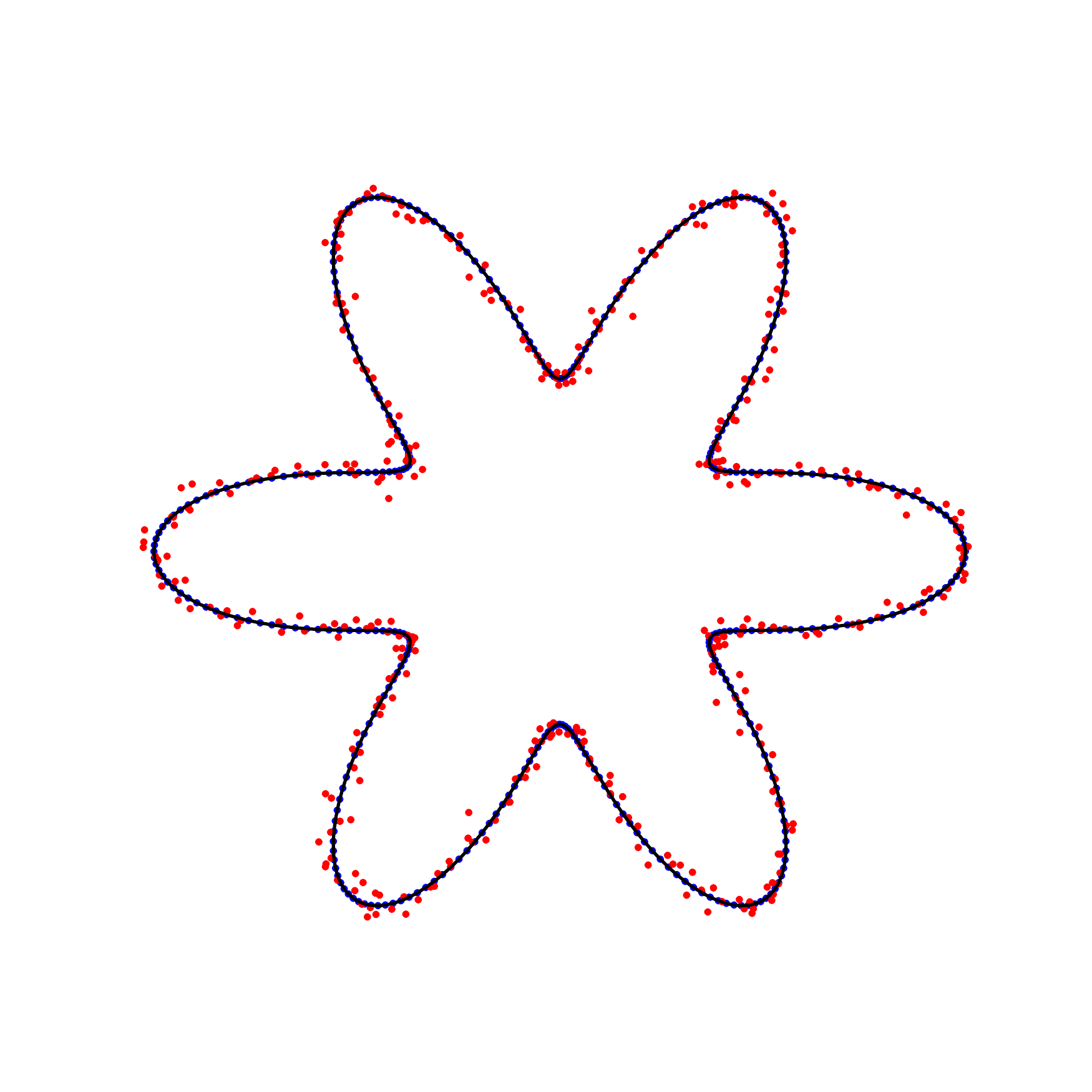}}
\subcaptionbox{time $t=0.01$}{\includegraphics[width=0.48\textwidth]{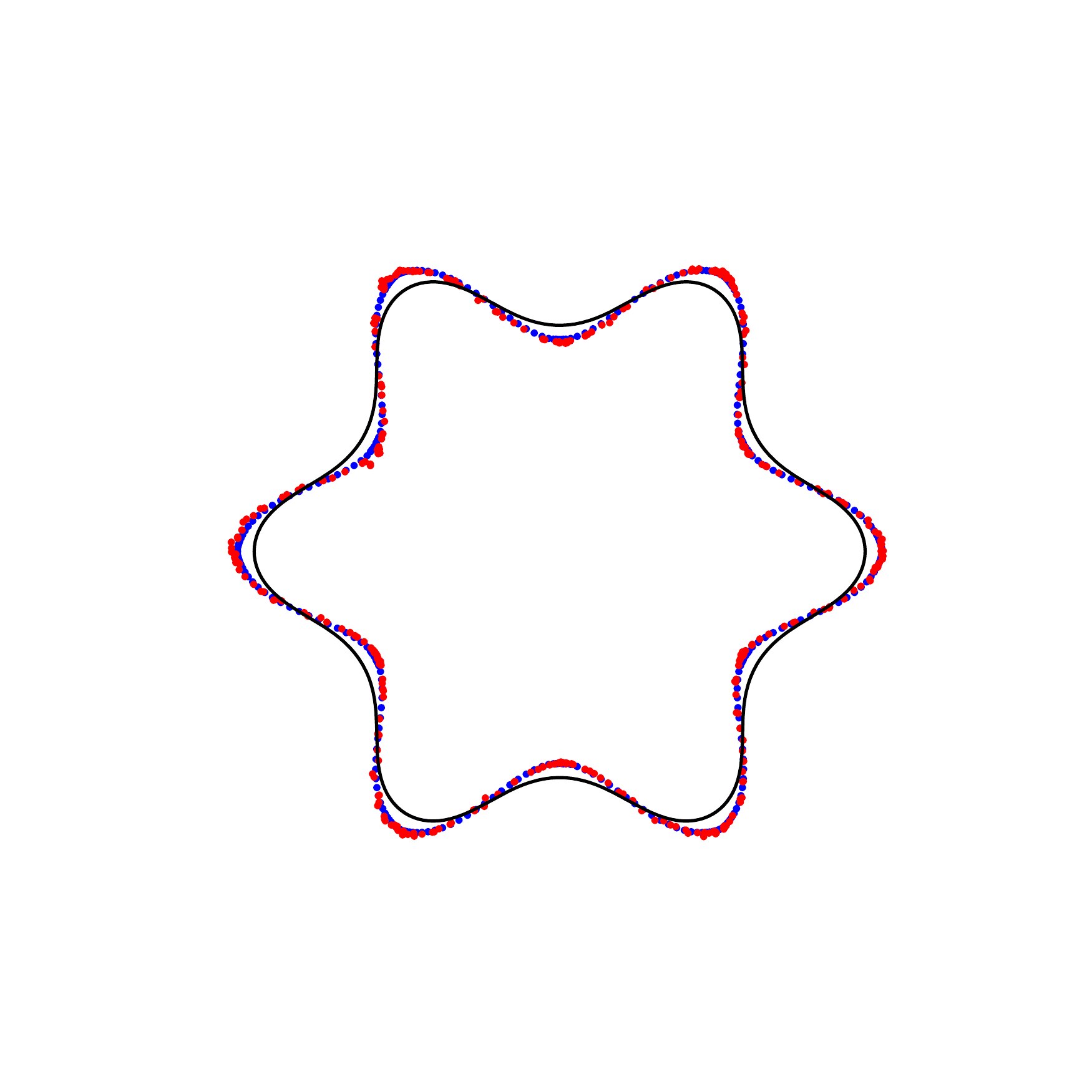}}
\subcaptionbox{time $t=0.025$}{\includegraphics[width=0.48\textwidth]{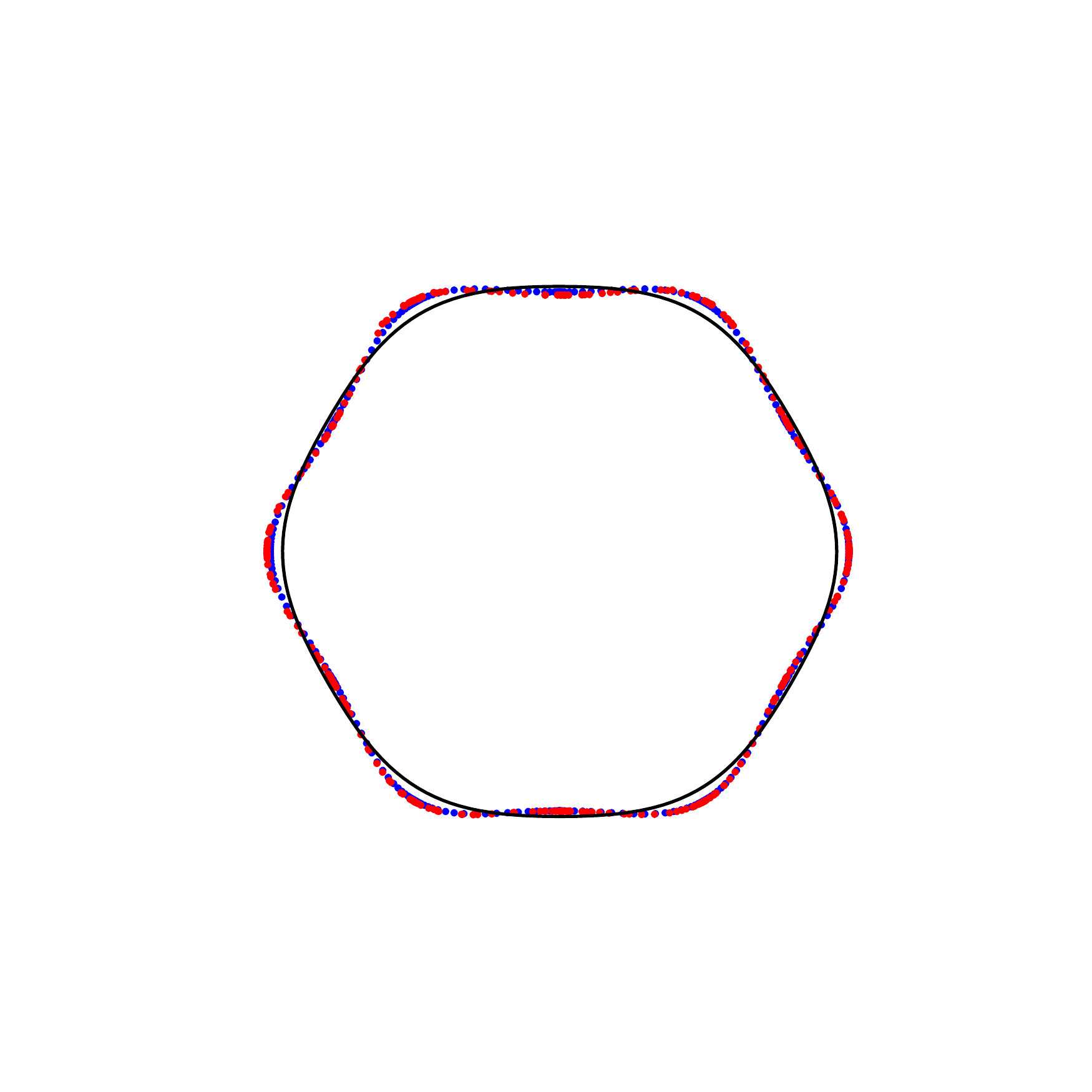}}
\subcaptionbox{time $t= 0.05$}{\includegraphics[width=0.48\textwidth]{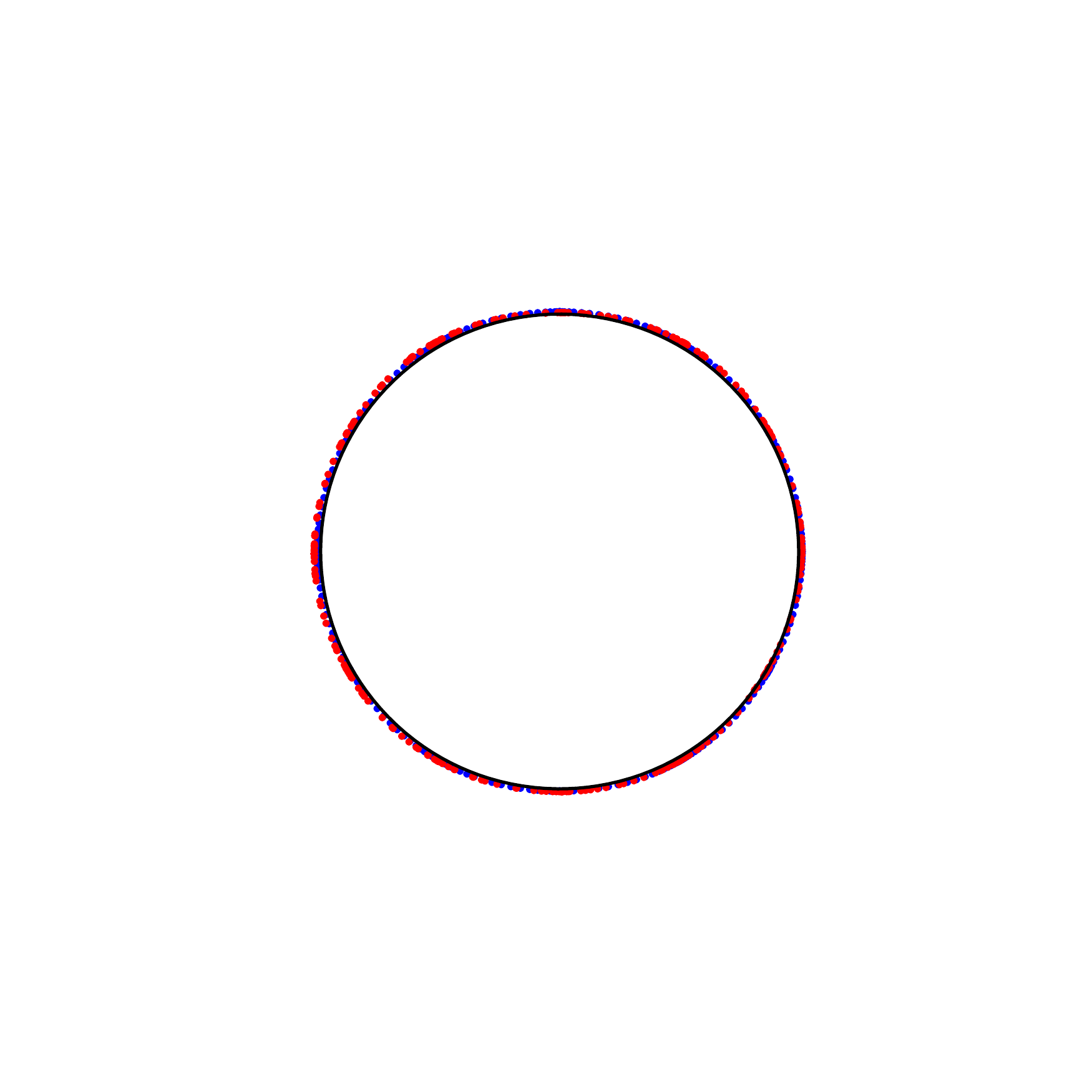}}
\end{minipage}
\begin{minipage}{0.43\textwidth}
\subcaptionbox{\label{fig:sphereCond}}{\includegraphics[width=0.95\textwidth]{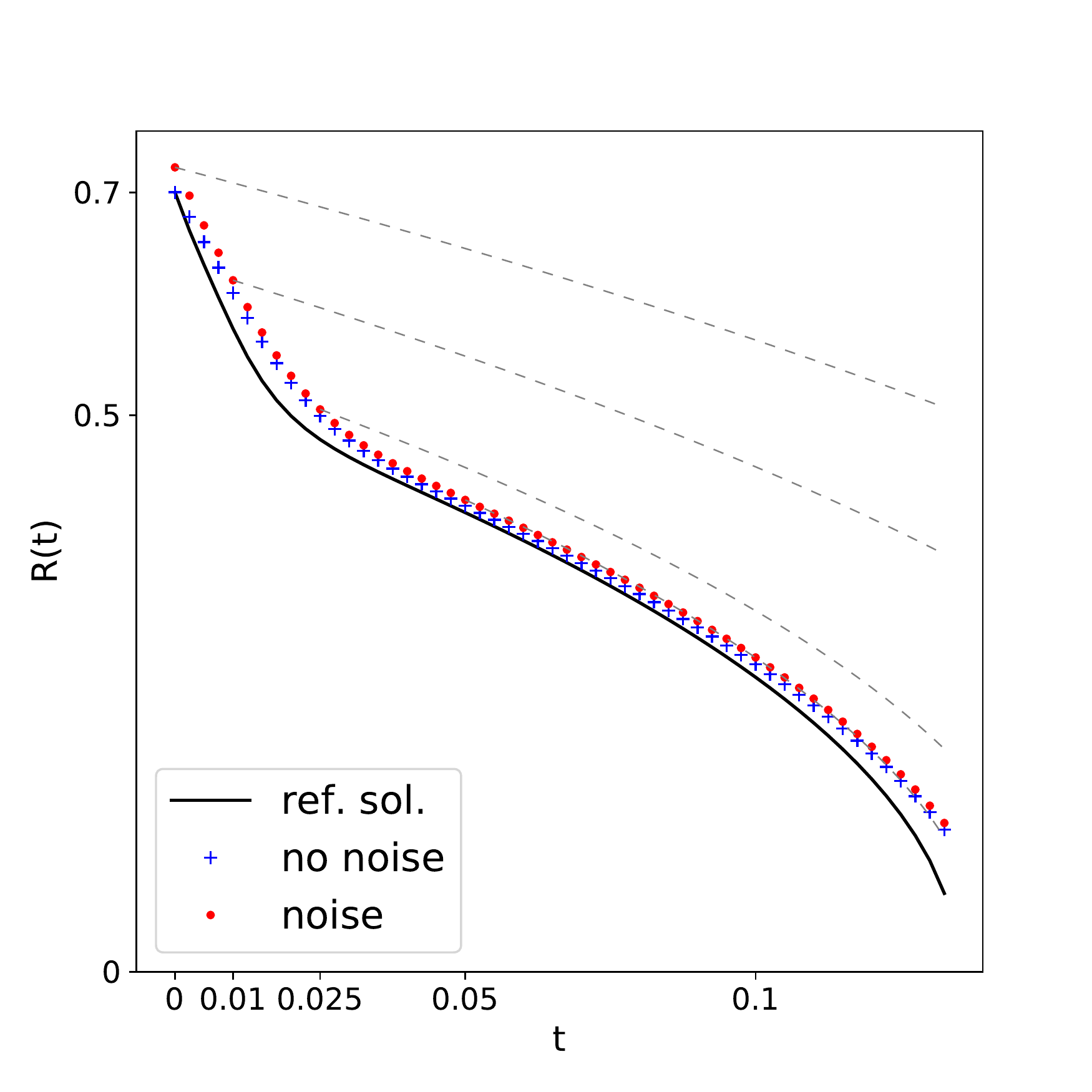}}
\end{minipage}
\caption{{\it Left figure (a) to (d):} Blue and red point clouds are the results given by 
our scheme without adding noise for  the \textcolor{ColorImplicit}{blue} point cloud and with white noise of standard deviation $5/N$ for the \textcolor{ColorImplicit}{red} point cloud.
The black curve is the reference solution. {\it Right figure (e):} we compute and represent the minimal radius $R(t)$ of the circle (centered at $0$) and including the points at time $t$, keeping the same parameters and color code as in the left figure. We moreover add grey dashed lines corresponding to the graph of $t \mapsto \sqrt{R(t_0) - 2(t - t_0)}$ for $t_0 \in \{ 0, 0.01, 0.025, 0.05\}$ and $R(t_0)$ is the radius obtained for the initially noisy shape (\textcolor{ColorImplicit}{red} dots). 
\label{figFlowerNoise}}
\end{figure}

\subsection{Singular evolutions in the plane} \label{sec:singular}

As we are dealing with point clouds, it is very easy to deal with changes of topologies, especially triple points arising when curves merge are quite naturally captured. Moreover, we know from \cite{BuetLeonardiMasnou} that the approximate curvature $H^\Pi (x_0, V)$ defined in \eqref{eq:projMeanCurvature} is not only consistent in the smooth context, but as well in presence of singular curvature when $\Pi = \Pi_S$ (i.e. $\Pi_{ij}= \Pi_{P_j}$).
Even though there  is no such rigorous consistency property for $\Pi \in \{\Pi_{(T_{x_0} M)^\perp} \circ \Pi_S, 2 \rm{Id}, 2 \Pi_{(T_{x_0} M)^\perp} \}$ (i.e. $\Pi_{ij} \in \{ \Pi_{(P_i)^\perp} \circ \Pi_{P_j}, 2\rm{Id}, 2\Pi_{(P_i)^\perp} \}$),
we compare in Figure~\ref{fig:curvatureAtJunctions} the behaviour of those operators in the presence of a singularity.
More precisely, we consider a junction of three infinite half lines meeting at $0$ and we focus on the values of the approximate mean curvature computed in a neighbourhood of $0$. The lines are uniformly discretized to  define a point cloud varifold $V$ with all masses $m_i$ equal and we associate with each point its exact tangent direction $P_i$ : one of the three possible directions (notice that there is no point exactly at the junction so that $P_i$ is well--defined). In the computation each neighbourhood contains exactly $60$ points, corresponding to a disk of radius $\epsilon \simeq 0.20$ near the singular point $0$. We plot the approximate mean curvature vector computed at each point and  the norm of this vector is color coded according to the colorbar on the right of each plot while the arrows indicate the direction. Notice that the  length of the arrows is rescaled to improve the readability of the plot.  The plot is centered at the singularity $0$ and the viewing window is scaled such that solely the $60$ points closest to the singularity are visualized. In particular the points that are outside do not interfer with the singularity in the computations.
We recall that the expected singular curvature of such a junction is $- (u_1 + u_2 + u_3) \delta_0$ where $u_1, u_2, u_3$ are the unit vectors pointing in the directions of the half-lines (see Example~\ref{ex:junction}).

\noindent In Figure~\ref{fig:curvatureAtJunctions}$(a)$--$(d)$, the three half-lines meet with $120$ degrees angles, forming a  \emph{triple point} having $0$ generalized mean curvature (see computations in Example~\ref{ex:junction}). We observe that projecting onto the normal $P_i^\perp$ (Figure~\ref{fig:curvatureAtJunctions}$(b)$ and $(c)$) ensures that the approximate mean curvature is $0$ up to a very small error $10^{-11}$ while the error term (see maximal intensity on the colorbar) is larger $10^{-3}$ in Figure~\ref{fig:curvatureAtJunctions}$(a)$ even in this simple situation. In Figure~\ref{fig:curvatureAtJunctions}$(d)$, there is a tangential component attracting points to the junction point.  Furthermore, the arrows are aligned with the lines and point towards $0$.
In Figure~\ref{fig:curvatureAtJunctions}$(e)$--$(h)$, the three half-lines meet with different angles and  we expect some singular curvature to be observed, which can best be seen  in Figure~\ref{fig:curvatureAtJunctions}$(e)$.
\begin{center}
\setcounter{subfigure}{0}
\begin{figure}[!htp] 
\subcaptionbox{$\Pi_{ij} = \Pi_{P_j}$}{\includegraphics[width=0.24\textwidth]{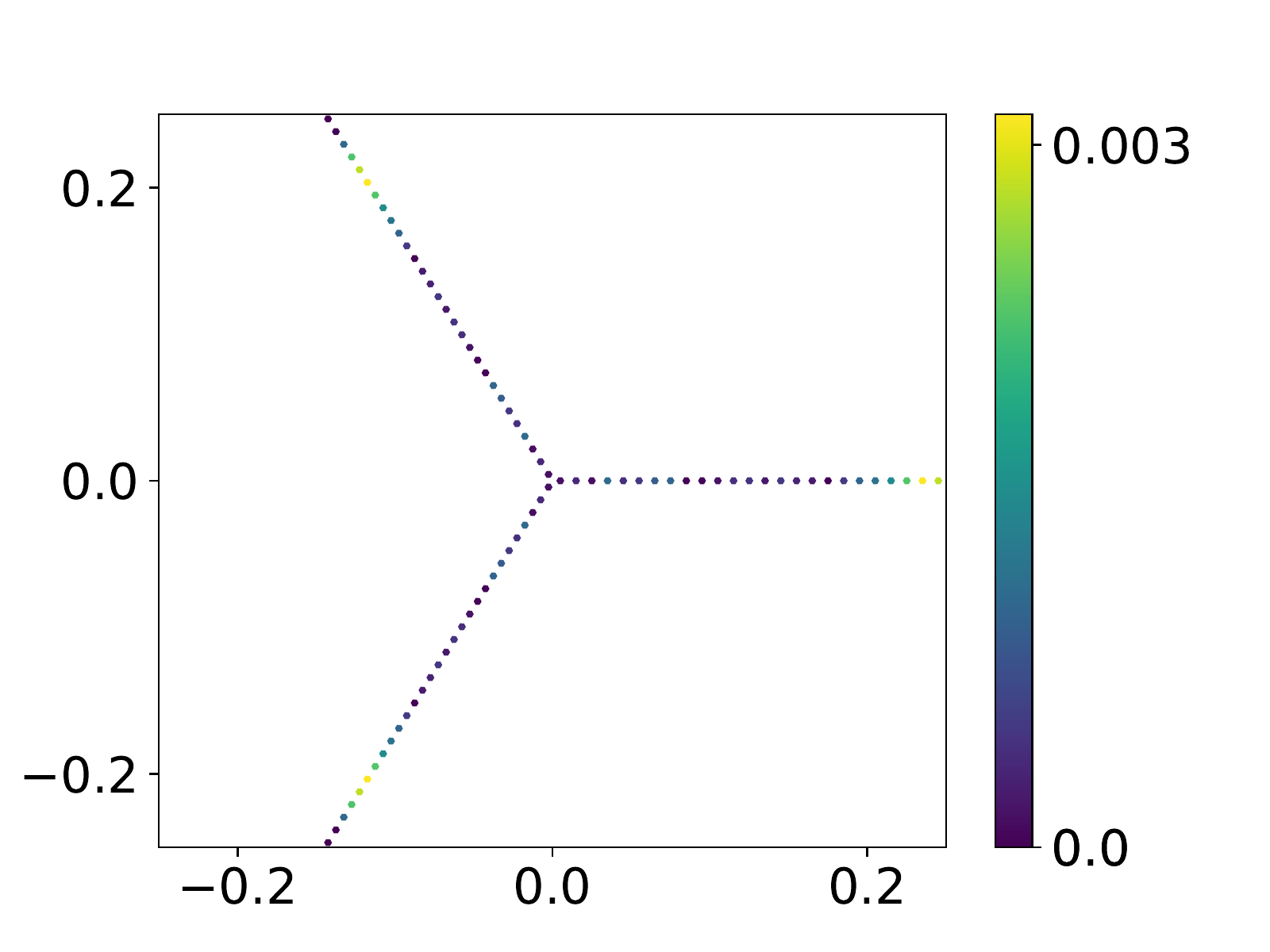}}
\subcaptionbox{$\Pi_{ij} = \Pi_{P_i^\perp} \circ \Pi_{P_j}$}{\includegraphics[width=0.24\textwidth]{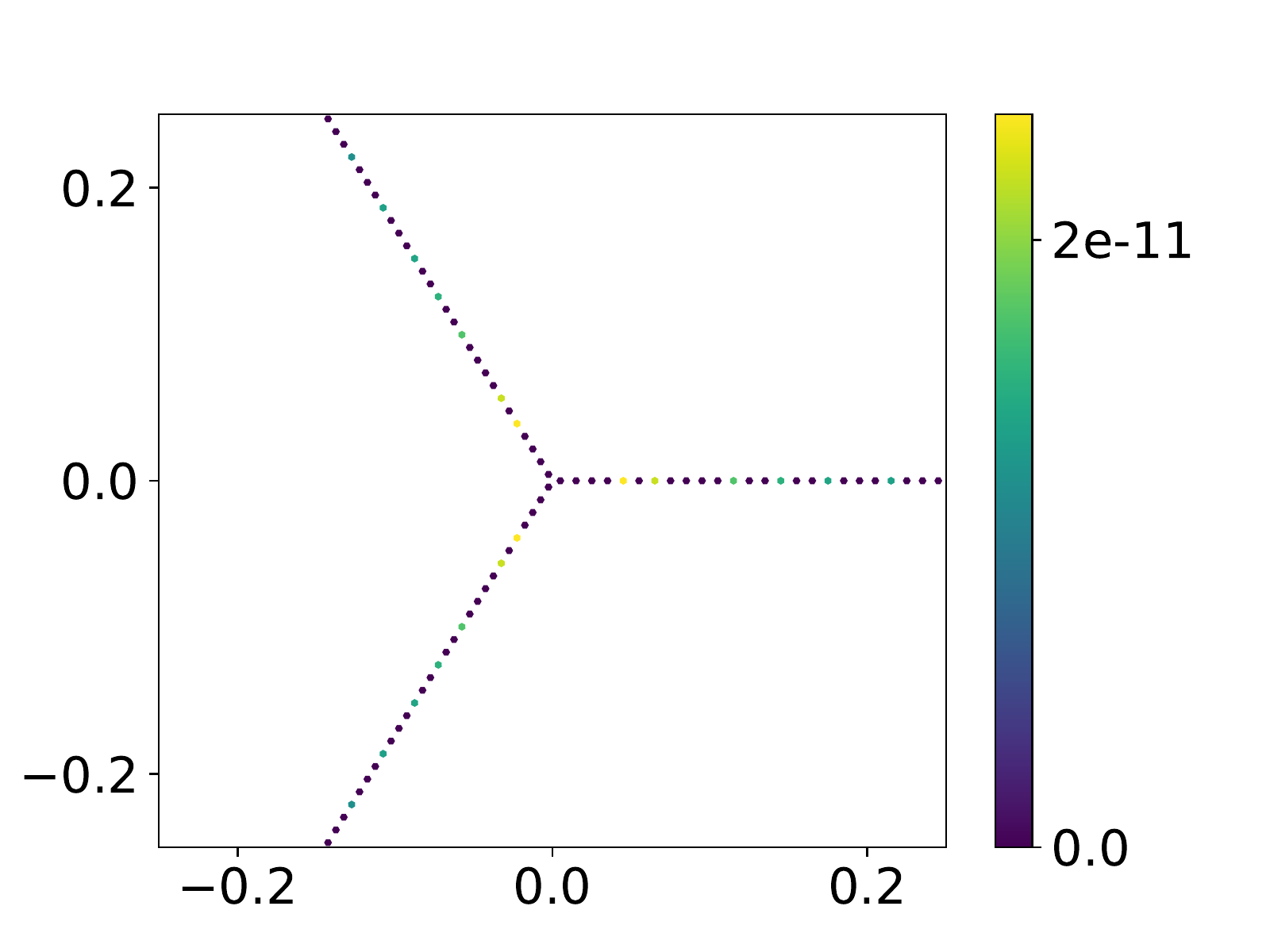}}
\subcaptionbox{$\Pi_{ij} = 2\Pi_{P_i^\perp}$}{\includegraphics[width=0.24\textwidth]{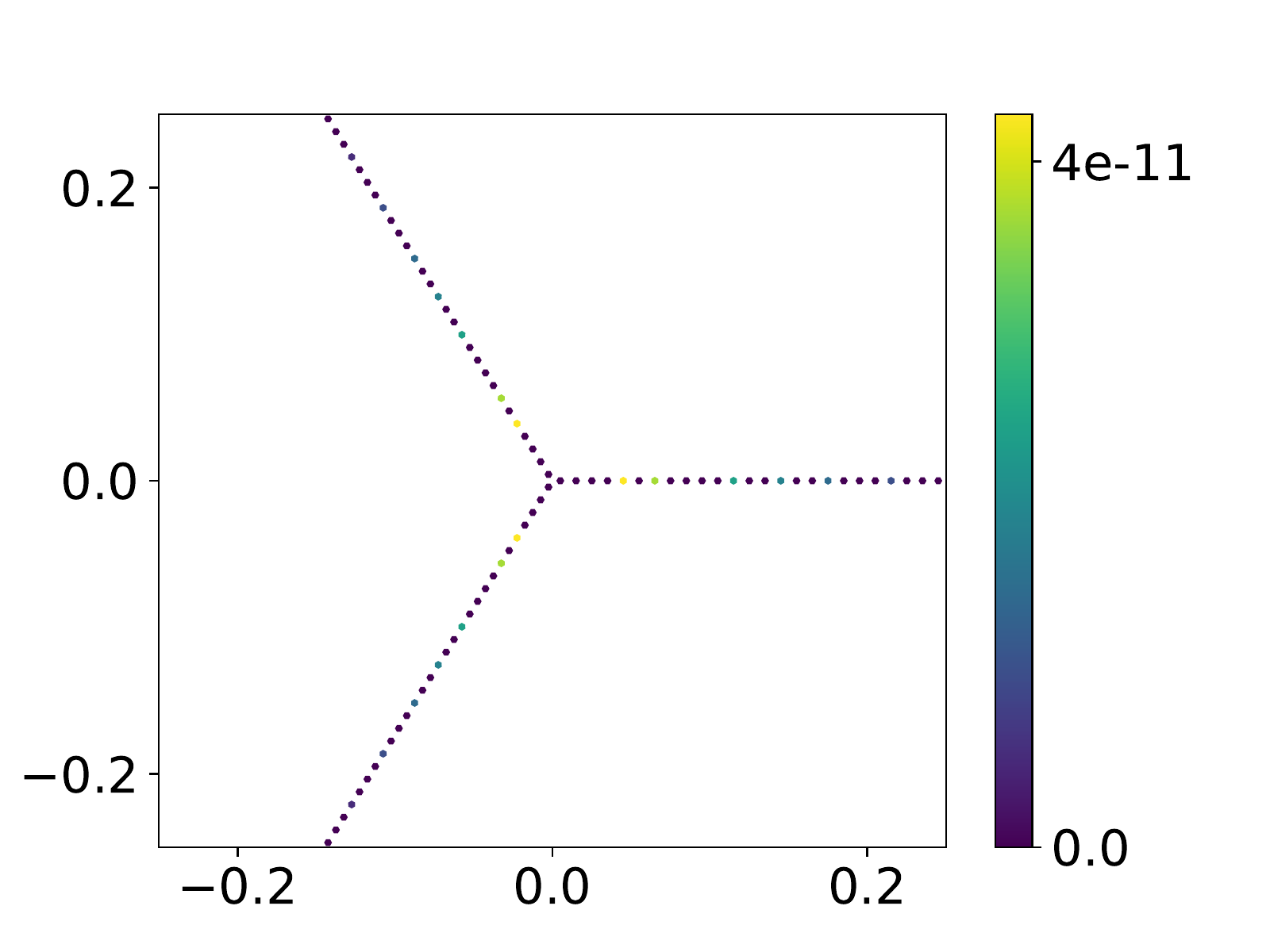}}
\subcaptionbox{$\Pi_{ij} = 2 \rm{Id}$}{\includegraphics[width=0.24\textwidth]{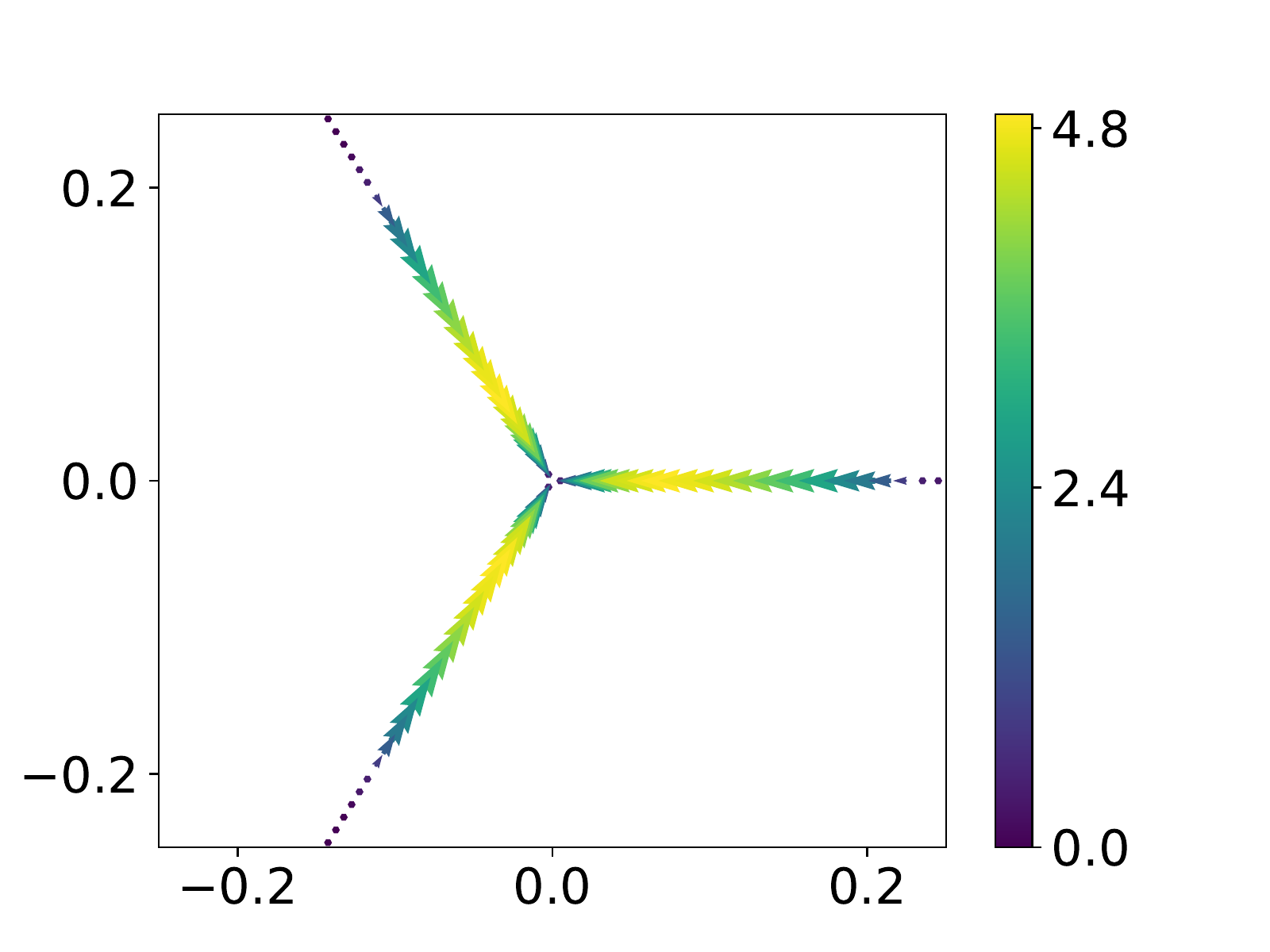}}
\subcaptionbox{$\Pi_{ij} = \Pi_{P_j}$}{\includegraphics[width=0.24\textwidth]{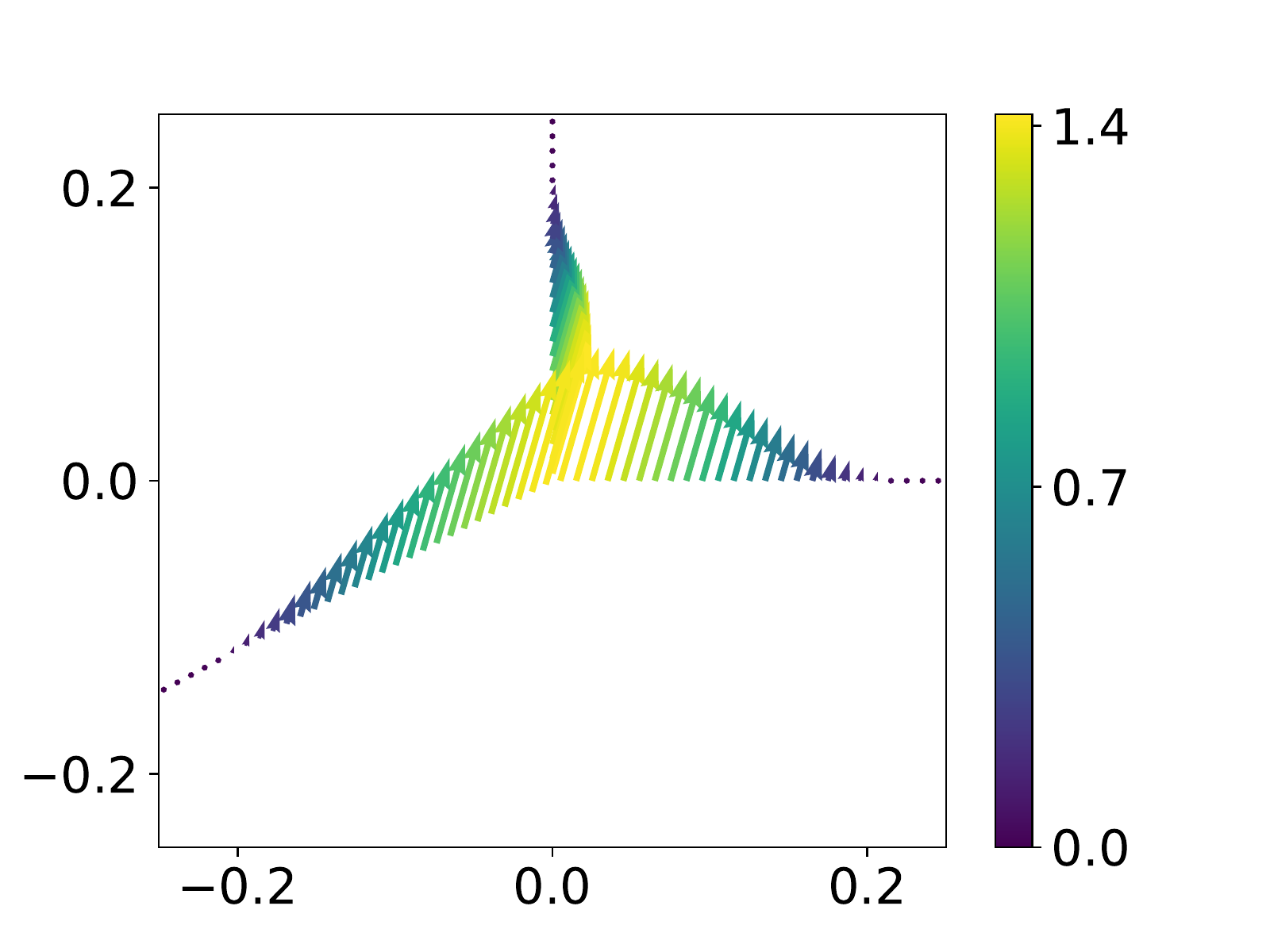}}
\subcaptionbox{$\Pi_{ij} = \Pi_{P_i^\perp} \circ \Pi_{P_j}$}{\includegraphics[width=0.24\textwidth]{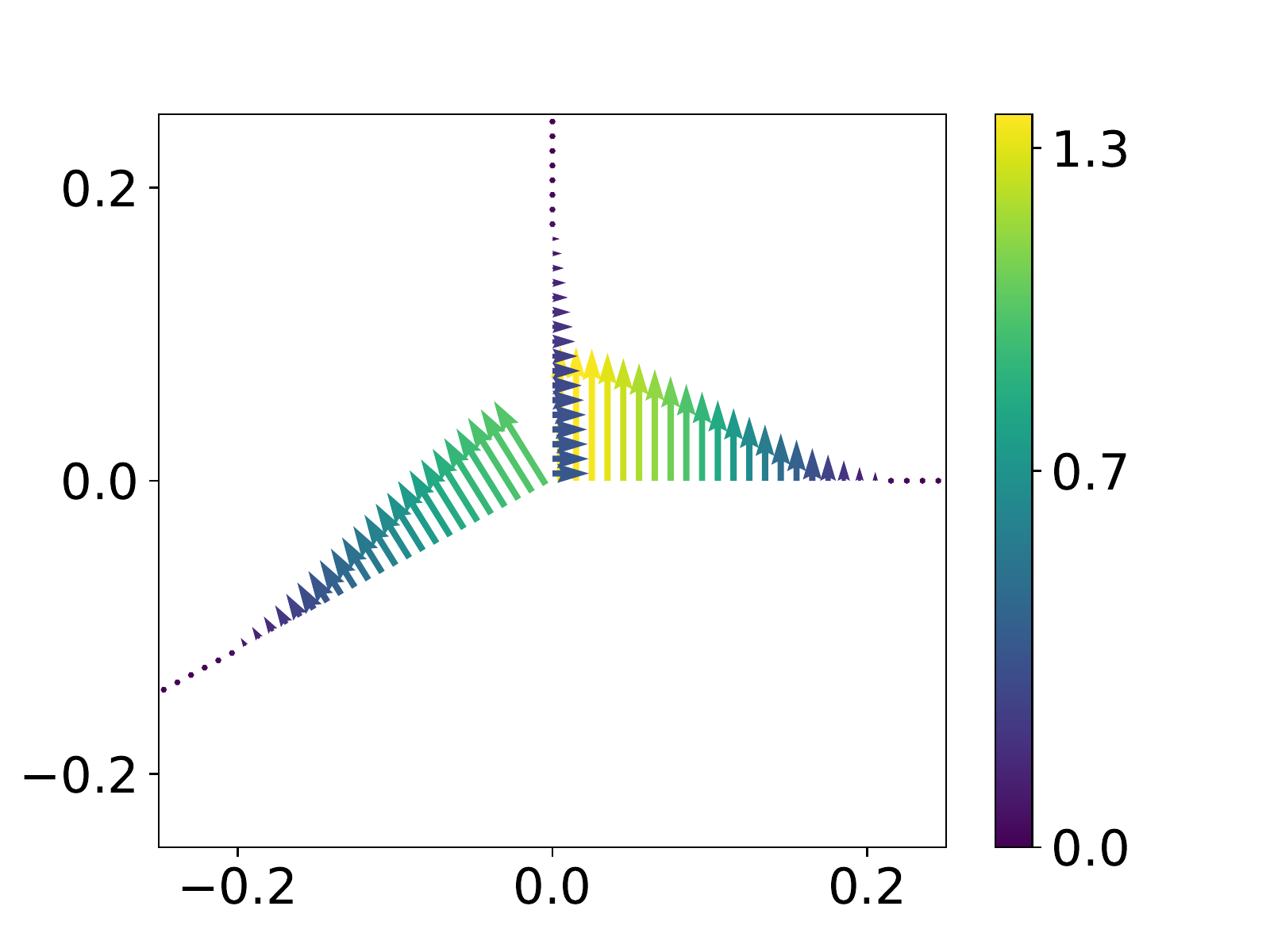}}
\subcaptionbox{$\Pi_{ij} = 2\Pi_{P_i^\perp}$}{\includegraphics[width=0.24\textwidth]{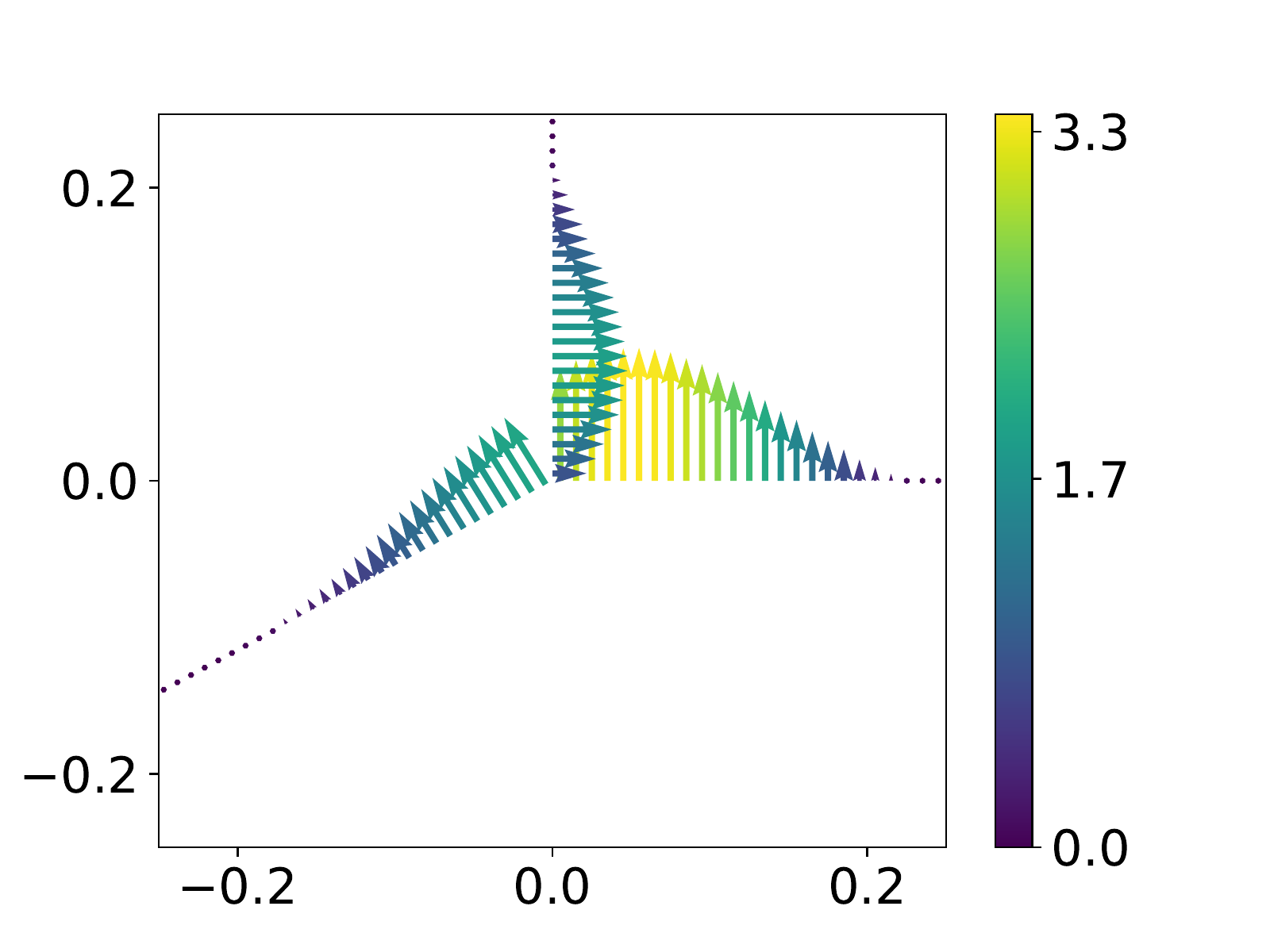}}
\subcaptionbox{$\Pi_{ij} = 2 \rm{Id}$}{\includegraphics[width=0.24\textwidth]{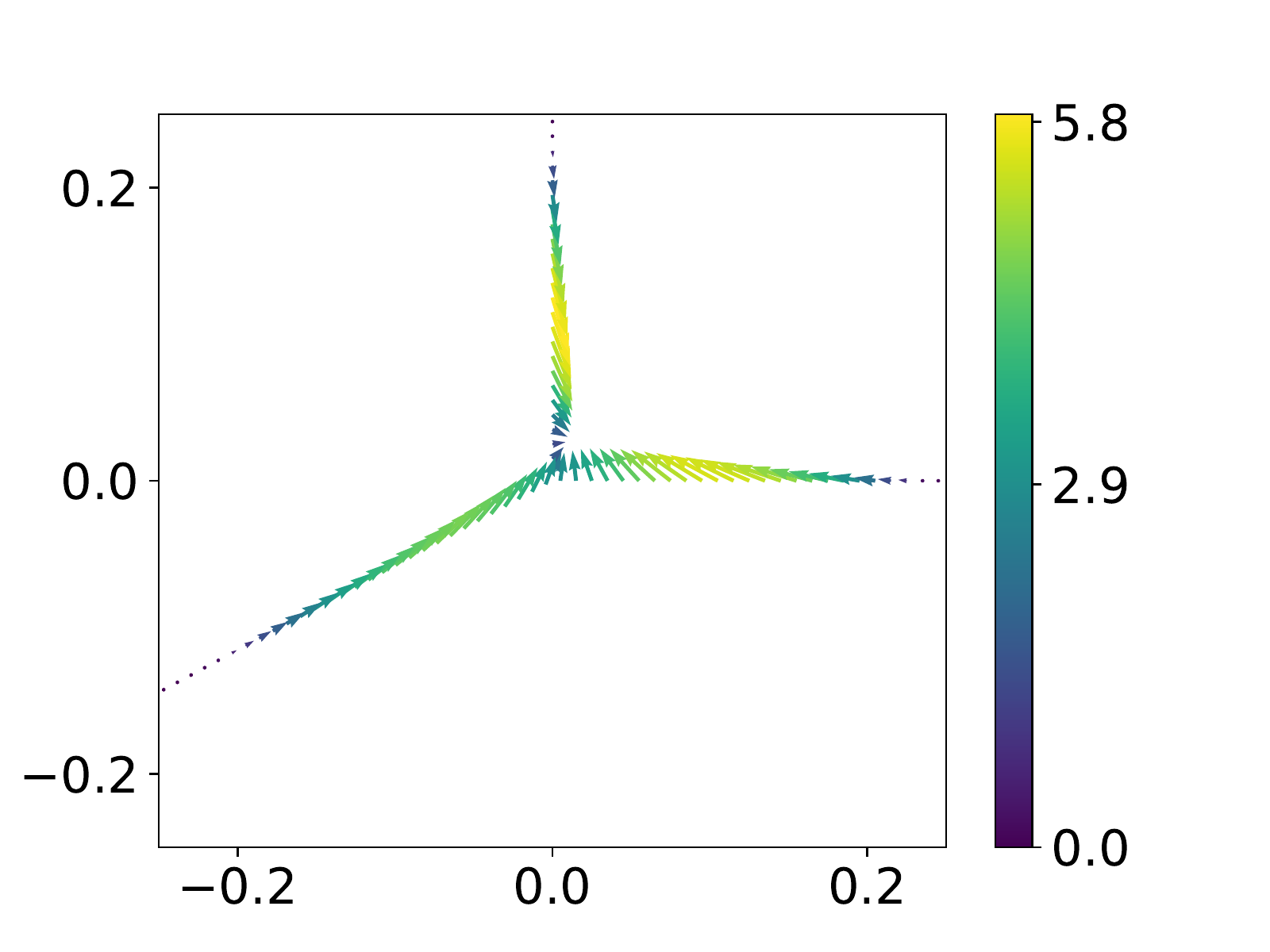}}
\caption{Computed mean curvature vector at junction point depending on the chosen projection operator. \label{fig:curvatureAtJunctions}}
\end{figure}
\end{center}

Unfortunately, choosing $\Pi_{ij} = \Pi_{P_j}$ leads to strong unstabilities of the curvature flow as noted previously, see Figure~\ref{figComparisonCircleOperatorsNoise} and we carry on with $\Pi_{ij} = 2 \Pi_{P_i^\perp}$.
It is then possible to perform mean curvature flow even after the creation of singularities. We propose some examples in $2D$ of such evolutions. 
In Figure~\ref{figSingDoubleCircle}, we perform a test on two crossing circles. We observe that both crossing points split up into two triple points almost instantaneously.  The different curve segments then merge until they form a single circular curve, which then follows the usual evolution. In this evolution, the circles are discretized with a total of $N = 1000$ points. The number of points used for computing curvature is  $k_\epsilon = 31$, for the tangent is $k_\sigma = 15$ and for the mass is $k_\delta = 7$. The time step is set to $\tau = 1/4N = 2.5 \cdot 10^{-4}$.
In Figure~\ref{figSingTripleCircle}, we perform a similar test on three crossing circles, which allows to observe  several mergers and the formation of multiple triple points. The circles are discretized with a total of $N = 1200$ points. The number of points used for computing curvature is  $k_\epsilon = 31$, for the tangent is $k_\sigma = 15$ and for the mass is $k_\delta = 7$. The time step is set to $\tau = 1/2N \simeq 4.16 \cdot 10^{-4}$. In Figure~\ref{figSingTripleCircle}$(g)$,  a zoom is shown at time $0.3$ to provide a better visualization of the triple point configuration.
\begin{center}
\setcounter{subfigure}{0}
\begin{figure}[!htp]
\subcaptionbox{time $0$}{\includegraphics[width=0.16\textwidth]{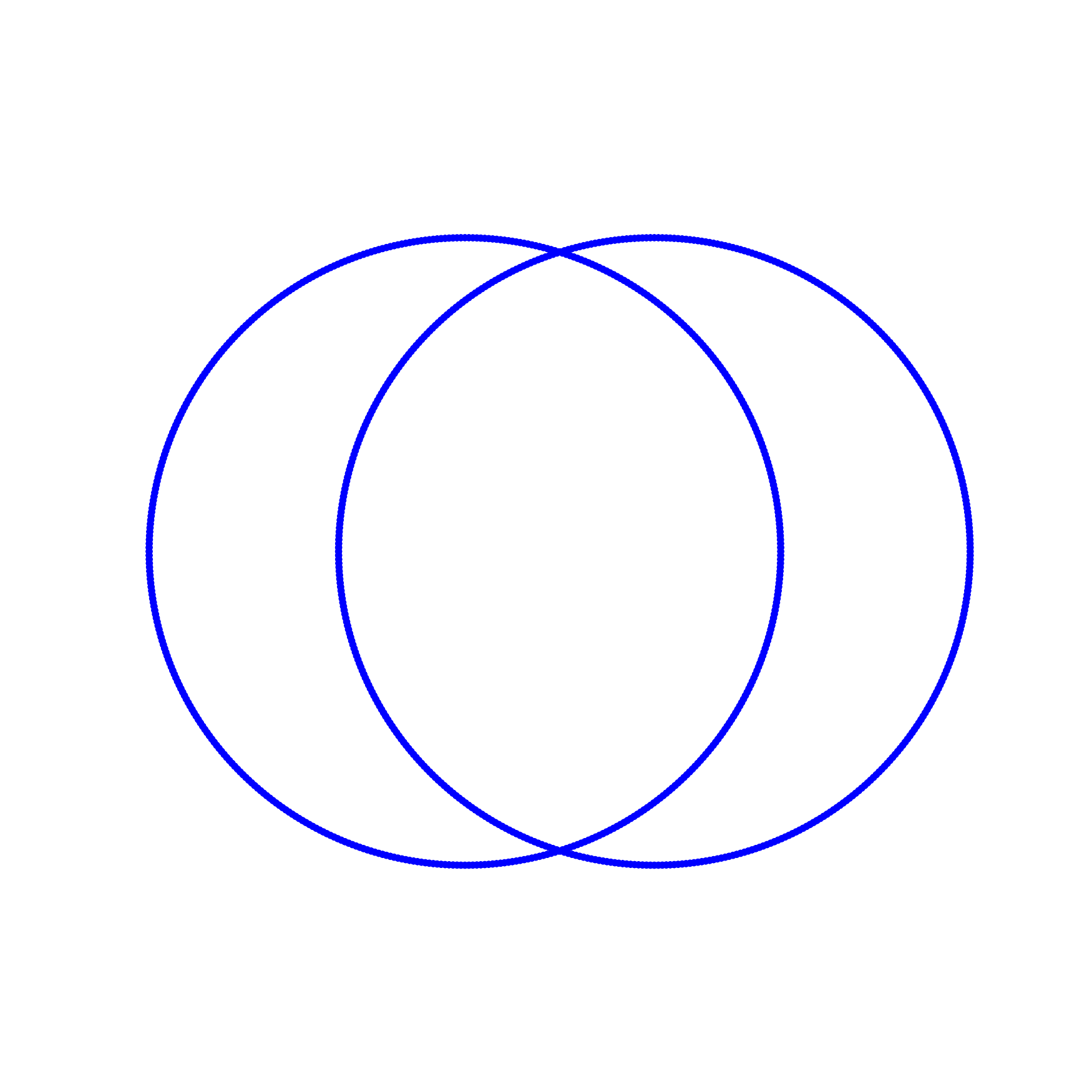}}
\subcaptionbox{time $0.003$}{\includegraphics[width=0.16\textwidth]{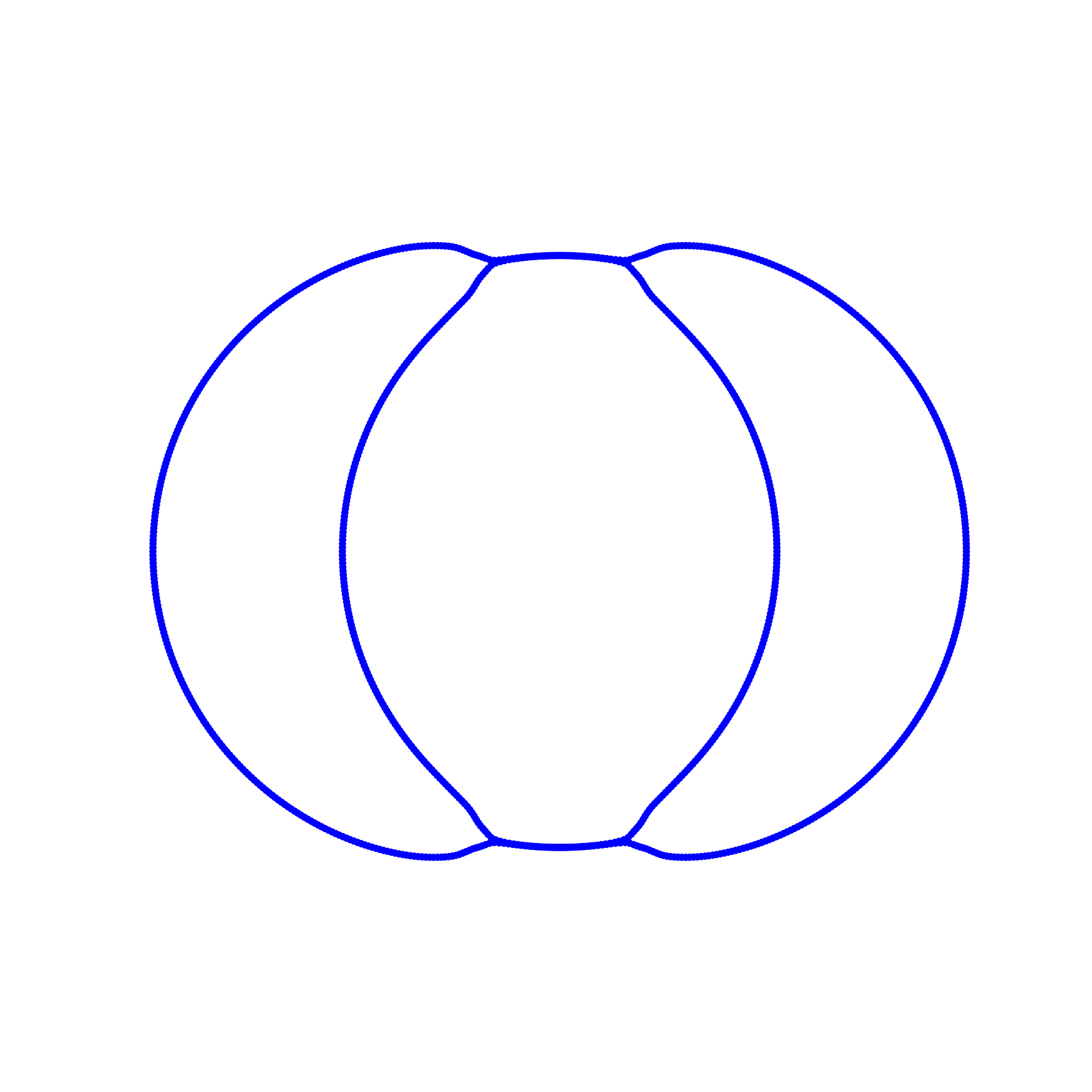}}
\subcaptionbox{time $0.03$}{\includegraphics[width=0.16\textwidth]{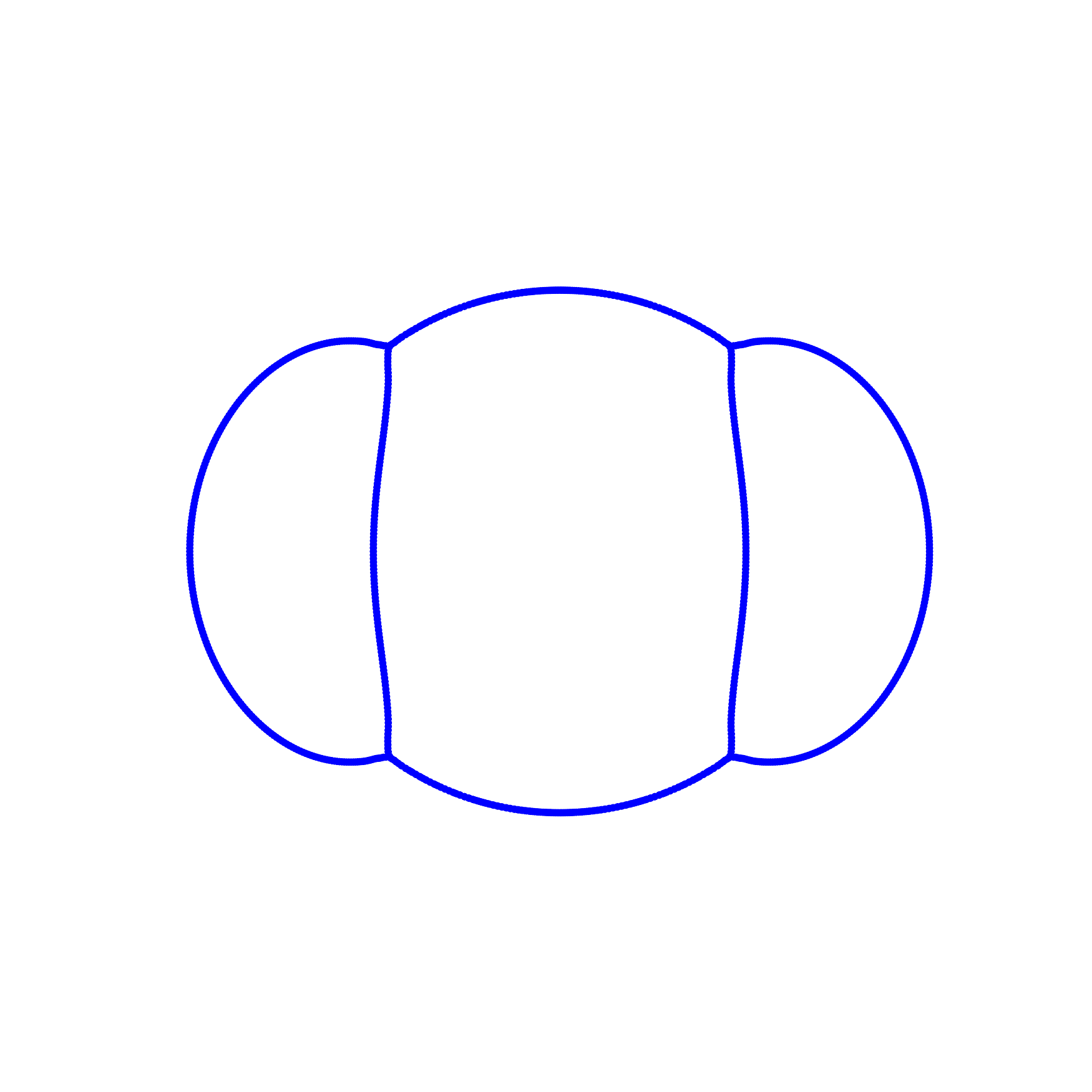}}
\subcaptionbox{time $0.05$}{\includegraphics[width=0.16\textwidth]{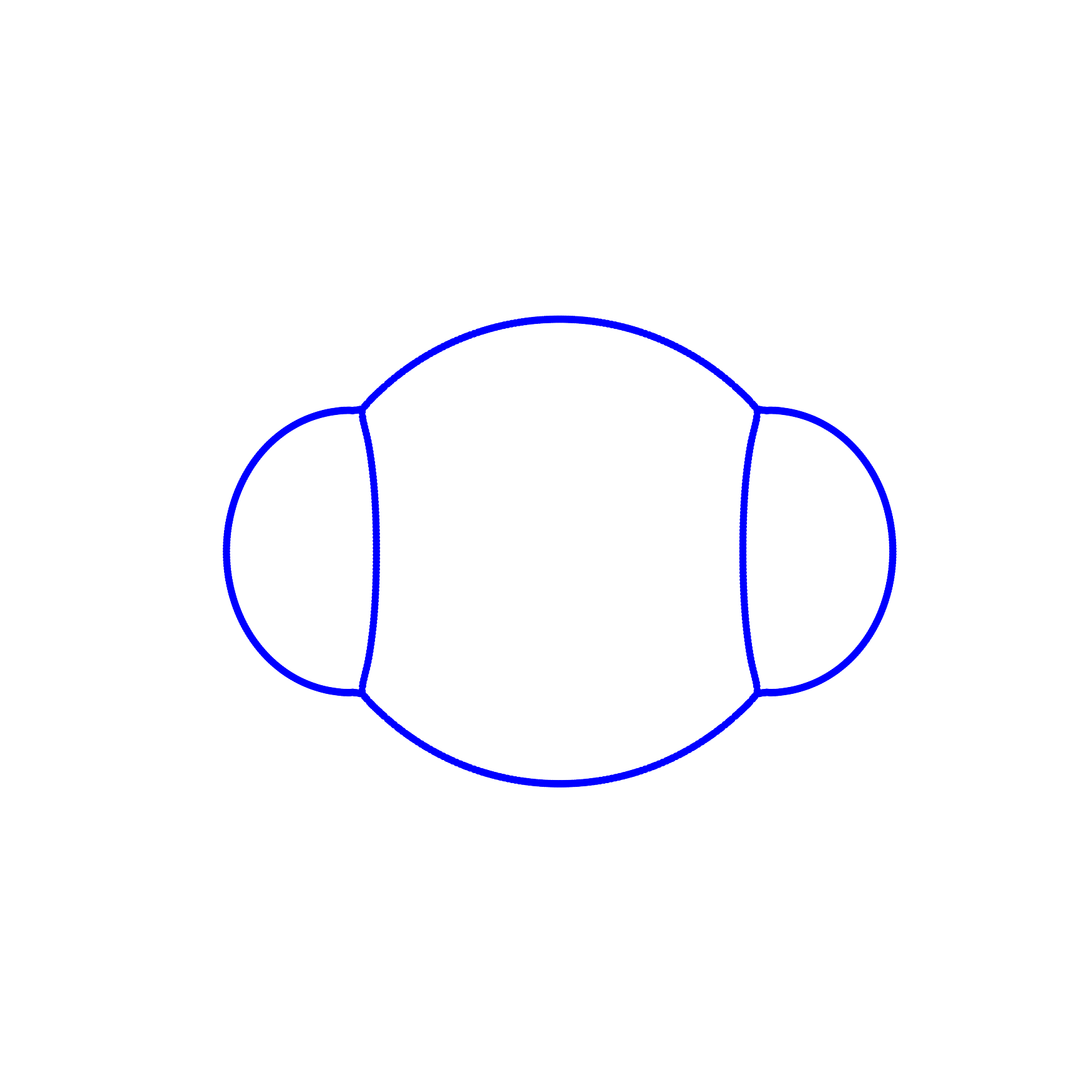}}
\subcaptionbox{time $0.07$}{\includegraphics[width=0.16\textwidth]{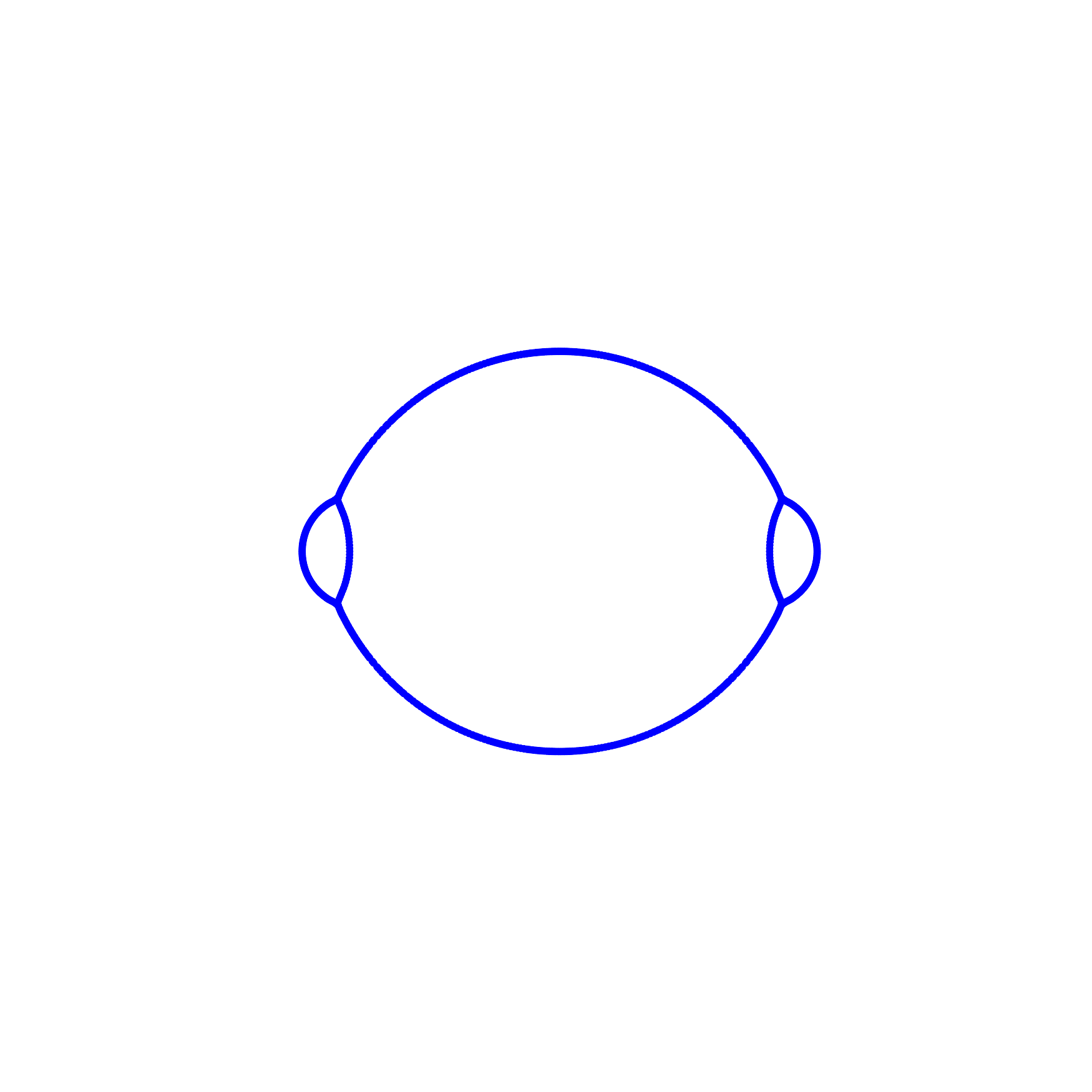}}
\subcaptionbox{time $0.11$}{\includegraphics[width=0.16\textwidth]{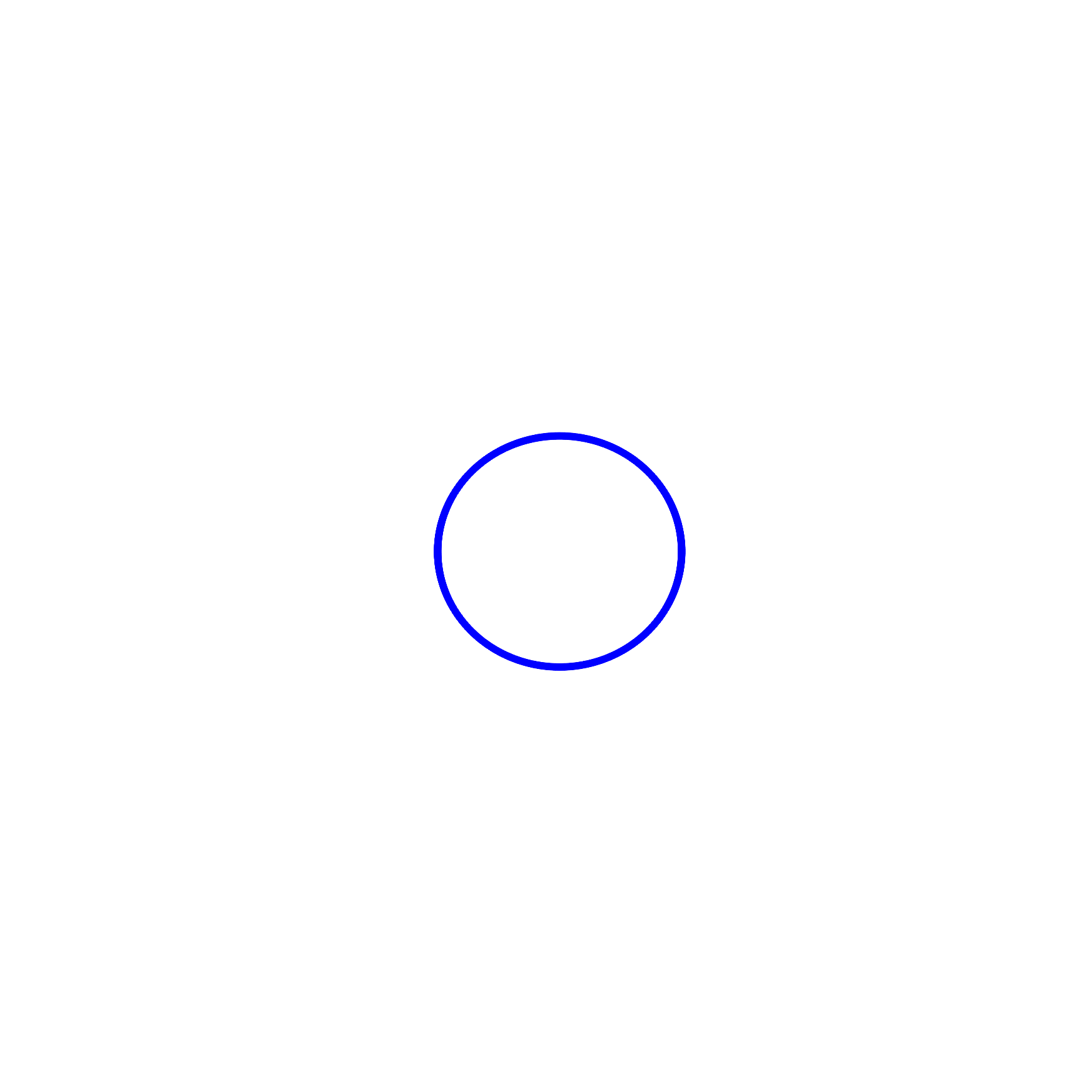}}
\caption{Evolution of two crossing circles under discrete curvature flow. \label{figSingDoubleCircle}}
\end{figure}
\end{center}
\begin{center}
\setcounter{subfigure}{0}
\begin{figure}[!htp]
\begin{minipage}{0.65\textwidth}
\subcaptionbox{time $0$}{\includegraphics[width=0.32\textwidth]{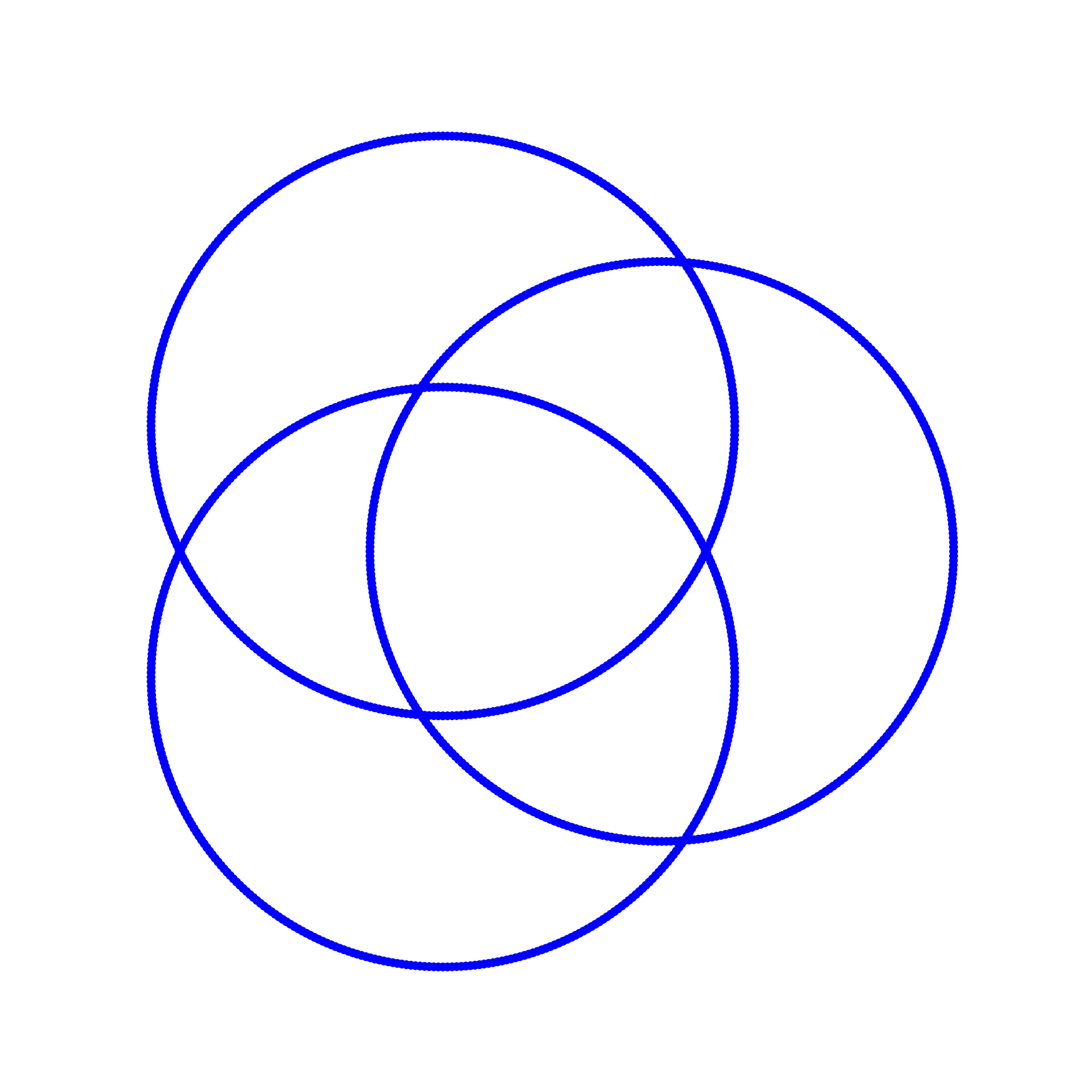}}
\subcaptionbox{time $0.1$}{\includegraphics[width=0.32\textwidth]{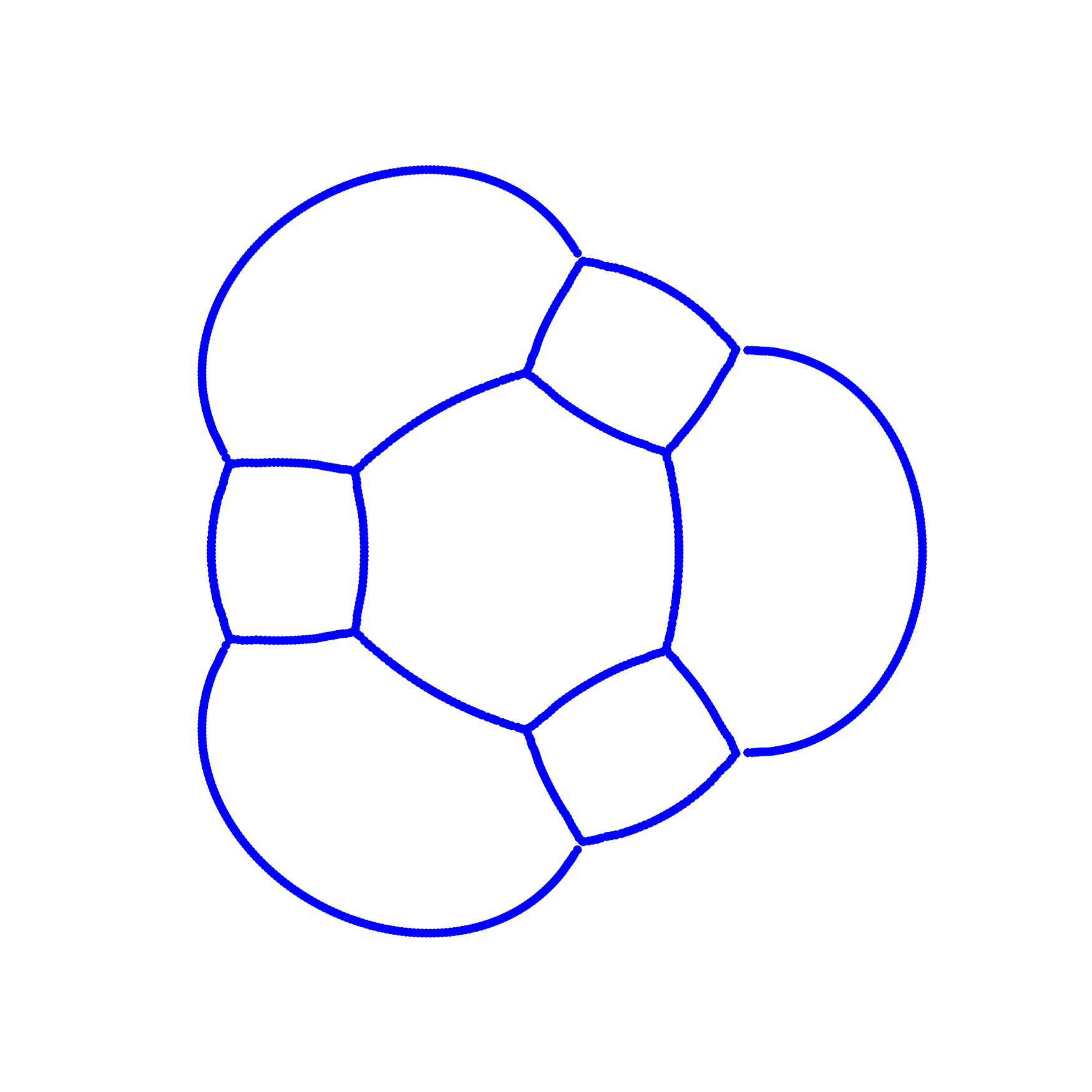}}
\subcaptionbox{time $0.2$}{\includegraphics[width=0.32\textwidth]{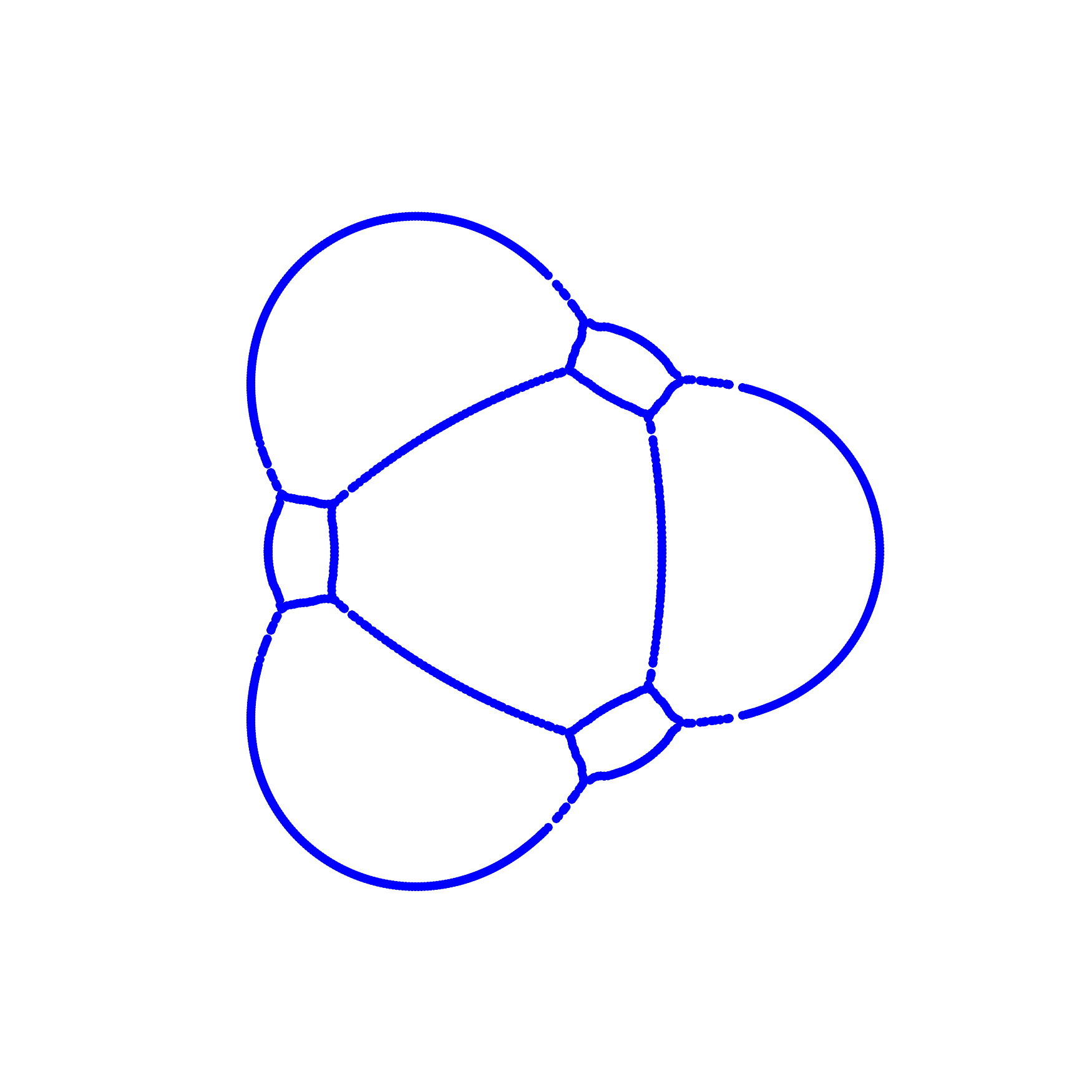}}\\
\subcaptionbox{time $0.3$}{\includegraphics[width=0.32\textwidth]{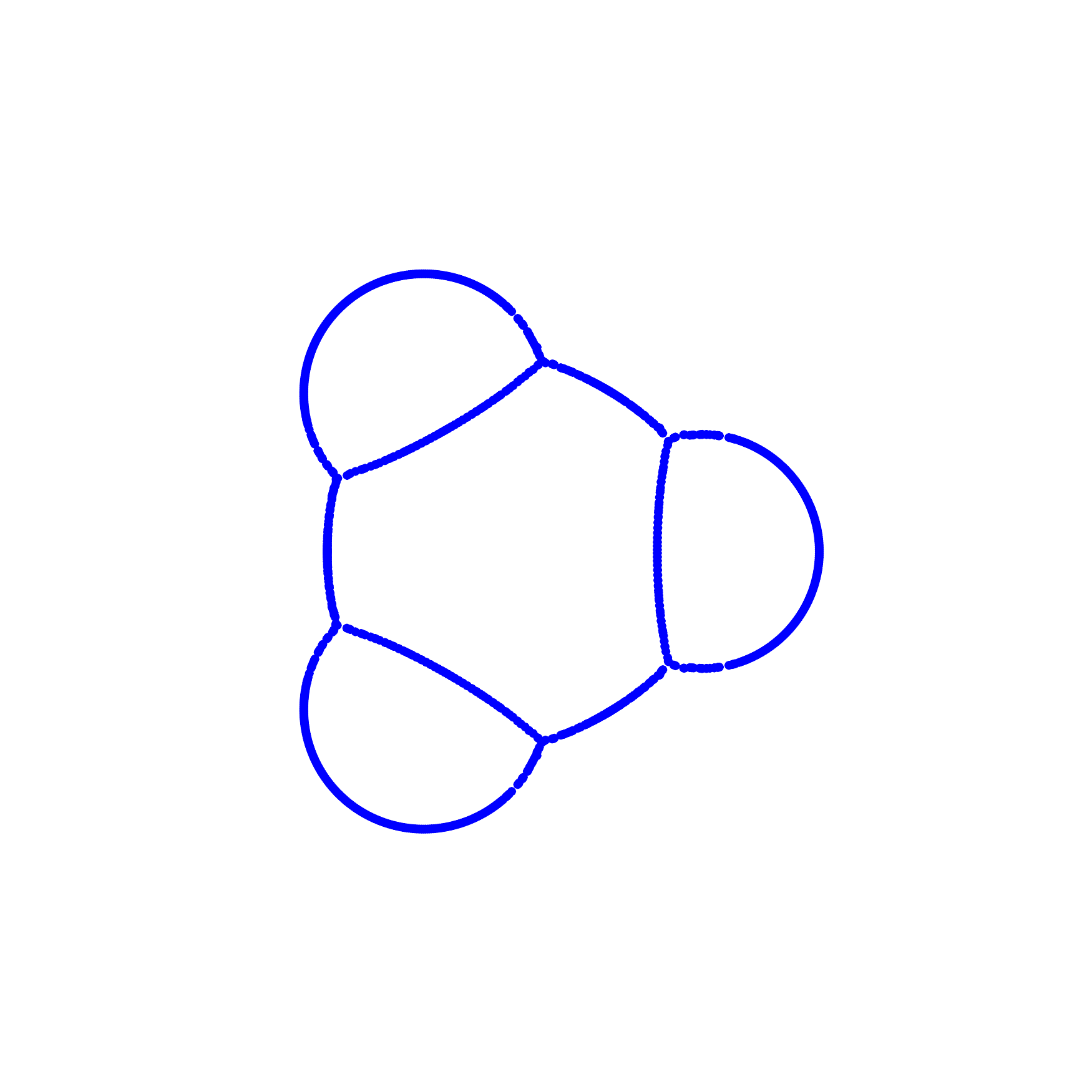}}
\subcaptionbox{time $0.4$}{\includegraphics[width=0.32\textwidth]{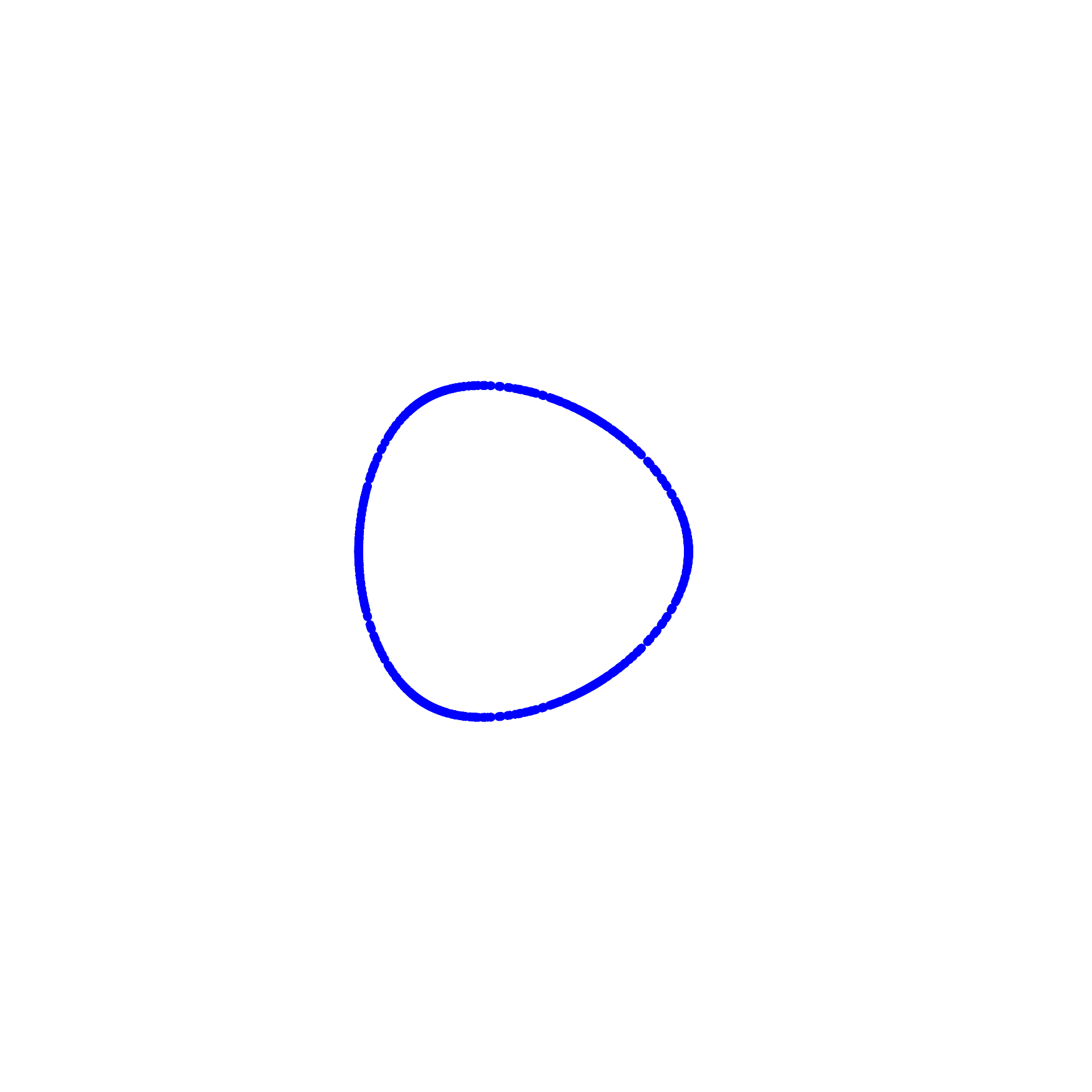}}
\subcaptionbox{time $0.5$}{\includegraphics[width=0.32\textwidth]{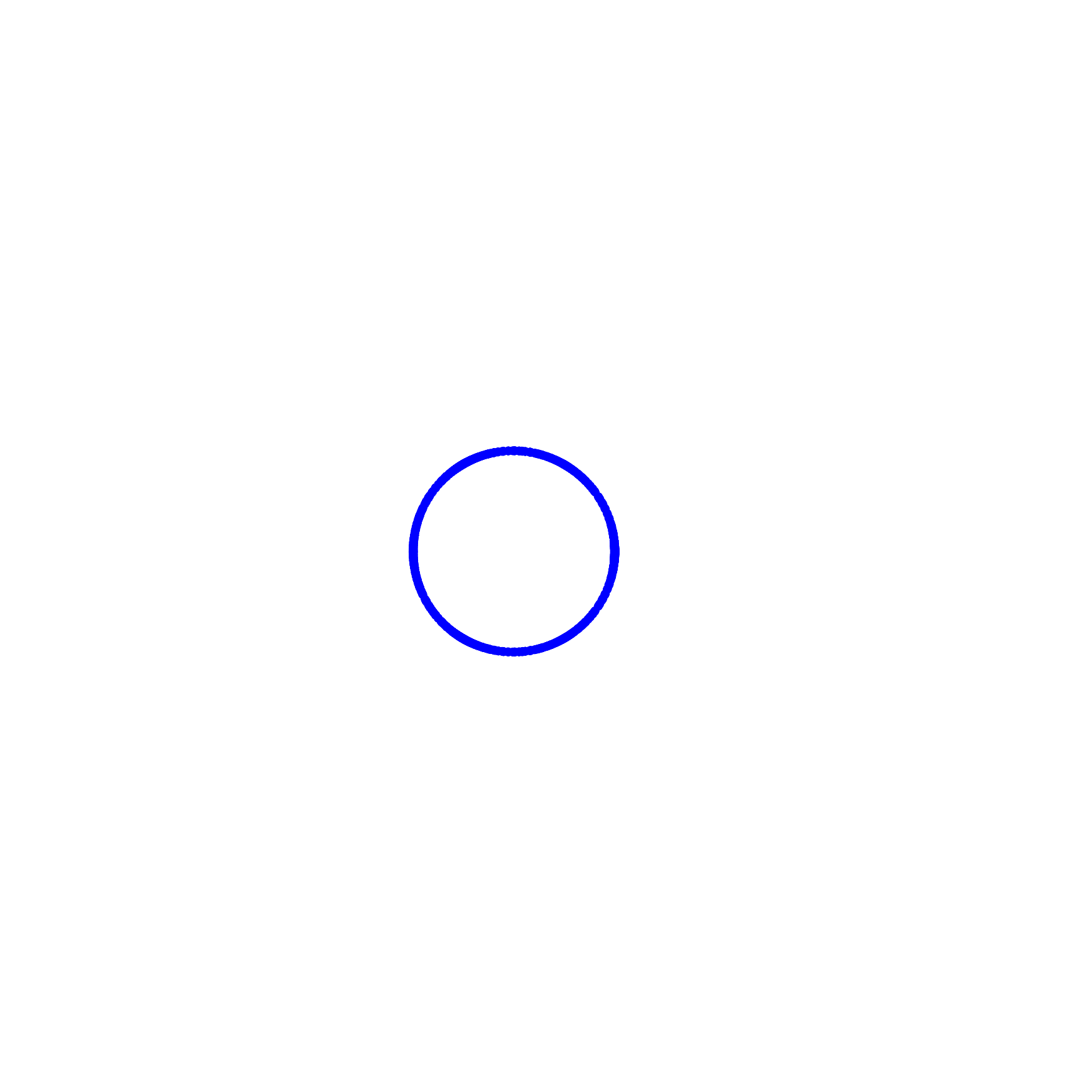}}
\end{minipage}
\begin{minipage}{0.33\textwidth}
\subcaptionbox{time $0.3$ (zoomed)}{\includegraphics[width=0.9\textwidth]{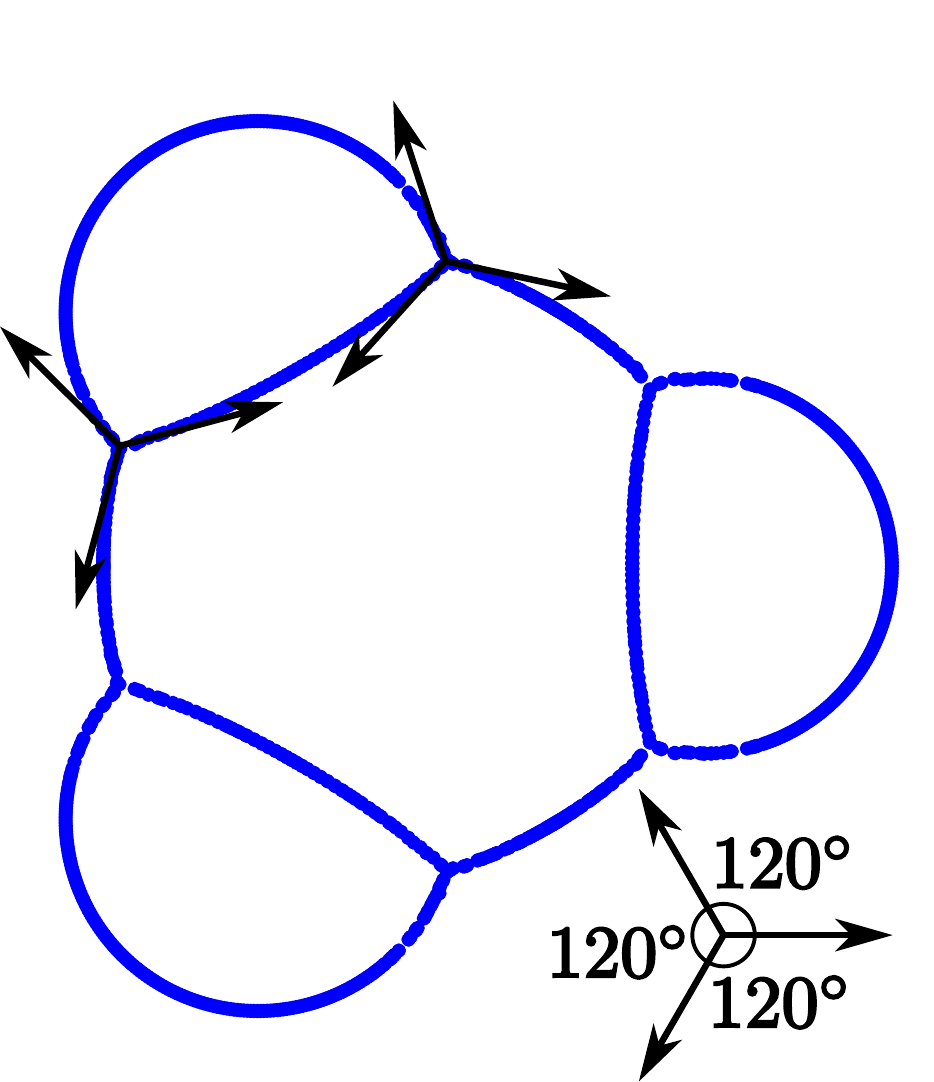}}
\end{minipage}
\caption{Evolution of three circles towards one circle under discrete curvature flow. \label{figSingTripleCircle}}
\end{figure}
\end{center}

In Figure~\ref{figSteiner1},  we consider an initial configuration consisting of a square whose $4$ corners are kept fixed.  Evolving this point cloud under curvature flow, we aim at recovering a shortest path connecting the $4$ corners, in the limit, usually called Steiner tree. There are $2$ shortest paths in this case, because of the symmetry of the configuration of the $4$ points, we hence compare our limit configuration to one of the two Steiner trees (plotted in black in each figure) connecting the $4$ corners.
The experiment is designed as follows, the initial square is discretized with $N = 300$ points. 
The number of points used for computing curvature is $k_\epsilon = 41$ and for computing mass  $k_\delta = 7$.  The number of points used for computing tangent is $k_\sigma = 11$ in the first evolution (Figure~\ref{figSteiner1}$(a)$--$(f)$) and it is $k_\sigma = 15$ in the second evolution (Figure~\ref{figSteiner1}$(g)$--$(l)$).
The time step is fixed to $\tau = \frac{1}{4N}$. 
In order to help connectedness to be preserved, we do not recompute the nearest neighbour graph (i.e. the $kd$--tree structure, see Remark~\ref{remk:knnTopology}) at each step, but only every $25$ iterations. 
The point cloud is depicted at times $0$, $0.1$, $0.2$, $0.3$, $0.4$ and $1$ as indicated in the legend.  The associated Steiner tree is outlined in black
and we observe that slight parameter changes may lead to the other Steiner tree. 
Let us mention here, that specific modulations of the parameters, in particular a too large kernel size for the evaluation of the curvature,
eventually leads to a loss of connectedness of the point cloud and thus one ends up with two segments joining two pairs of opposite corners
or even worse break ups.
\begin{center}
\setcounter{subfigure}{0}
\begin{figure}[!htp]
\subcaptionbox{time $0$}{\includegraphics[width=0.16\textwidth]{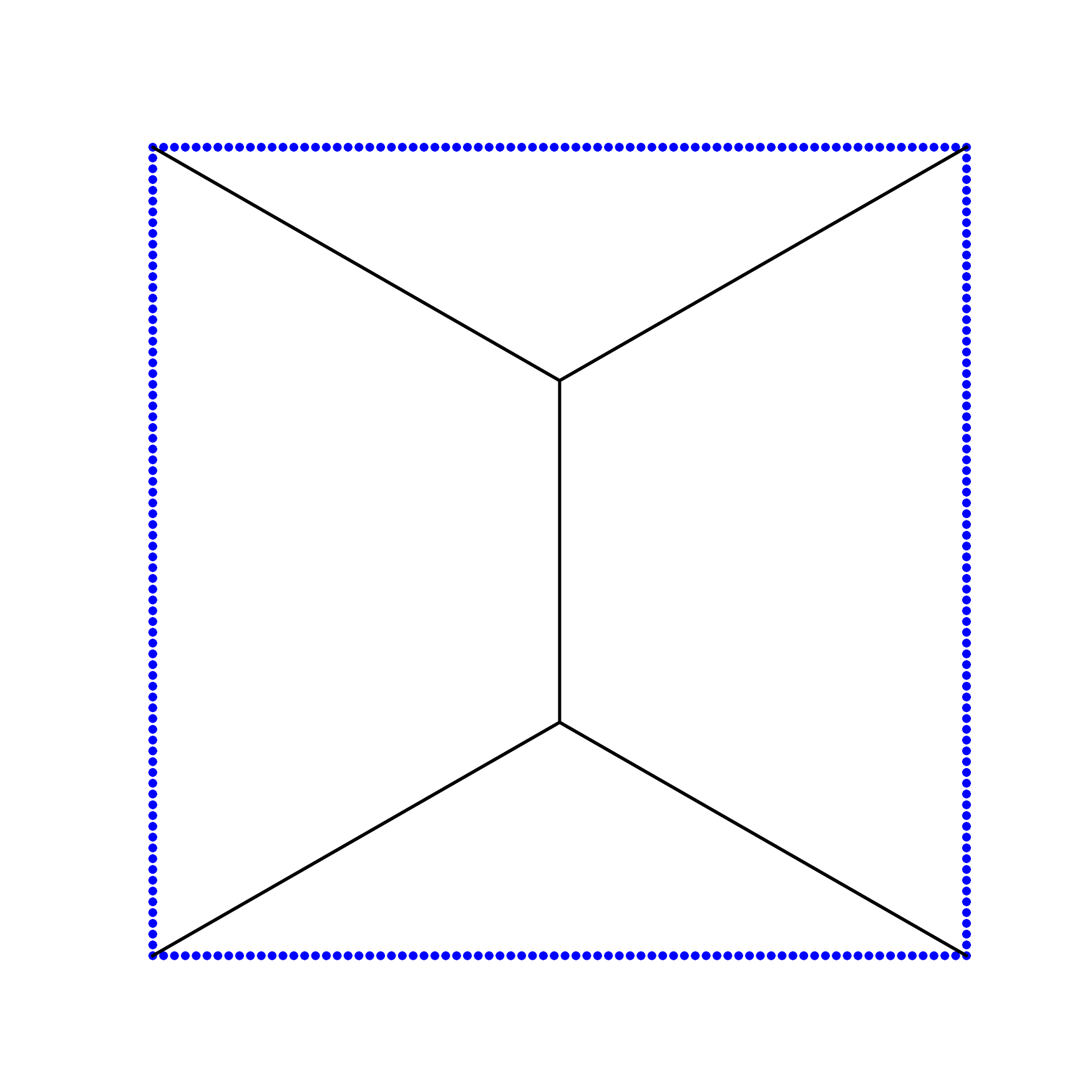}}
\subcaptionbox{time $0.1$}{\includegraphics[width=0.16\textwidth]{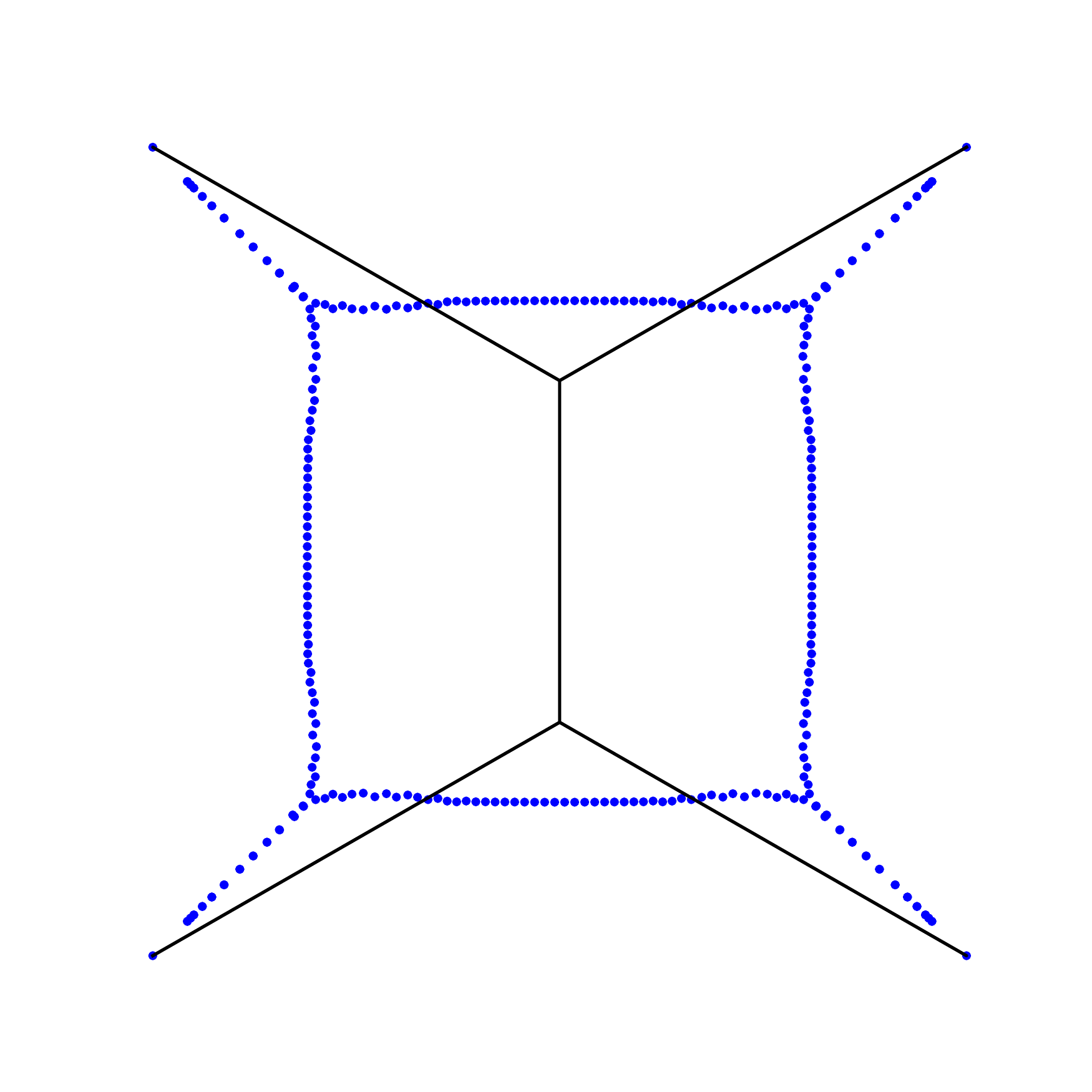}}
\subcaptionbox{time $0.2$}{\includegraphics[width=0.16\textwidth]{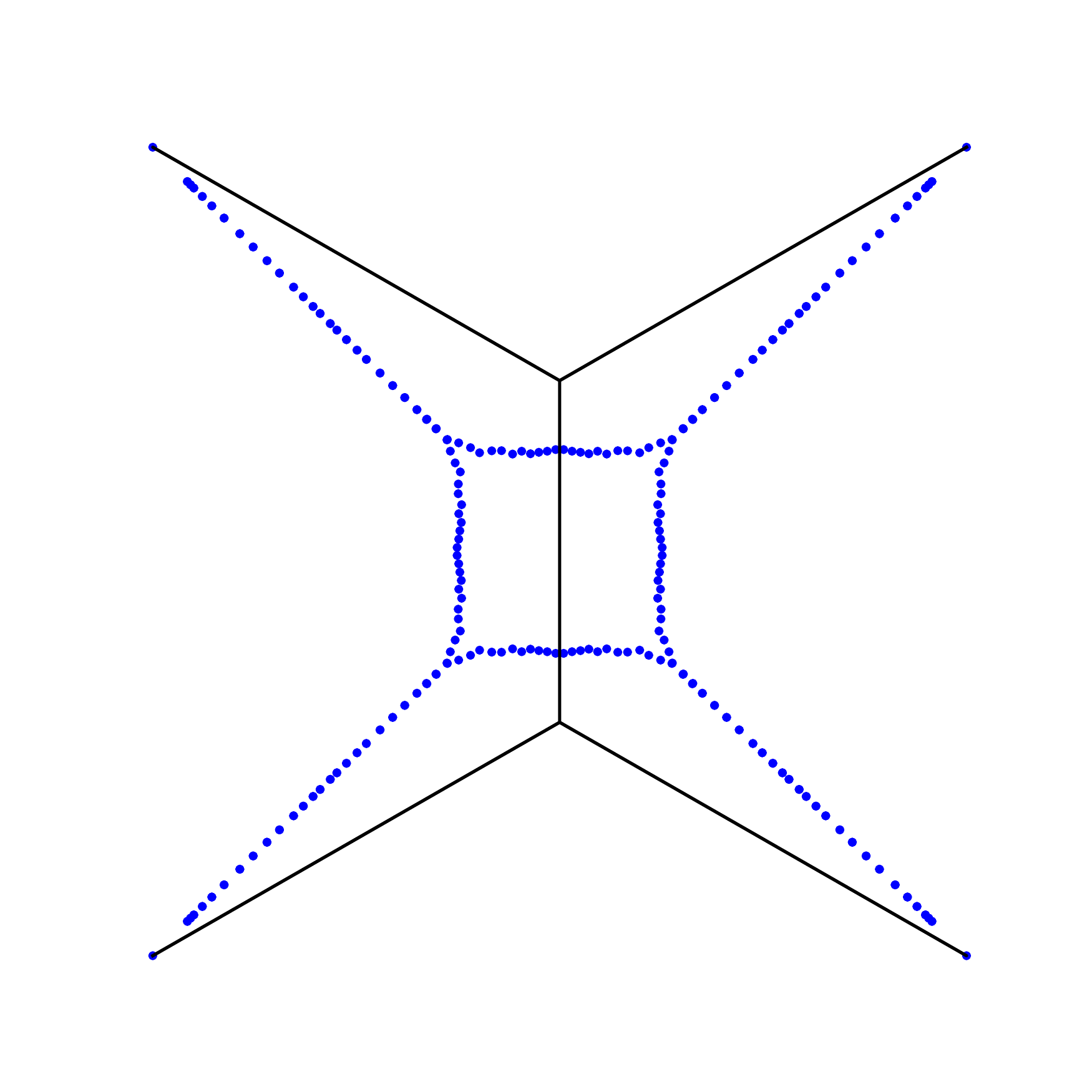}}
\subcaptionbox{time $0.3$}{\includegraphics[width=0.16\textwidth]{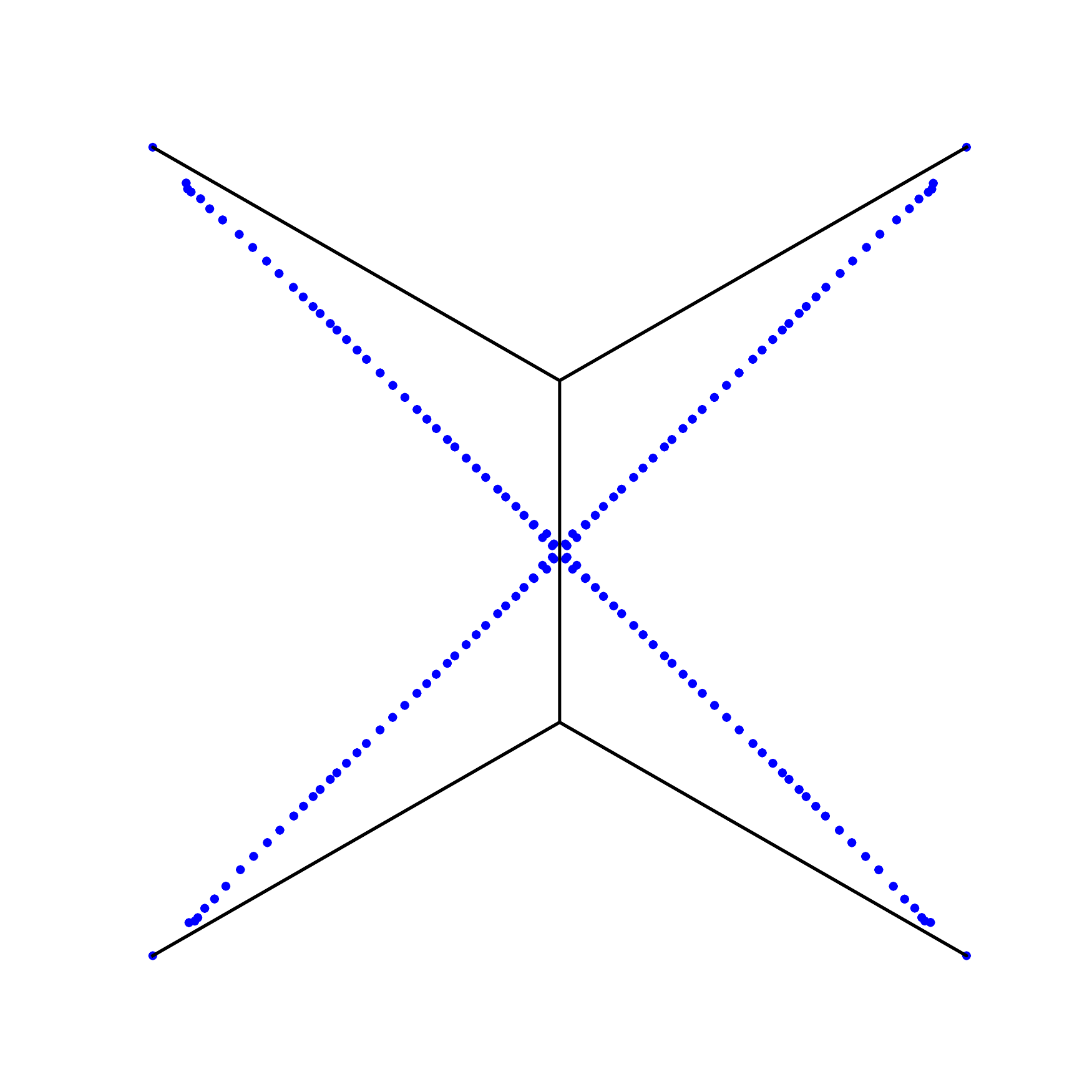}}
\subcaptionbox{time $0.4$}{\includegraphics[width=0.16\textwidth]{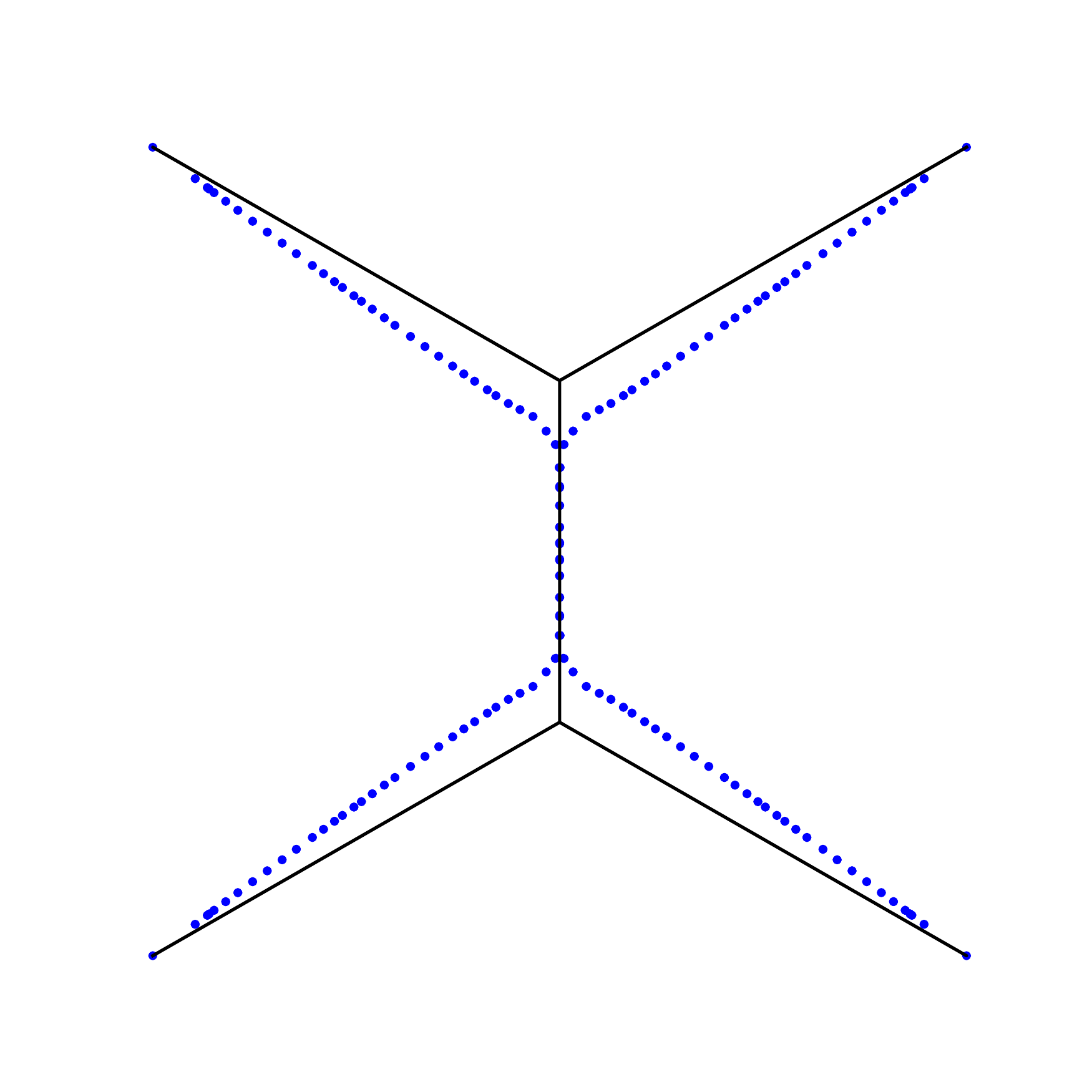}}
\subcaptionbox{time $1$}{\includegraphics[width=0.16\textwidth]{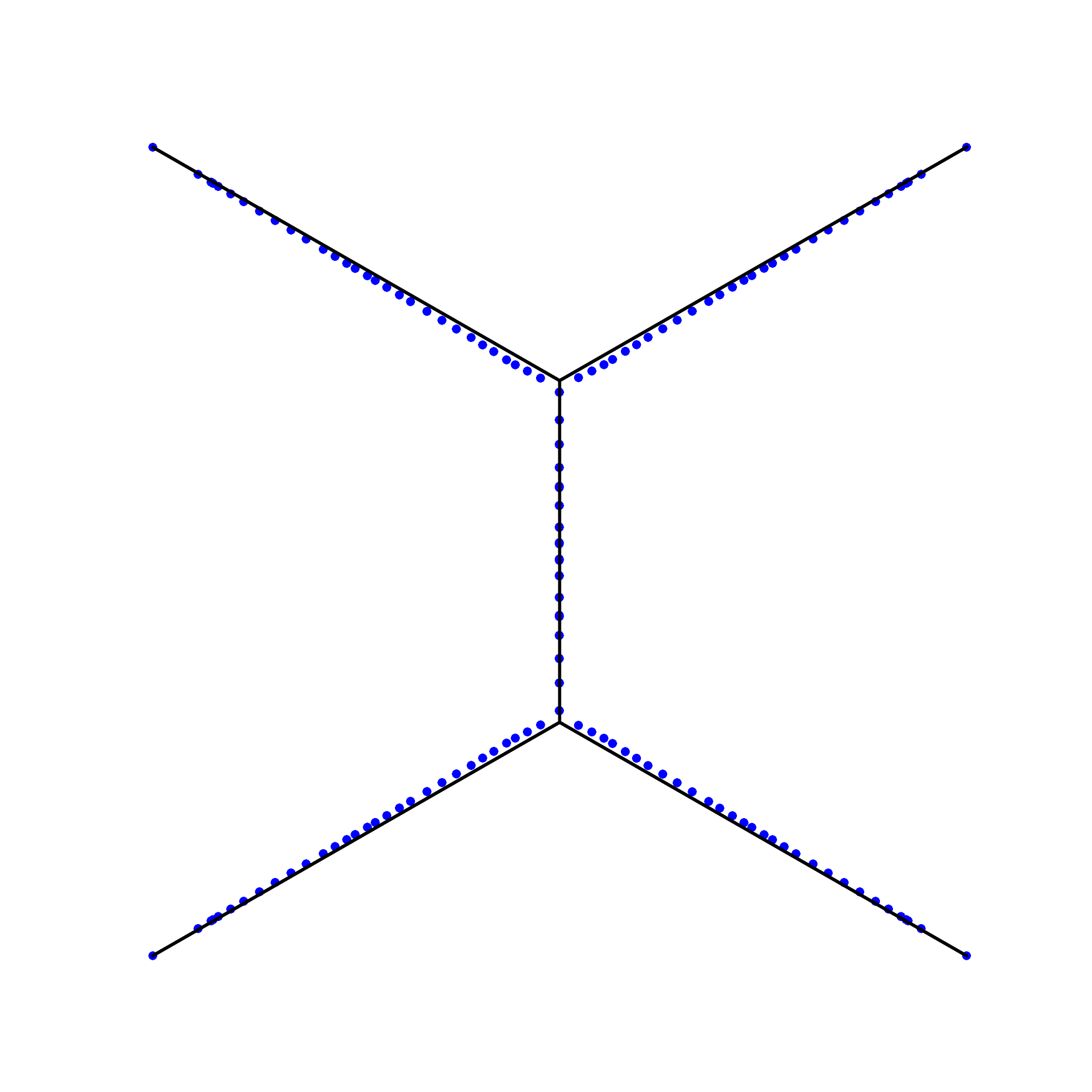}}\\
\subcaptionbox{time $0$}{\includegraphics[width=0.16\textwidth]{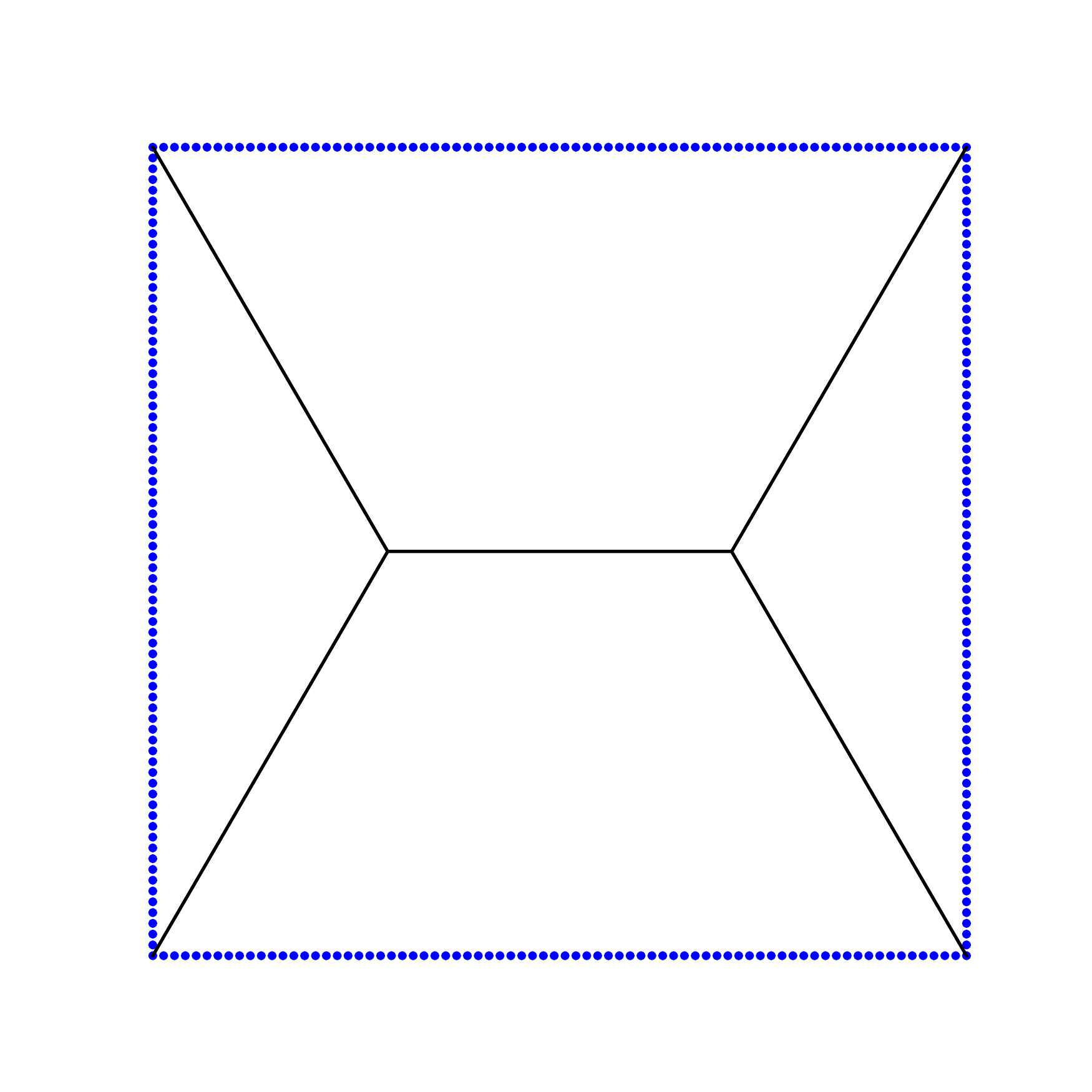}}
\subcaptionbox{time $0.1$}{\includegraphics[width=0.16\textwidth]{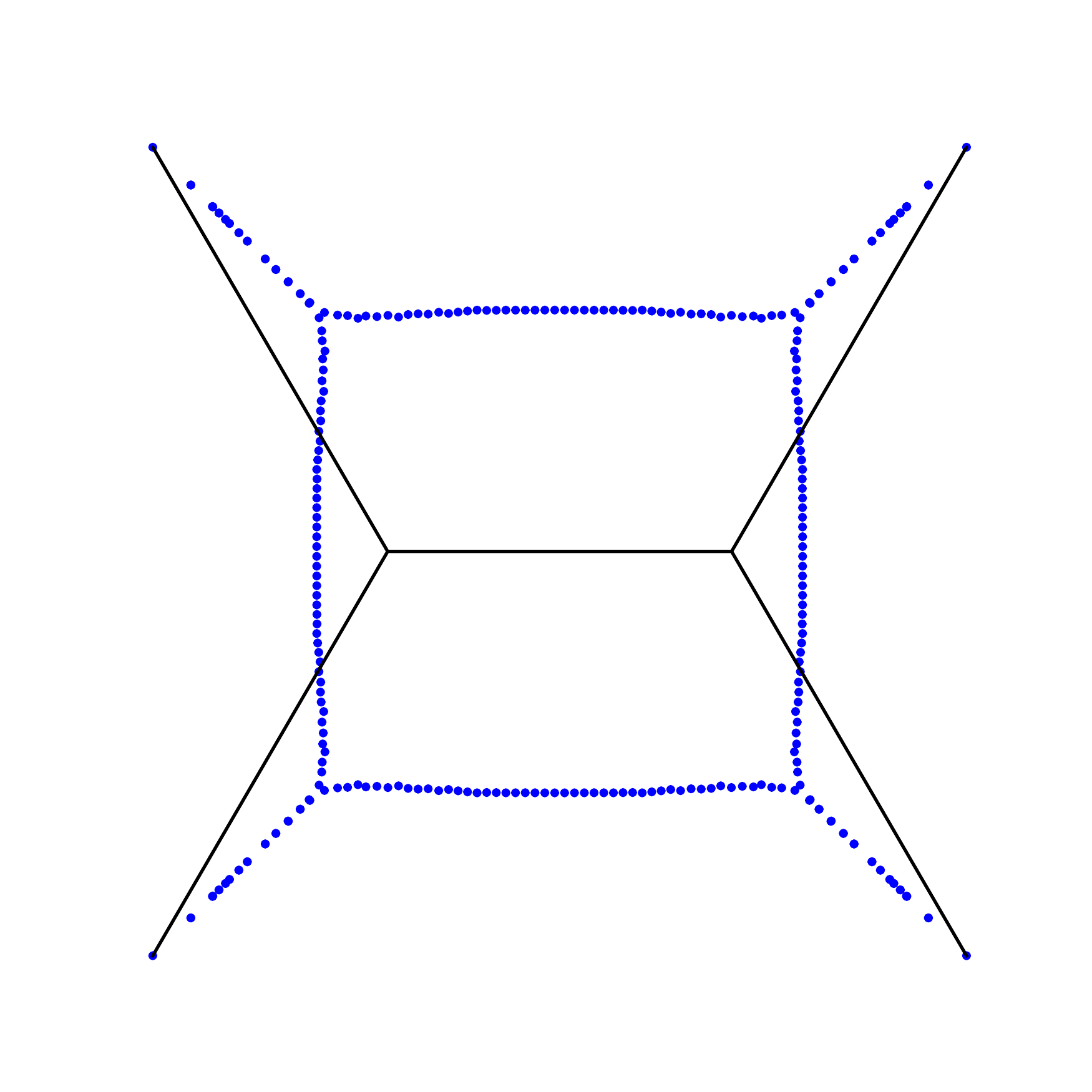}}
\subcaptionbox{time $0.2$}{\includegraphics[width=0.16\textwidth]{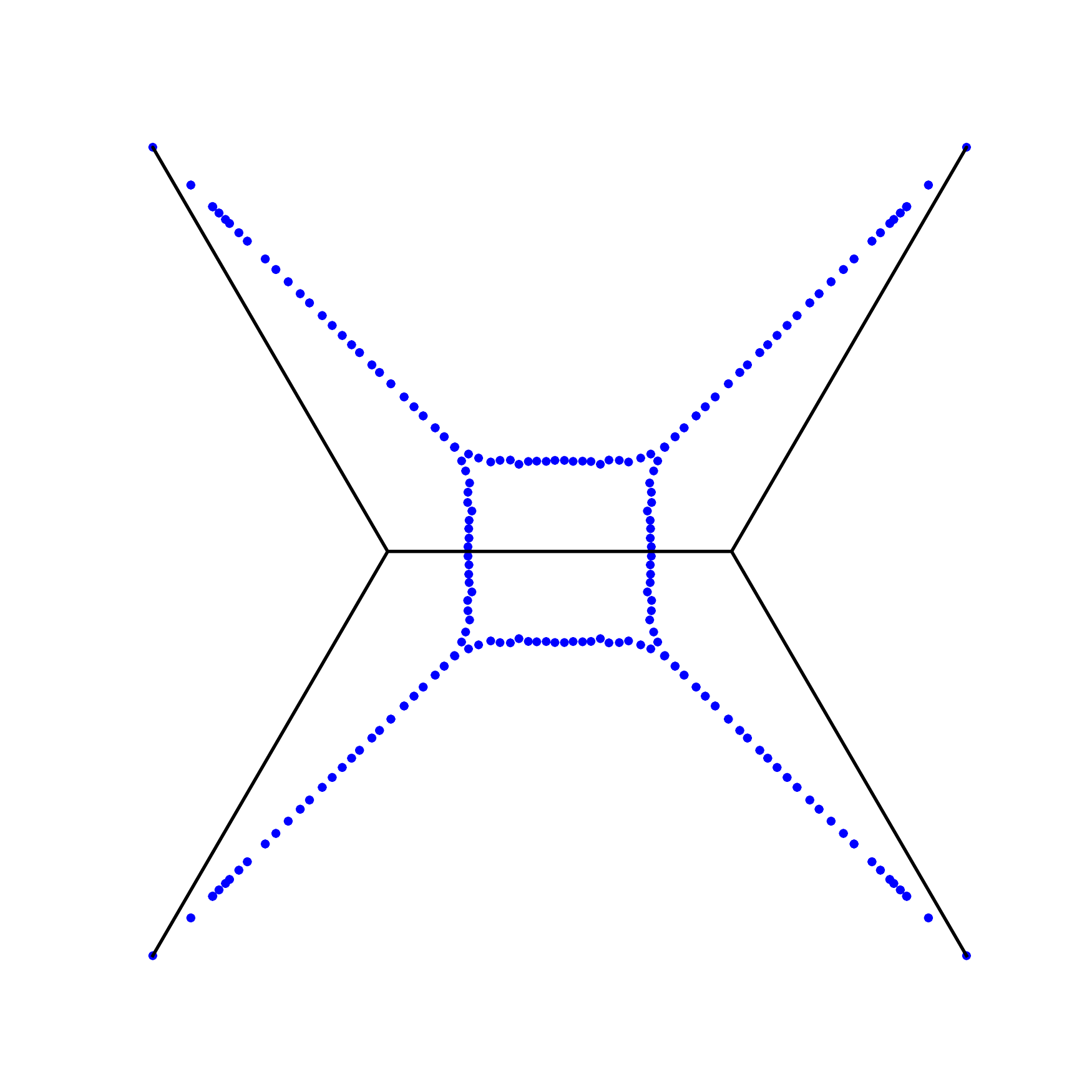}}
\subcaptionbox{time $0.3$}{\includegraphics[width=0.16\textwidth]{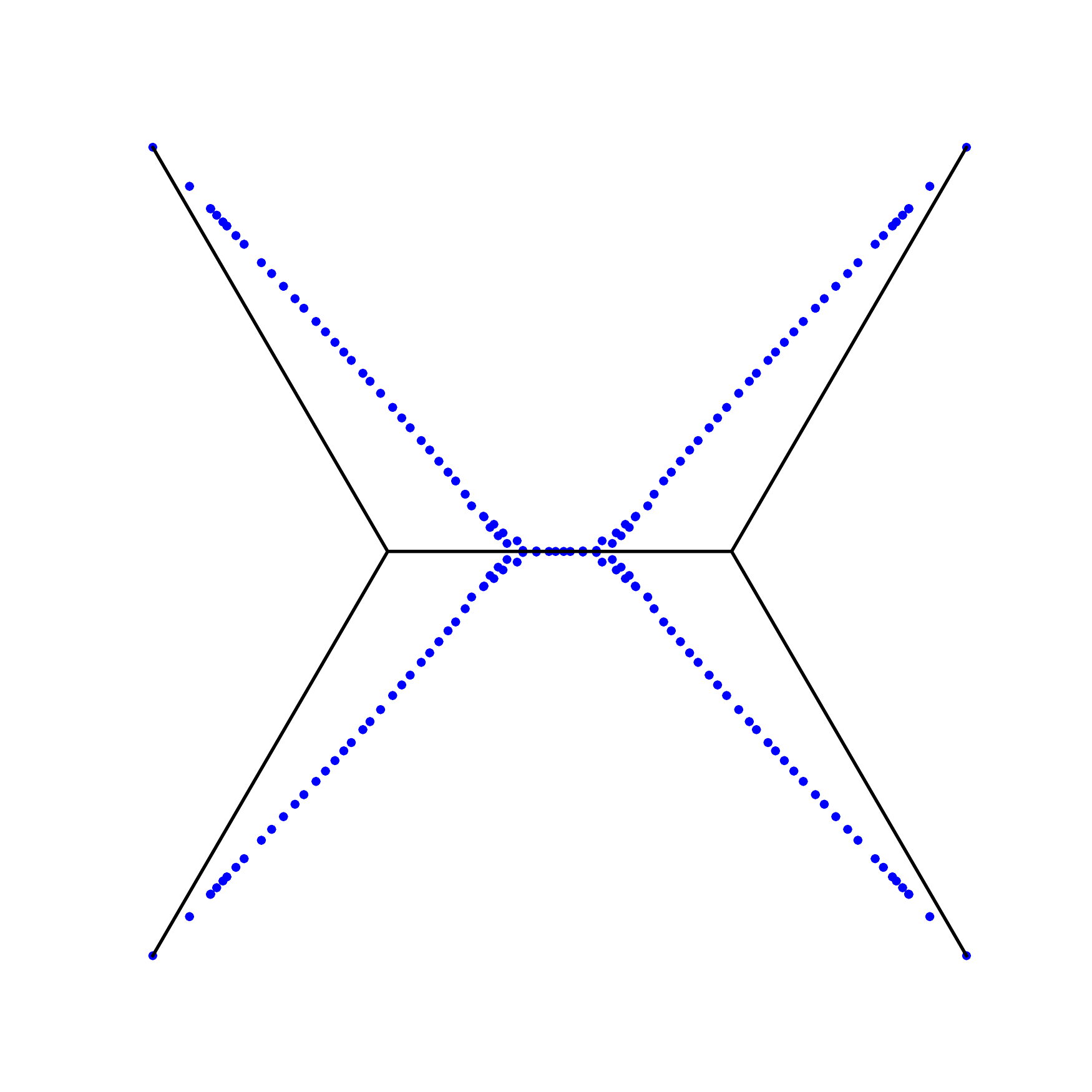}}
\subcaptionbox{time $0.4$}{\includegraphics[width=0.16\textwidth]{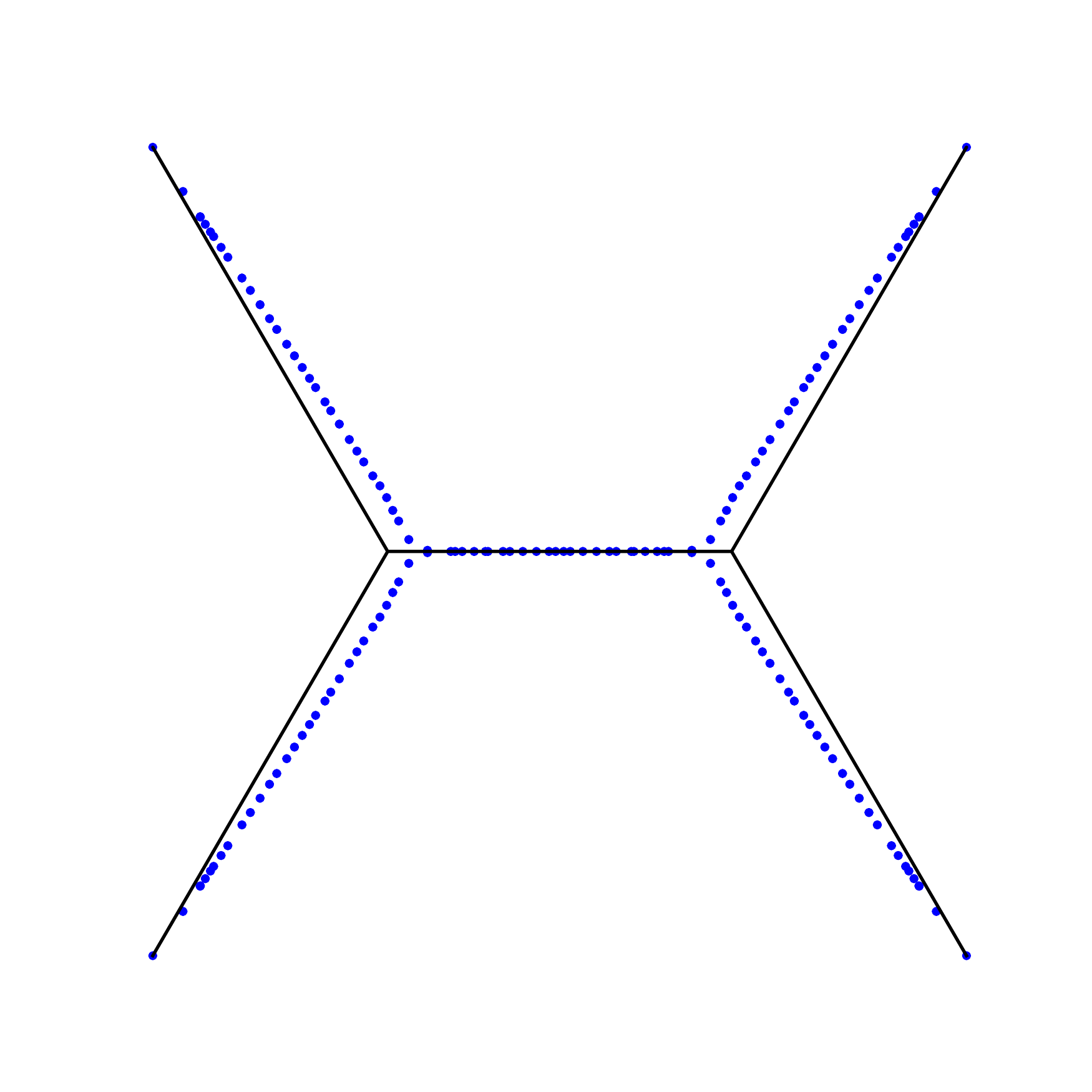}}
\subcaptionbox{time $1$}{\includegraphics[width=0.16\textwidth]{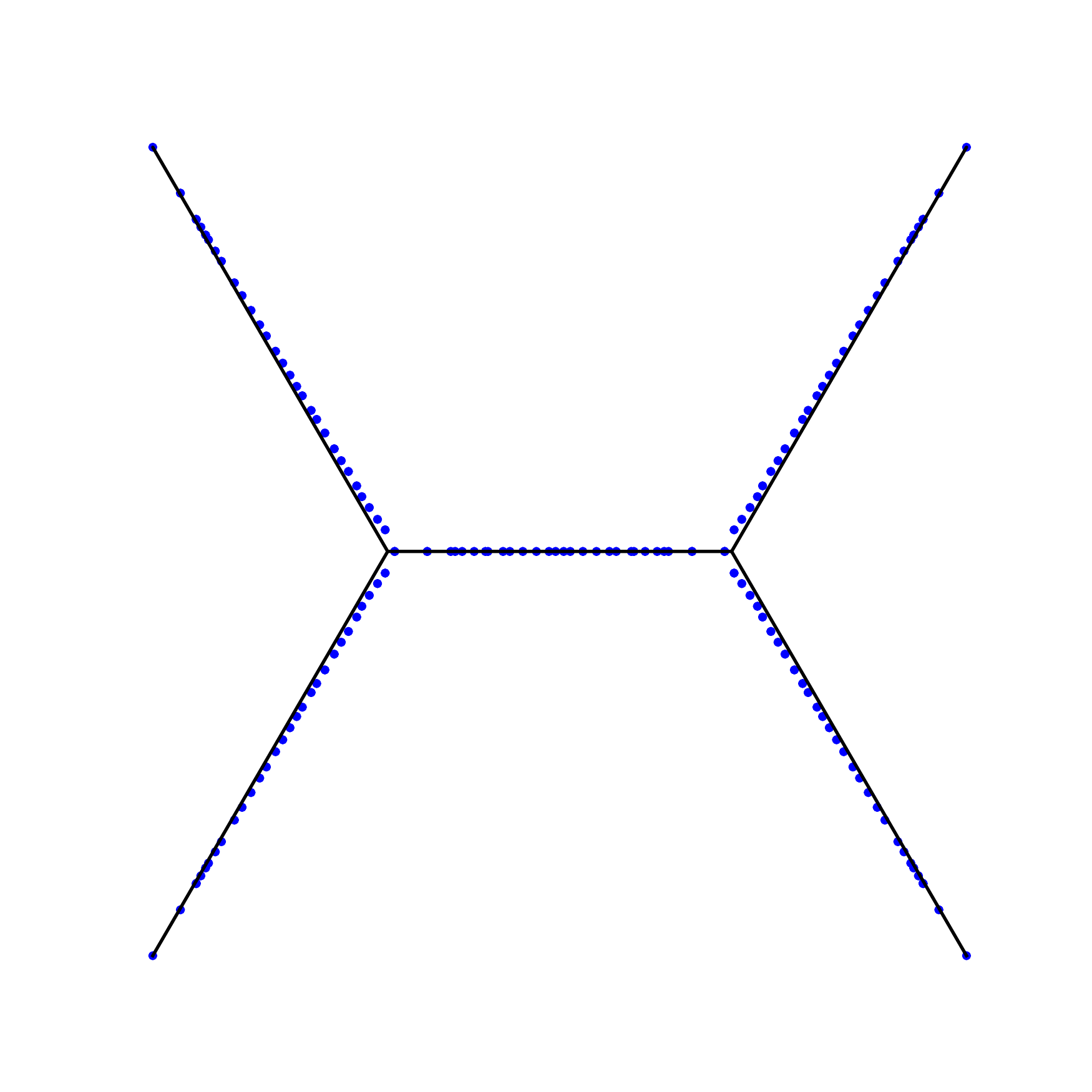}}
\caption{Evolution of a square whose corners are fixed to one $(a)$--$(f)$ or the other $(g)$--$(l)$ Steiner tree connecting the $4$ fixed corners. 
\label{figSteiner1}}
\end{figure}
\end{center}

\subsection{Singular evolutions of surfaces} \label{sec:singularSurf}

We now perform numerical tests in $3D$: on a surface leaning on the edges of tetrahedron and on the edges of a cube, and we evidence that, for some parameters, we recover some well--known soap films in the limit of the mean curvature flow. For the tetrahedron, we obtain a cone on the edges which is one of the three possible minimal cones in $\R^3$ (see \cite{Ta76}). For the cube we obtain a surface in Fig.~\ref{figCube} (i) with close to planar facets connecting the edges of the cube with a square in the center. This reflects not only experimentally observed soap films but also a theoretical result by Brakke  \cite{Br91} which gives a lower area competitor of the cone over the edges sharing the geometry of our numerical result with slightly bended facets.
It is known that mean curvature vector points is the direction to choose in order to decrease area.  Indeed, the mean curvature vector is the $\xL^2$--gradient of the area functional and varifolds generalized mean curvature relies on this characterization. Based on that observation, when the mean curvature flow is well--defined it should yield a minimal surface in the limit. However, when the flow creates singularities it is more complicated to analyse, also on the computational side.  For instance, triangulated surfaces are not well--adapted to handle topology changes. Our discrete flow based on point cloud representation allows to observe the evolution from an initial surface spanning the edges of the tetrahedron or the cube to one of the soap films spanning the same boundary. We insist that the discrete flow is automatic, once the parameters are set, there is no further manual intervention. Note that for $2D$ surfaces, creation of holes is an unwanted but succesfull strategy to decrease area while spanning the same $1D$ boundary (the edges of the tetrahedron here). And in the same line, connecting the edges with many very thin structures rather than a surface gives a lower area since lines have zero area. While in $1D$ we only had to care about wether connectedness was preserved, in $2D$ the topology is richer and the choice of parameters is crucial to discard, when possible, those unwanted behaviours.

We now give the details of both numerical experiments, starting with the tetrahedron.
We begin with a point cloud discretizing the faces of a tetrahedron with $N = 6052$ points. The corners of the tetrahedron $0$, $(0,0,1)$, $(0,1,0)$ and $(0,0,1)$ are fixed as well as the edge--points (not moved at iterations).
The number of points used for computing curvature is $k_\epsilon = 26$, tangent is $k_\sigma = 23$ and mass is $k_\delta = 17$. In order to help topology to be preserved, we do not recompute the $kd$--tree structure at each step, but only every $2$ iterations. 
The time step is fixed to $\tau = 0.005$. The point cloud is  depicted every $12$ steps of the evolution as indicated in the legend of Figure~\ref{figTetra}.

\begin{center}
\setcounter{subfigure}{0}
\begin{figure}[!htp]
\subcaptionbox{Step $1$}{\includegraphics[width=0.16\textwidth]{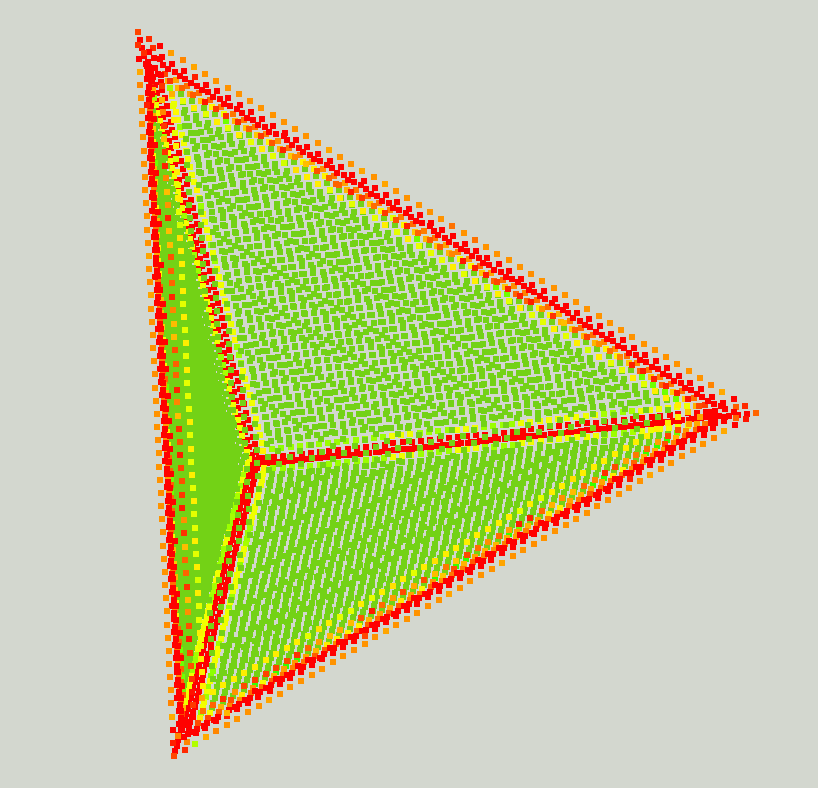} \includegraphics[width=0.16\textwidth]{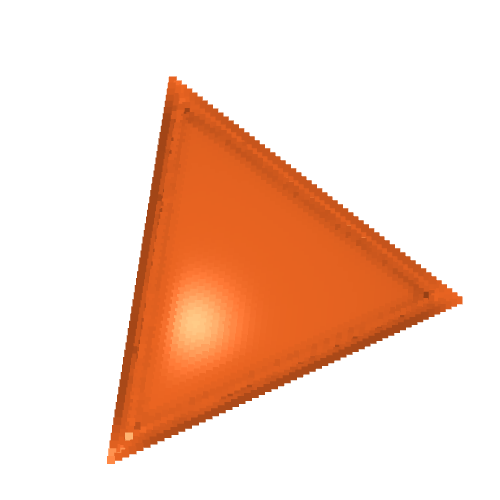}}
\subcaptionbox{Step $13$}{\includegraphics[width=0.16\textwidth]{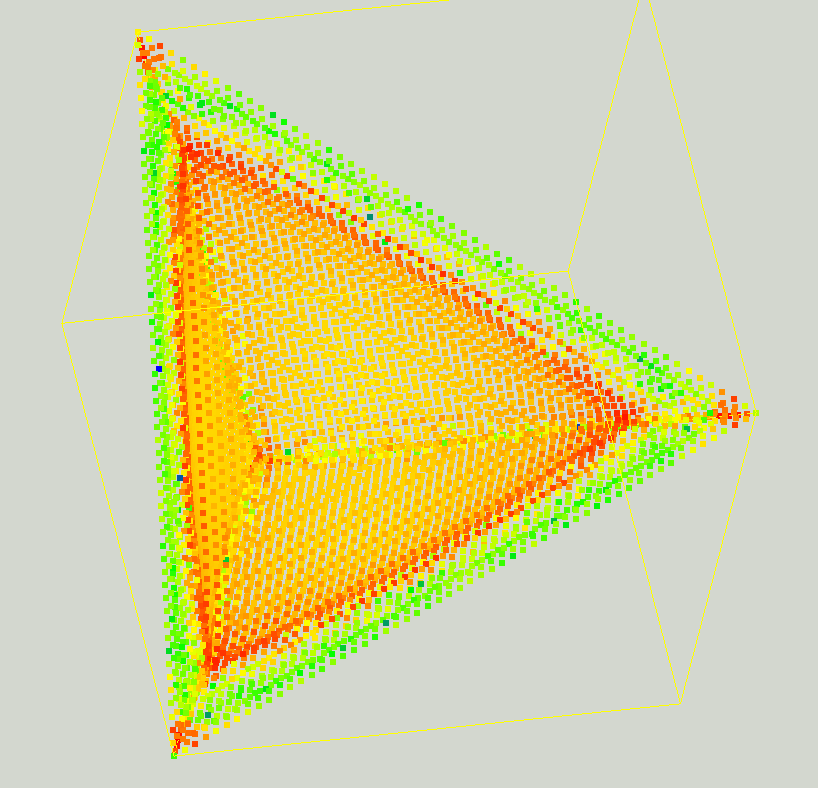} \includegraphics[width=0.16\textwidth]{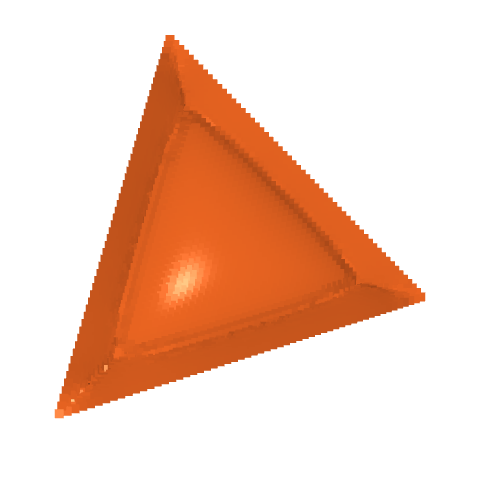}}
\subcaptionbox{Step $25$}{\includegraphics[width=0.16\textwidth]{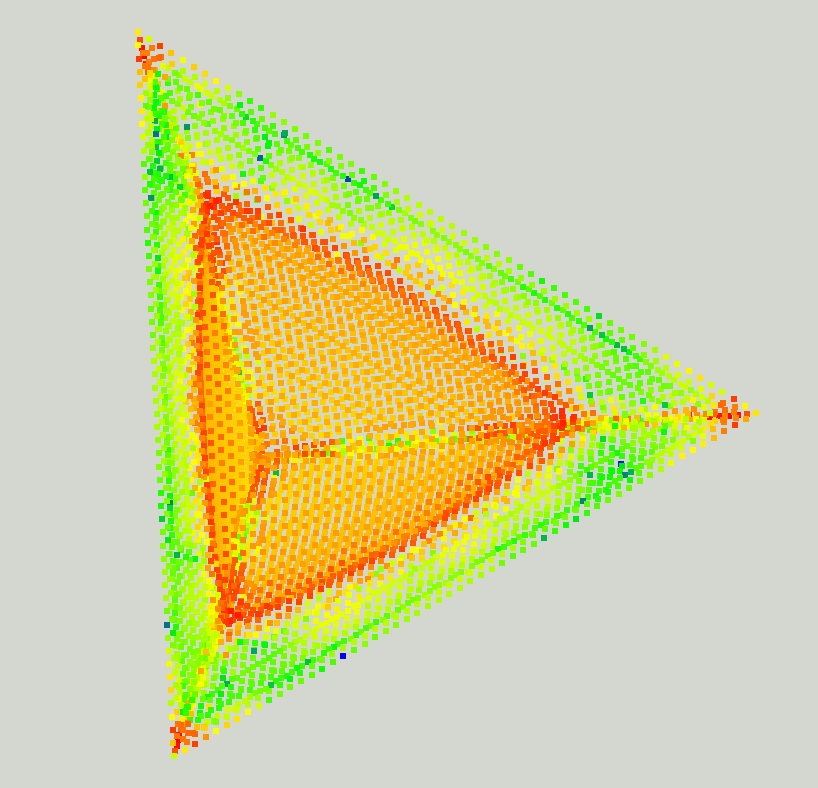} \includegraphics[width=0.16\textwidth]{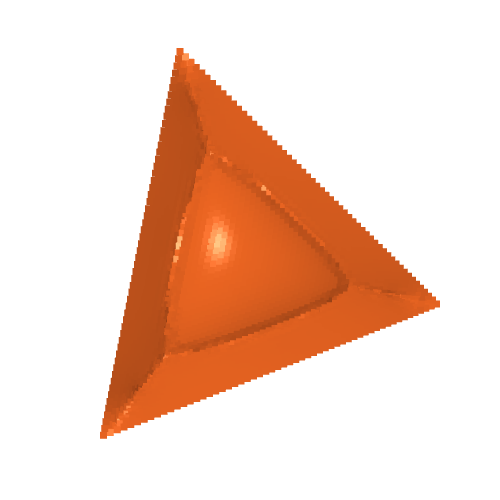}}
\\
\subcaptionbox{Step $37$}{\includegraphics[width=0.16\textwidth]{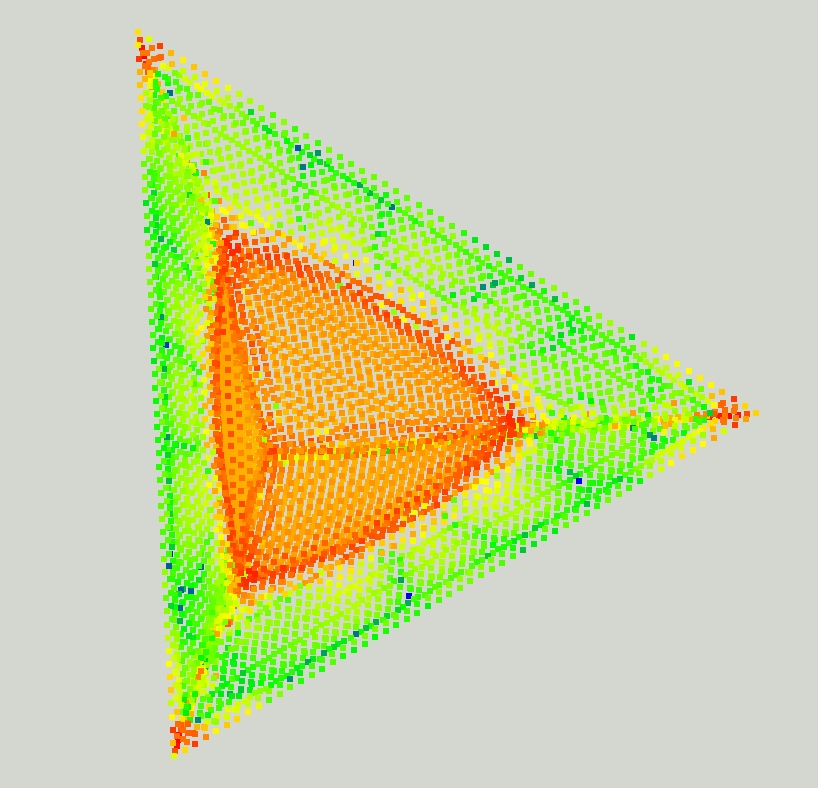} \includegraphics[width=0.16\textwidth]{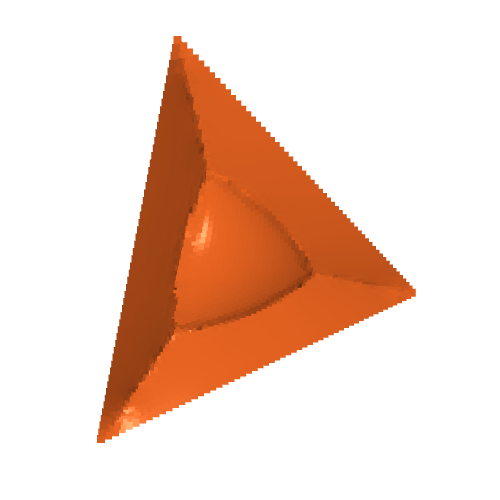}}
\subcaptionbox{Step $49$}{\includegraphics[width=0.16\textwidth]{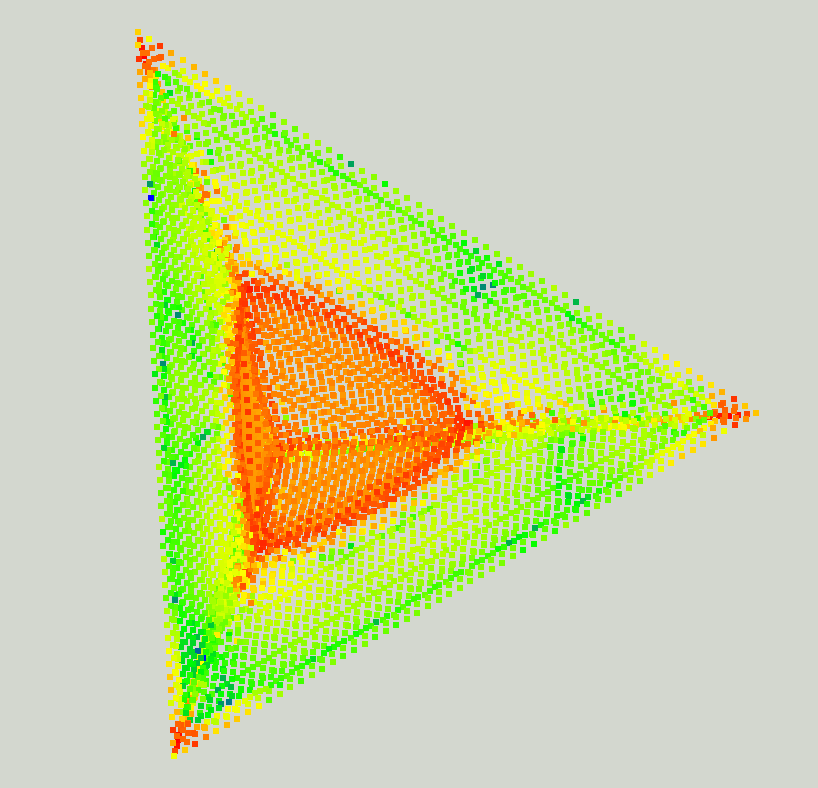} \includegraphics[width=0.16\textwidth]{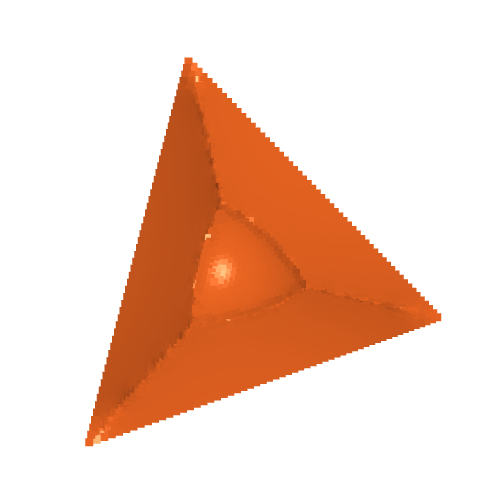}}
\subcaptionbox{Step $61$}{\includegraphics[width=0.16\textwidth]{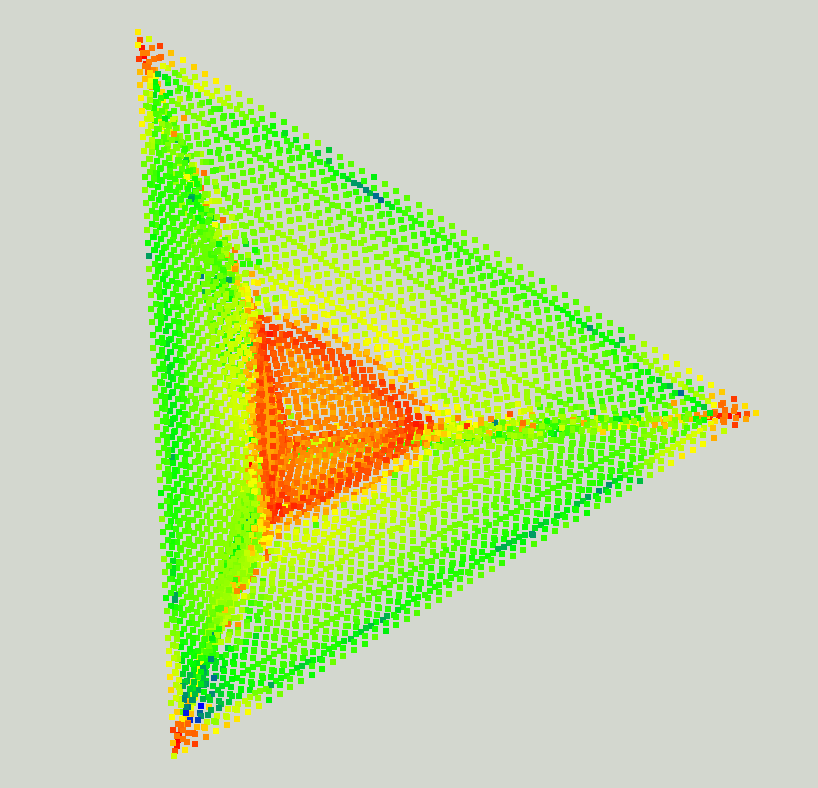} \includegraphics[width=0.16\textwidth]{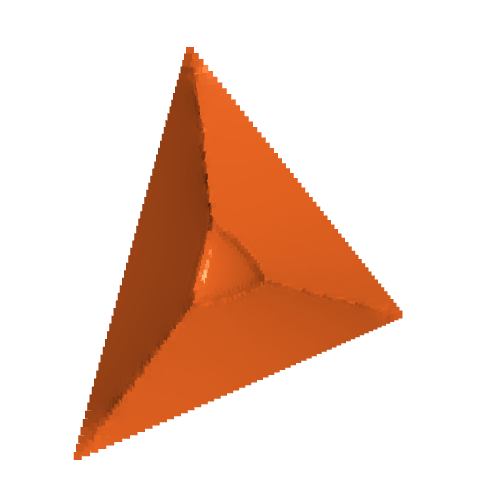}}
\\
\subcaptionbox{Step $73$}{\includegraphics[width=0.16\textwidth]{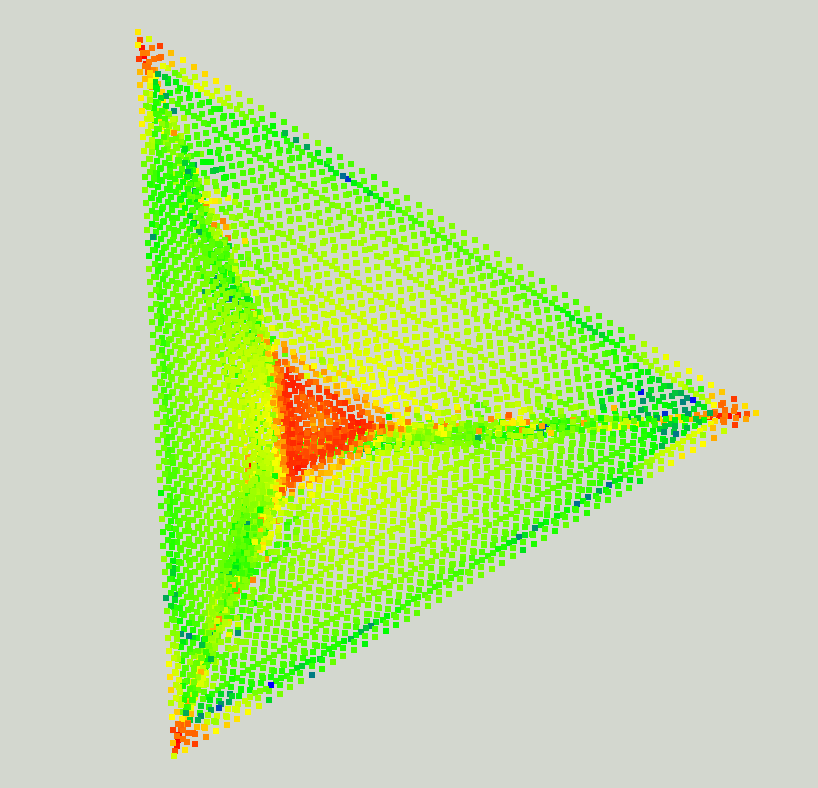} \includegraphics[width=0.16\textwidth]{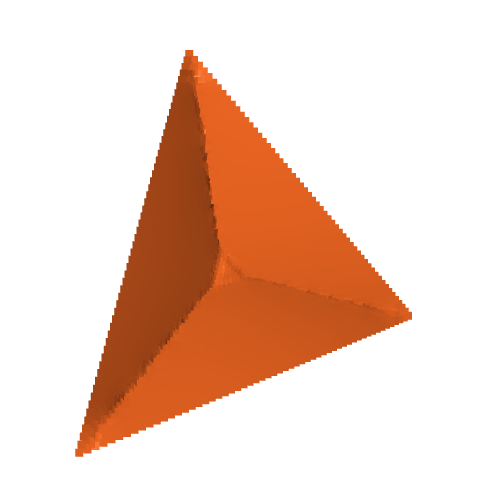}}
\subcaptionbox{Step $85$}{\includegraphics[width=0.16\textwidth]{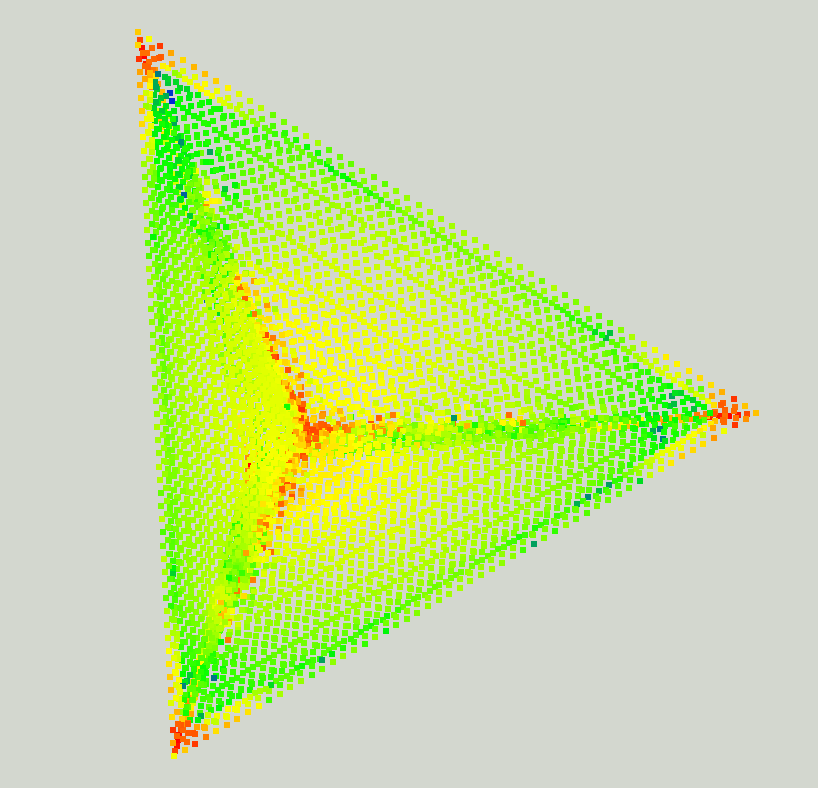} \includegraphics[width=0.16\textwidth]{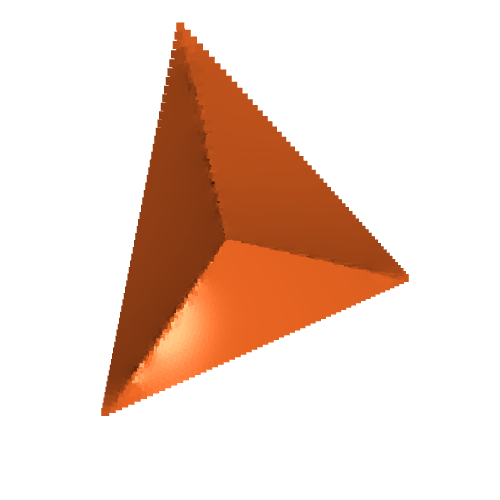}}
\subcaptionbox{Step $97$}{\includegraphics[width=0.16\textwidth]{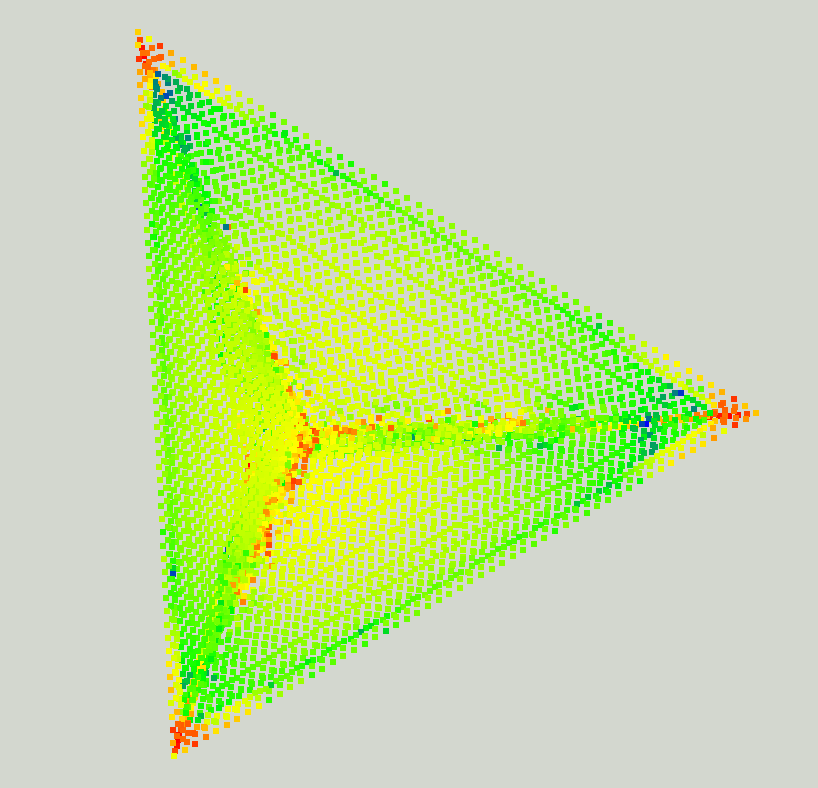} \includegraphics[width=0.16\textwidth]{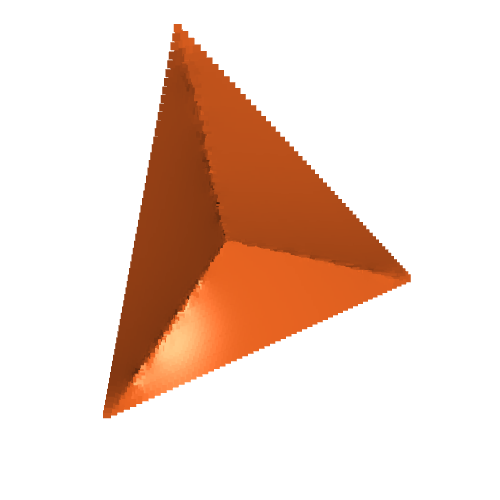}}
\caption{Different steps of the semi-linear scheme \eqref{eqLinearSemiImplicitScheme} performed on (the surface of) a tetrahedron whose edges are fixed, discretized with $N = 6052$ points and for a time--step $\tau = 0.005$.  The norm of the approximate mean curvature vector   is color coded on the left and on the right a shaded visualization of the point clouds using square shaped splats with proper point normals is shown. } \label{figTetra}
\end{figure}
\end{center}

We proceed similarly for the cube.  The faces of a cube with side--length $1$ are discretized with a total of $N = 18600$ points. The number of points used for the computation of curvature is $k_\epsilon = 21$, tangent is $k_\sigma = 23$ and mass is $k_\delta = 9$. 
We do not recompute the $kd$--tree structure at each step, but only every $50$ iterations. The time step is fixed to $\tau = 0.01$. The point cloud is  depicted every $400$ steps of the evolution as indicated in the legend of Figure~\ref{figCube}.

\begin{center}
\setcounter{subfigure}{0}
\begin{figure}[!htp]
\subcaptionbox{Step $1$}{\includegraphics[width=0.16\textwidth]{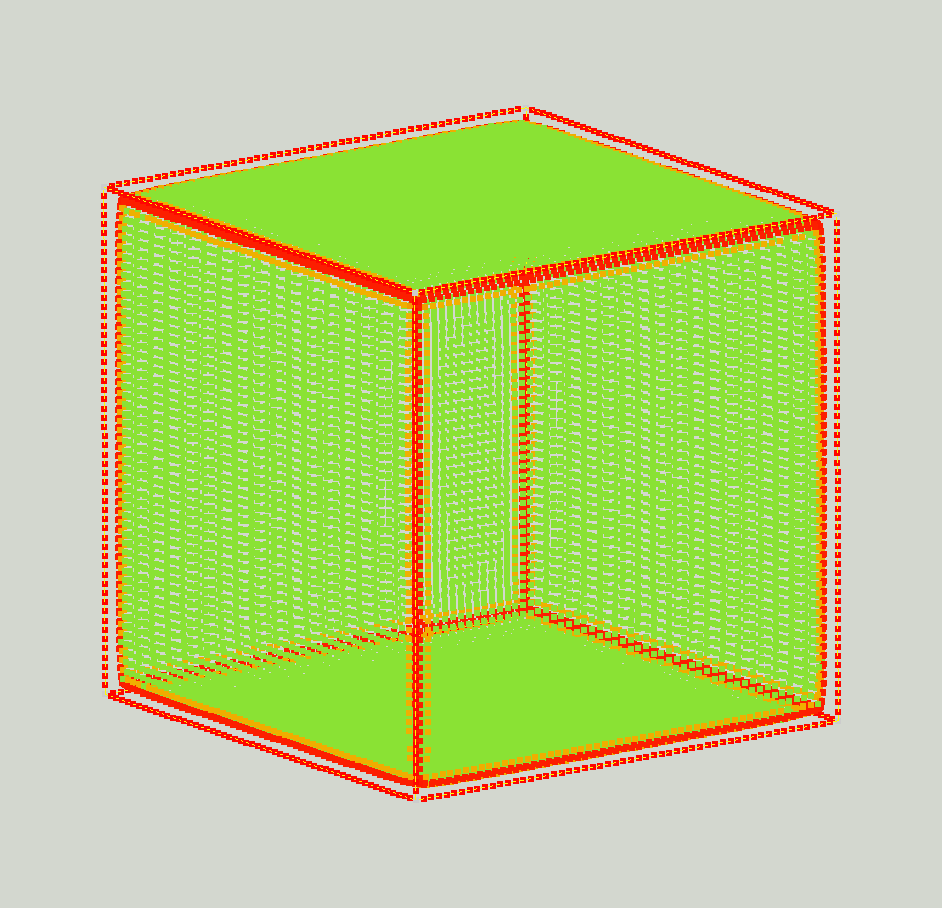} \includegraphics[width=0.16\textwidth]{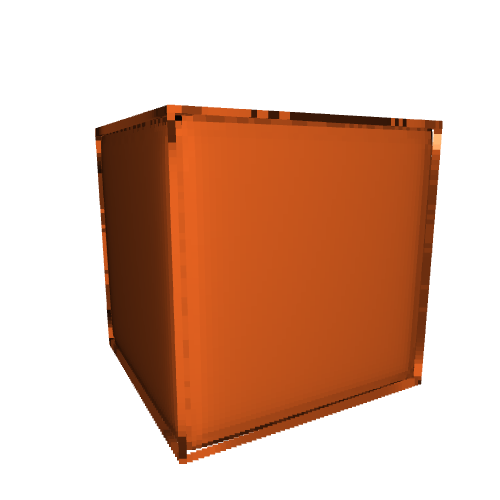}}
\subcaptionbox{Step $401$}{\includegraphics[width=0.16\textwidth]{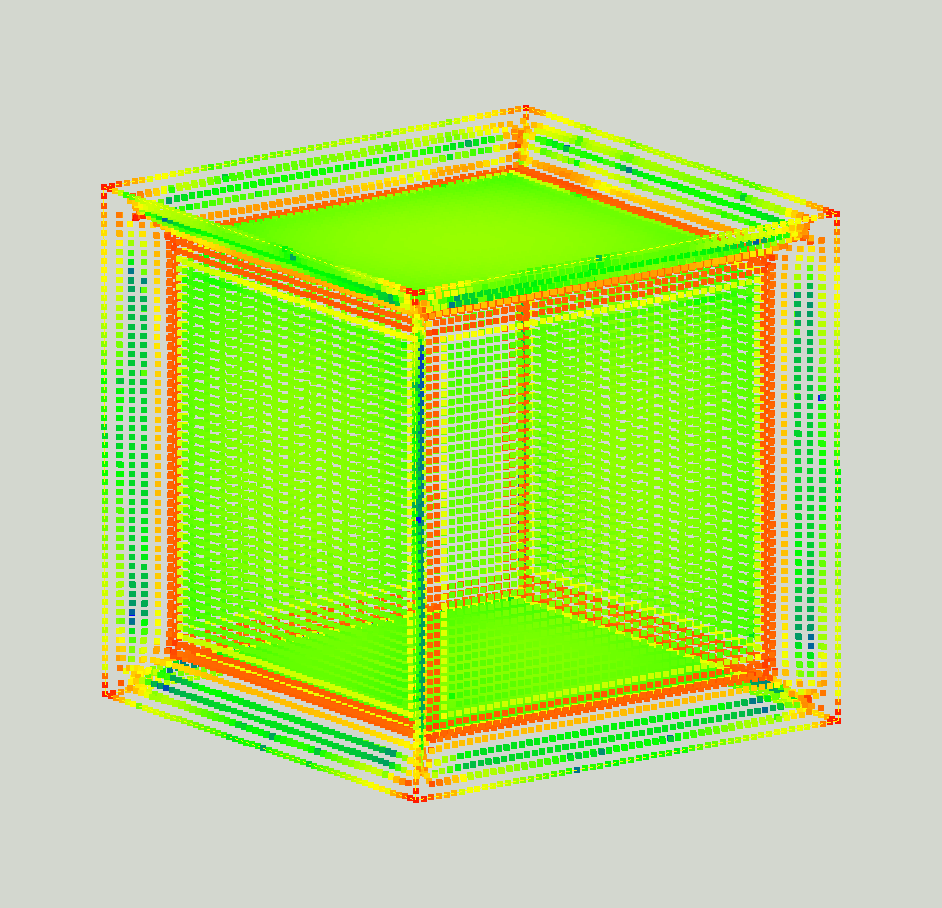} \includegraphics[width=0.16\textwidth]{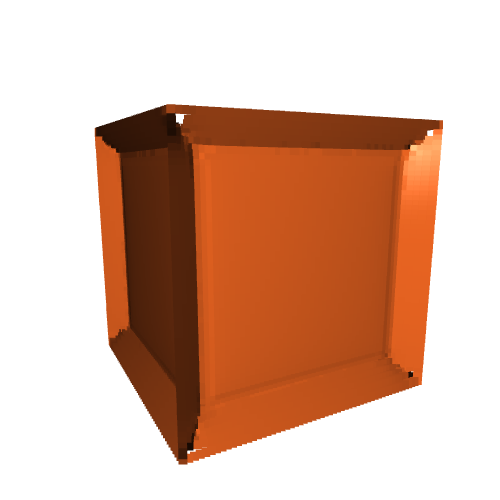}}
\subcaptionbox{Step $801$}{\includegraphics[width=0.16\textwidth]{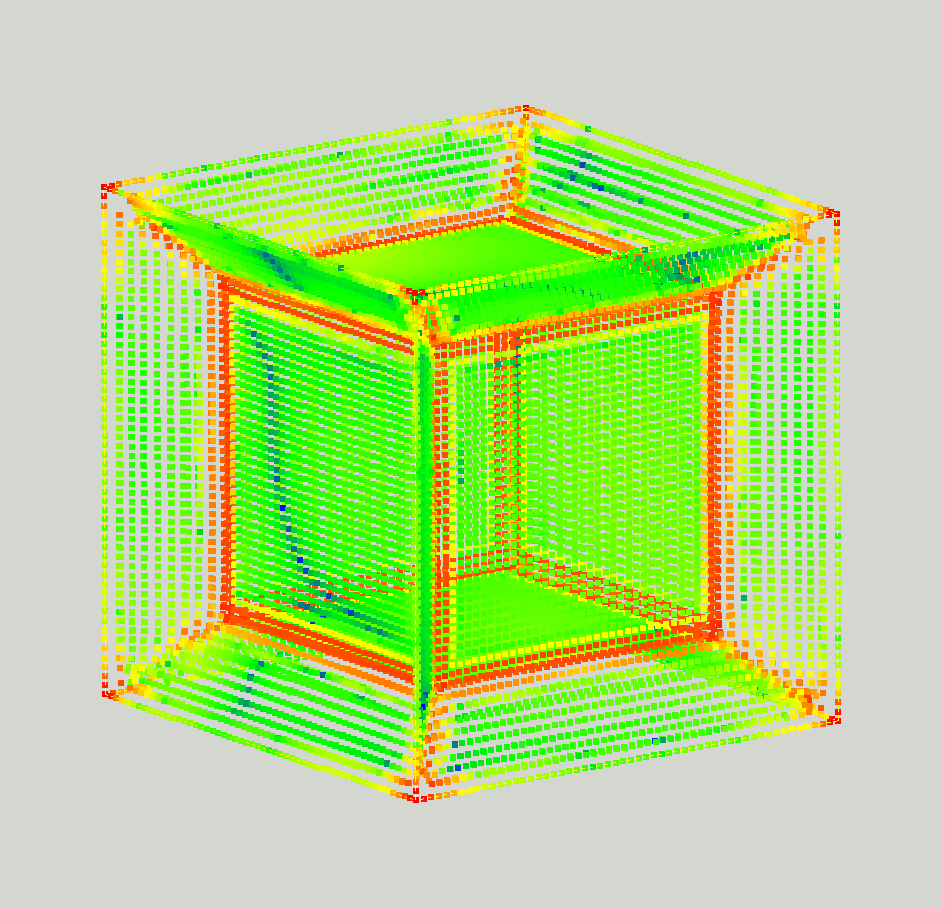} \includegraphics[width=0.16\textwidth]{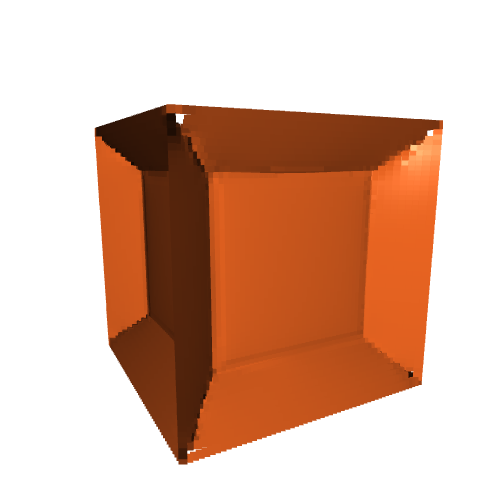}}
\\
\subcaptionbox{Step $1201$}{\includegraphics[width=0.16\textwidth]{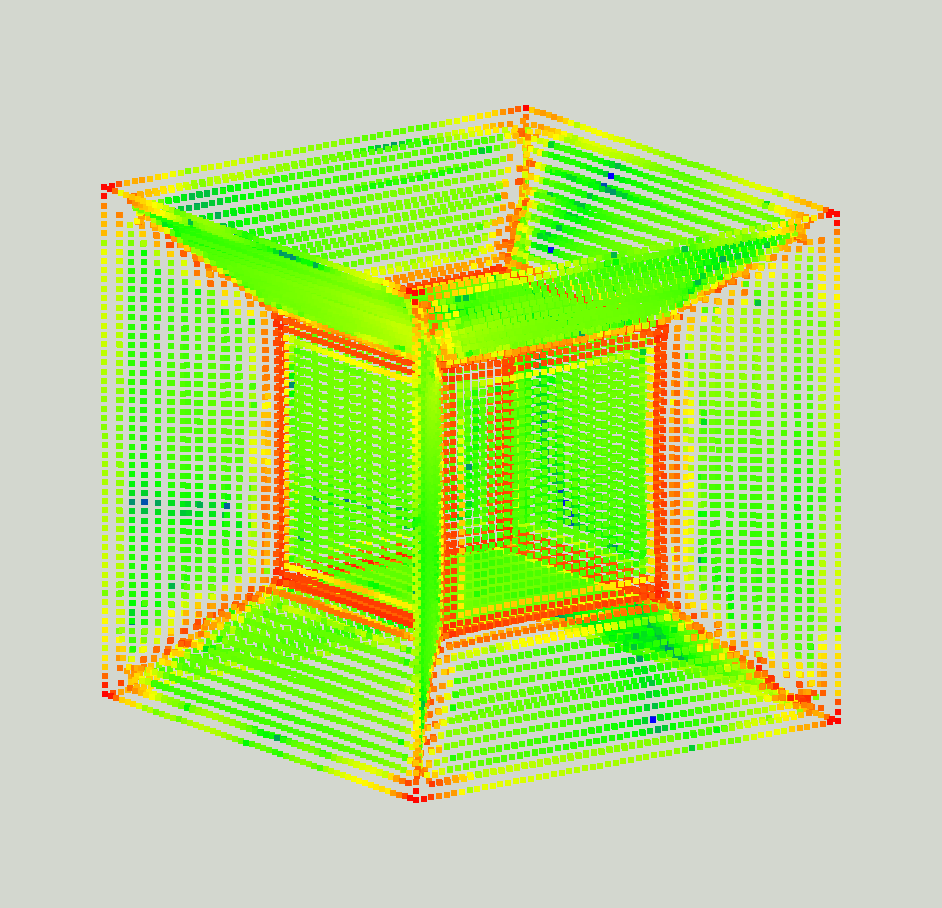} \includegraphics[width=0.16\textwidth]{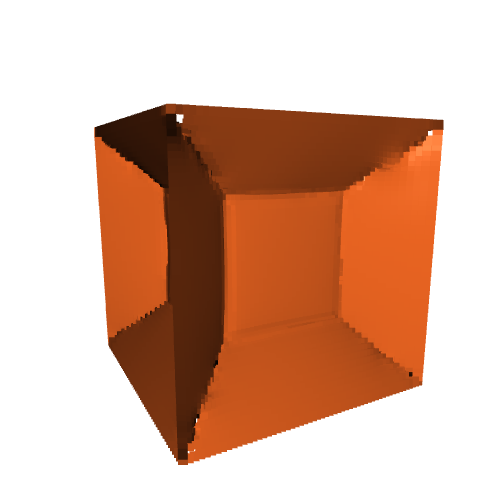}}
\subcaptionbox{Step $1601$}{\includegraphics[width=0.16\textwidth]{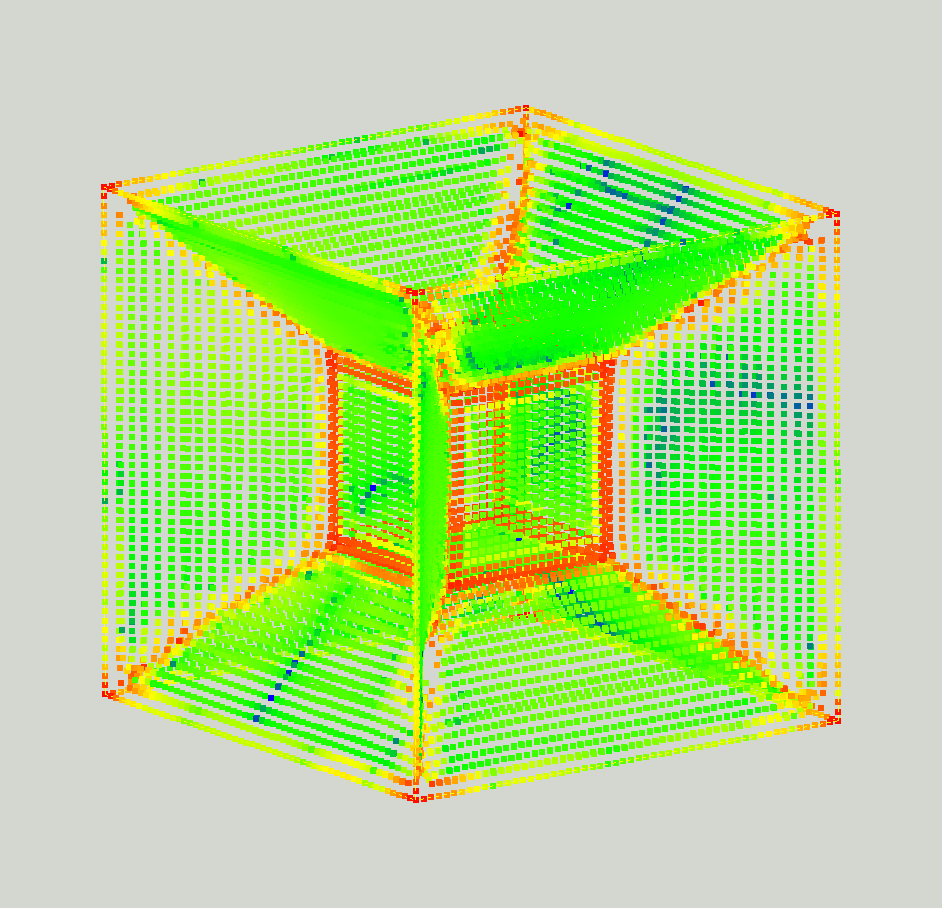} \includegraphics[width=0.16\textwidth]{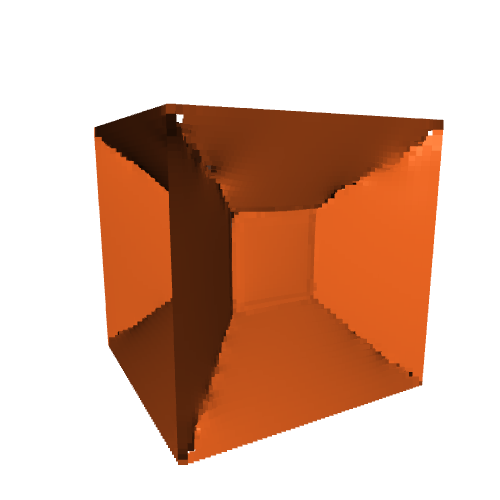}}
\subcaptionbox{Step $2001$}{\includegraphics[width=0.16\textwidth]{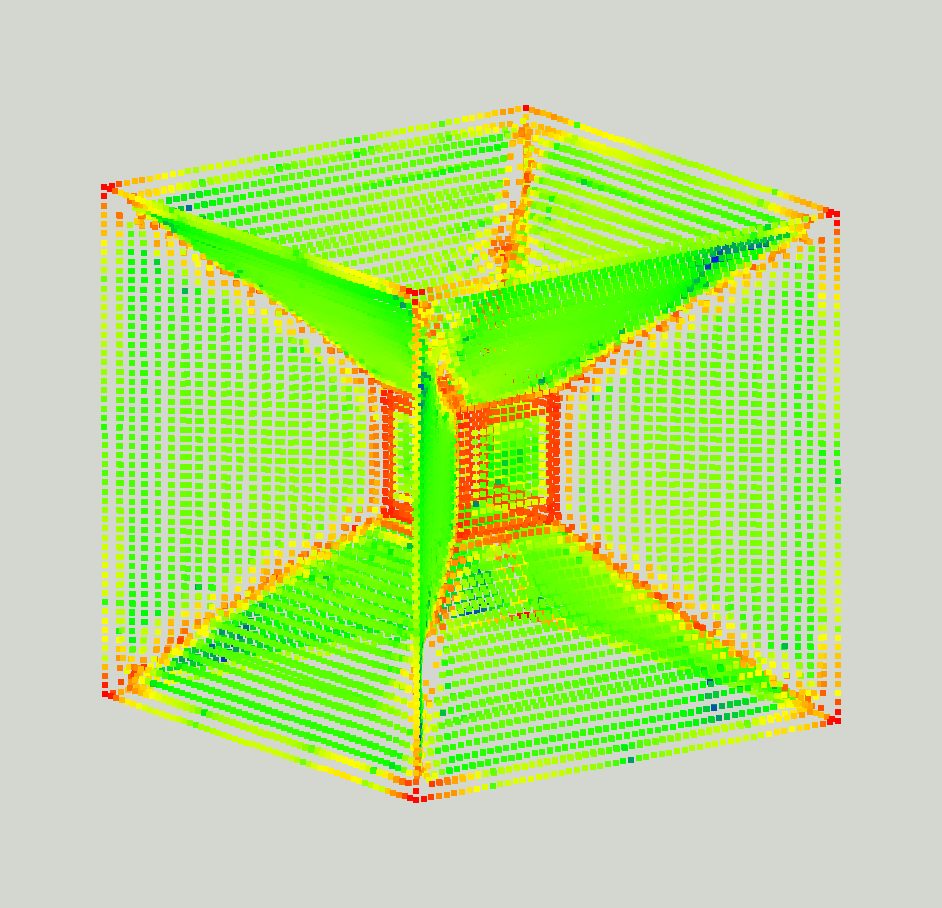} \includegraphics[width=0.16\textwidth]{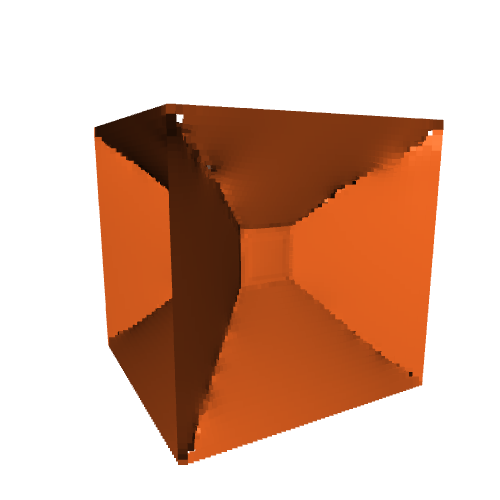}}
\\
\subcaptionbox{Step $2401$}{\includegraphics[width=0.16\textwidth]{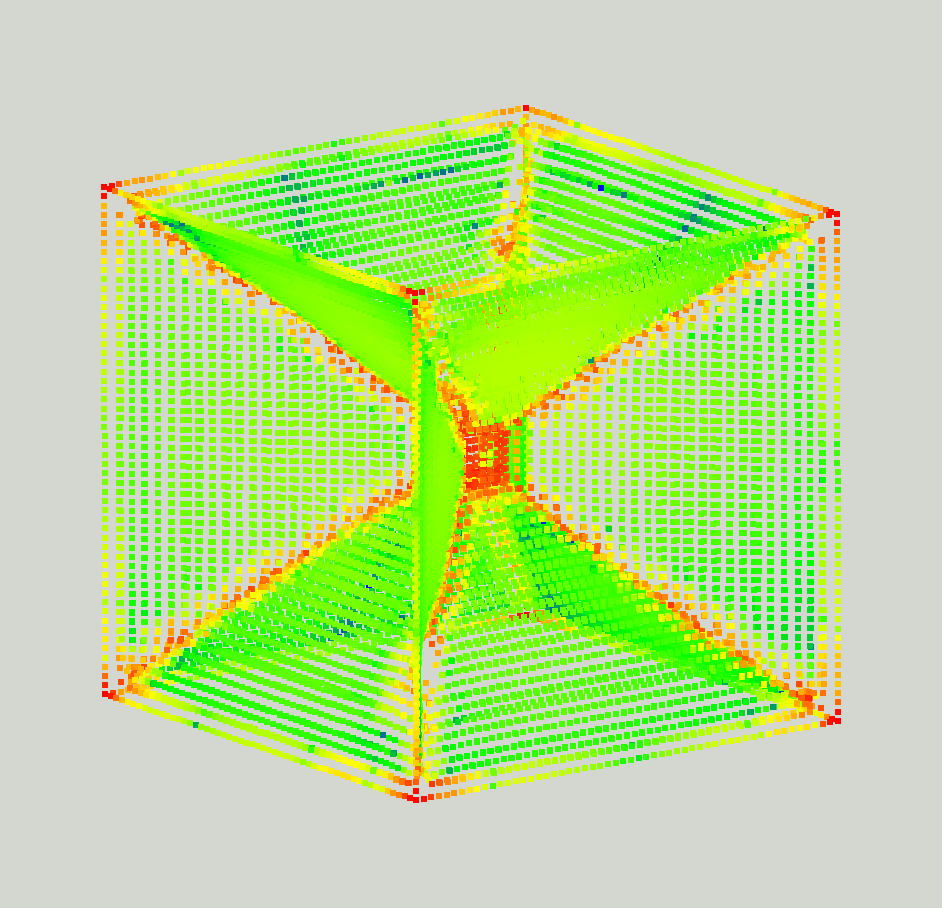} \includegraphics[width=0.16\textwidth]{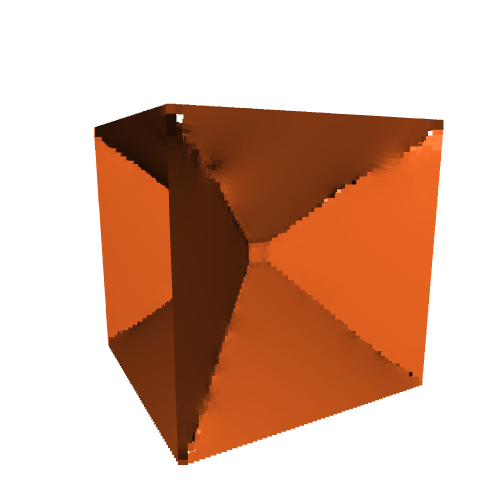}}
\subcaptionbox{Step $2701$}{\includegraphics[width=0.16\textwidth]{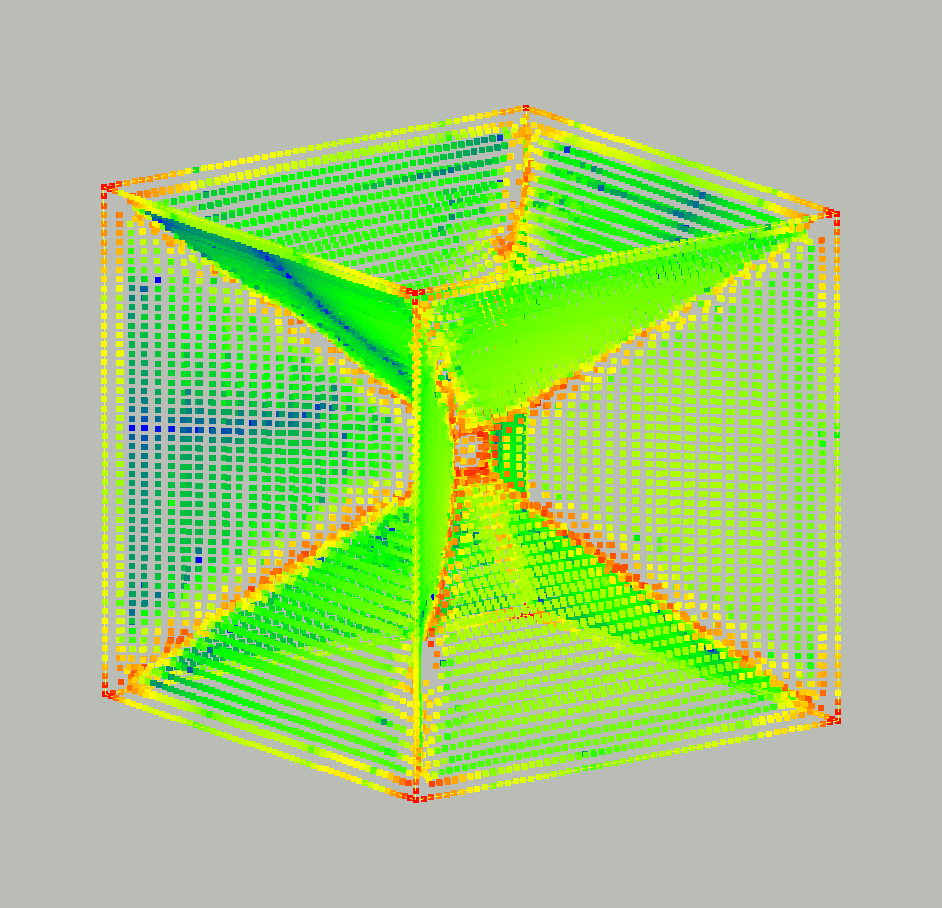} \includegraphics[width=0.16\textwidth]{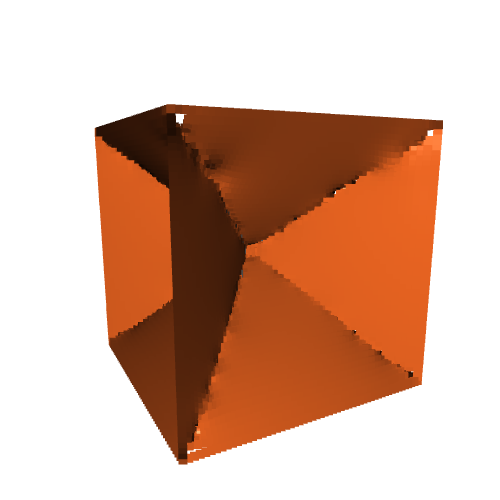}}
\subcaptionbox{Step $2701$ rotated}{\includegraphics[width=0.16\textwidth]{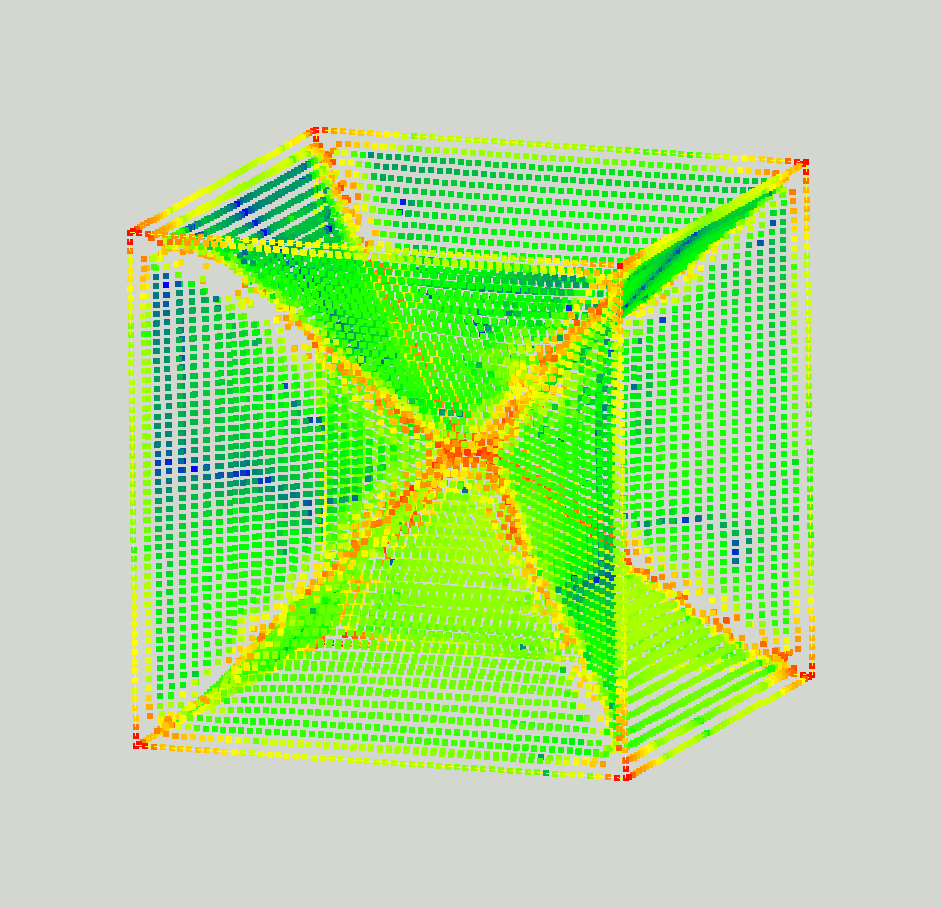} \includegraphics[width=0.16\textwidth]{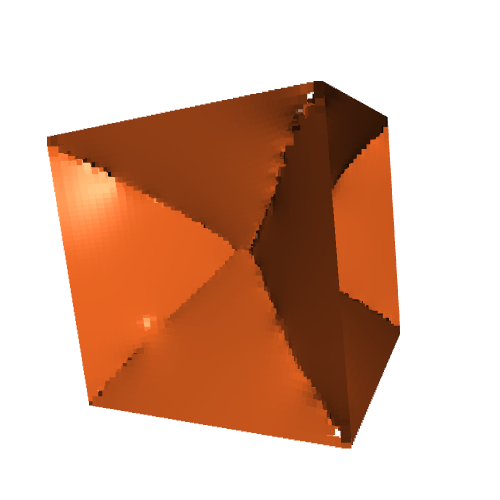}}
\caption{Different steps of the semi-linear scheme \eqref{eqLinearSemiImplicitScheme} performed on (the surface of) a cube whose edges are fixed, discretized with $N = 18600$ points and for a time--step $\tau = 0.01$.  Again, 
the norm of the approximate mean curvature vector  is color coded on the left and on the right a shaded visualization of the point clouds using square shaped splats with proper point normals is shown.} \label{figCube}
\end{figure}
\end{center}

\subsection*{Acknowledgments}
Blanche Buet acknowledges support of CNRS  project JCJC 2018 PEPS of INSMI `A unified framework for surface approximation through varifolds'.
Martin Rumpf acknowledges support of the Collaborative Research Center 1060 fun\-ded by 
the Deutsche Forschungsgemeinschaft (DFG, German Research Foundation)  
and the Hausdorff Center for Mathematics, funded by the DFG
under Germany's Excellence Strategy - GZ 2047/1, Project-ID 390685813.
\bibliographystyle{alpha} 
\bibliography{./notes/biblioMCF.bib}

\end{document}